\documentclass[12pt,a4paper]{amsart}

\usepackage{amsfonts}
\usepackage{amsmath}
\usepackage{amssymb}
\usepackage{amsthm}
\usepackage{mathrsfs}
\usepackage{color}
\usepackage{ulem}

\usepackage{graphicx}
\usepackage[all]{xy}

\usepackage{enumerate}

\usepackage{hyperref}

\newcommand{\f}{\varphi}

\newcommand{\aA}{{\mathcal{A}}}
\newcommand{\bB}{{\mathcal{B}}}
\newcommand{\cC}{{\mathcal{C}}}
\newcommand{\dD}{{\mathcal{D}}}
\newcommand{\eE}{{\mathcal{E}}}
\newcommand{\fF}{{\mathcal{F}}}

\newcommand{\hH}{{\mathcal{H}}}
\newcommand{\mM}{{\mathcal{M}}}

\newcommand{\lL}{{\mathcal{L}}}
\newcommand{\iI}{{\mathcal{I}}}

\newcommand{\kK}{{\mathcal{K}}}

\newcommand{\pP}{{\mathcal{P}}}

\newcommand{\sS}{{\mathcal{S}}}
\newcommand{\tT}{{\mathcal{T}}}
\newcommand{\uU}{{\mathcal{U}}}
\newcommand{\xX}{{\mathcal{X}}}
\newcommand{\yY}{{\mathcal{Y}}}

\newcommand{\la}{\lambda}
\newcommand{\infl}{\iota}
\newcommand{\defl}{\delta}

\newcommand{\tr}{{\mbox{$t$-struc}\-tu\-r}}

\newcommand{\bcdot}{{\mbox{\boldmath{$\cdot$}}}}

\newcommand{\Hom}{\mathop{\textrm{Hom}}\nolimits}
\newcommand{\Ext}{\mathop{\textrm{Ext}}\nolimits}
\newcommand{\End}{\mathop{\textrm{End}}\nolimits}
\newcommand{\Fun}{\mathop{\textrm{Fun}}\nolimits}

\newcommand{\Coh}{\mathop{\textrm{Coh}}\nolimits}
\newcommand{\Lex}{\mathop{\textrm{Lex}}\nolimits}
\newcommand{\Rex}{\mathop{\textrm{Rex}}\nolimits}
\newcommand{\Id}{\mathop{\textrm{Id}}\nolimits}

\newcommand{\Ex}{\mathop{\textrm{Ex}}\nolimits}
\newcommand{\opp}{\mathop{\textrm{op}}\nolimits}
\newcommand{\Ab}{\mathop{\mathcal{A}\textrm{b}}\nolimits}

\newcommand{\wt}[1]{\widetilde{#1}}
\newcommand{\ol}[1]{\overline{#1}}

\newtheorem{LEM}{Lemma}[section]

\newtheorem*{THM*}{Theorem}
\newtheorem*{PROP*}{Proposition}
\newtheorem{THM}[LEM]{Theorem}
\newtheorem{PROP}[LEM]{Proposition}
\newtheorem{COR}[LEM]{Corollary}

\theoremstyle{definition}
\newtheorem{EXM}[LEM]{Example}
\newtheorem{REM}[LEM]{Remark}
\newtheorem{DEF}[LEM]{Definition}

\begin{document}
\title{Abelian envelopes of exact categories and highest weight categories}
\today
	\author{Agnieszka Bodzenta}
\author{Alexey Bondal}

\address{Agnieszka Bodzenta\\
	Institute of Mathematics, University of Warsaw \\ Banacha 2 \\ Warsaw 02-097,
	Poland} \email{A.Bodzenta@mimuw.edu.pl}

\address{Alexey Bondal\\
	Steklov Mathematical Institute of Russian Academy of Sciences, Moscow, Russia, and \newline
	Center of Pure Mathematics, Moscow Institute of Physics and Technology, Russia, and
	\newline
	Kavli Institute for the Physics and Mathematics of the Universe (WPI), The University of Tokyo, Kashiwa, Chiba 277-8583, Japan} 
\email{bondal@mi-ras.ru}

\begin{abstract}
	We define admissible and weakly admissible subcategories in exact categories and
	prove that the former induce semi-orthogonal decompositions on the derived categories. 
	We develop the theory of thin exact categories, an exact-category analogue of triangulated categories generated by exceptional sequences.
	
	The right and left abelian envelopes of exact categories are introduced, an example
	being the category of coherent sheaves on a scheme as the right envelope of the category
	of vector bundles.
	The existence of right (left) abelian envelopes is proven for exact categories with projectively (injectively) generating subcategories with weak (co)kernels.
	
	We show that highest weight categories are precisely the right/left envelopes of thin categories. Ringel duality on highest weight categories is interpreted as a duality between the right and  left abelian envelopes of a thin exact category. 
	A duality for thin exact categories compatible with Ringel duality is introduced by means of derived categories and
	Serre functor on them.
\end{abstract}
	
	\dedicatory{Dedicated to Claus Michael Ringel on his 79\textsuperscript{th} Anniversary}
	
	\maketitle
	
	\tableofcontents
	
\section{Introduction}

This paper is a result of an attempt to understand the categorical foundations of the theory of highest weight categories. Highest weight categories, being the same as the categories of representations for quasi-hereditary algebras, emerged in papers of Cline, Parshall and Scott \cite{CPS, ParSco}  as a result of a formalisation of properties of categories of perverse sheaves on nicely stratified topological spaces. They attracted a lot of attention in Representation Theory because categories of representations of various origin turned out to be highest weight, \textit{cf.} \cite{ParSco, Donkin, Rin, DR, BrunStrop}.

By definition, highest weight categories are abelian categories possessing a set of objects, called {\em standard objects}, which have some good properties, which imply, in particular, that they form a full exceptional sequence \cite{B} in the derived category \cite{CPS}. 

Our point of view in this paper, inspired by the work of Claus Michael Ringel, is that to understand highest weight categories one should shift the attention from abelian to exact categories. Studying the properties of the relevant exact categories lead us to the notion of a \emph{thin} exact category.  

We consider an exact category counterpart of the notion of admissible subcategory in a triangulated category \cite{B}. It happens that there are two levels of admissibility for subcategories in exact categories, which we call {\em weak admissibility} and just {\em admissibility}. Thin category is an exact category that has a filtration by admissible subcategories, such that graded factors are categories of vectors spaces. One can think of it  as an exact category endowed with a suitable replacement of what-would-be a full exceptional sequence, but for exact categories. This gadget has already been discussed by V. Dlab and C. Ringel in the abelian category set-up under the name of {\em standarisable collection} \cite{DR}. 

To match the definition in the triangulated world, we show that  the derived category of an admissible subcategory in an exact category is an admissible subcategory (in the `triangulated sense') of the derived category of the ambient category.   

The main innovative tool of this paper are abelian envelopes of exact categories. We distinguish the right $\aA_r(\eE)$ and left $\aA_l(\eE)$  envelopes of an exact category $\eE$. One known way to associate an abelian category to an exact category is via the {\em exact abelian hull}. It is the universal abelian category among those which allow an exact functor from $\eE$ into them. The \emph{left and right envelopes} are similar universal abelian categories but for functors which are only right/left exact. The existence of the hull was proven in \cite{Adel, Stein}. An arbitrary exact category does not need to have the right abelian envelope (see Example \ref{exm_no_env}). We prove the existence of the right/left envelope under the condition that the exact category has enough projectives/injectives and the full subcategory of projectives/injectives  has weak kernels/cokernels. We discuss further instances when the envelopes exist. For example, we show that the
right envelope of the category of vector bundles on a good enough scheme is the category of coherent sheaves.
Under suitable finiteness conditions we prove the derived
equivalence of an exact category with its envelopes.

We show that highest weight categories are precisely the  right/left envelopes of thin exact categories. The thin category related to the highest weight category is the extension closure of the set of (co)standard objects in it.
The envelopes of thin categories possess canonical strict (co)localising filtrations which are relevant to the partial order on simple objects in the highest weight categories. We give a characterisation of the thin subcategory inside its envelope via adjunction morphisms induced by the filtration. 

Ringel duality (originally for quasi-hereditary algebras) acquires in our context a natural interpretation as a transfer from the right to the left envelope and vice versa. Also it leads to a duality (which we call also Ringel duality) on thin categories, where the dual of thin category $\eE$ is constructed as the intersection inside derived category $\dD^b(\eE)$ of $\aA_l(\eE)$ with 
the image of $\aA_r(\eE)$ under the inverse of the Serre functor.

\vspace{0.3cm}
\textbf{Structure of the paper and its results in detail.}

We begin by recalling the notion of an exact category (see Section \ref{ssec_prelim}). 
The main (and the only) source of examples of exact categories are full subcategories $\eE\subset \aA$ of abelian categories closed under extensions. Then short exact sequences in $\aA$ whose all terms lie in $\eE$ give an exact structure on $\eE$.

We study the structure of an exact category $\eE$ by means of (perpendicular) torsion pairs. 
Full subcategories $\tT, \fF\subset \eE$ form a \emph{torsion pair} if $\Hom(\tT, \fF) =0$ and any object $E\in \eE$ fits into a conflation $T\to E \to F$ with $T\in \tT$ and $F\in \fF$. Such conflation is functorial in $E$, see Section \ref{ssec_weak_l_and_r_adm}. The data of a torsion pair $(\tT, \fF)$ is equivalent to a \emph{weakly right admissible} subcategory $\infl_* \colon \tT \to \eE$ (i.e. $\infl_*$ admits the right adjoint $\infl^!$ and the adjunction counit $\infl_*\infl^!E \to E$ is an inflation, for any $E\in \eE$), similarly for a \emph{weakly left admissible} subcategory $\defl_*\colon \fF\to \eE$, see Proposition \ref{prop_torsion=adjoint}. 

Torsion pair $(\tT, \fF)$ is  \emph{perpendicular} if further $\Ext^1(\tT, \fF) =0$. We define a \emph{right (left) admissible subcategory} as a torsion class $\tT$ (torsion-free class $\fF$) in a perpendicular torsion pair.
By Theorem \ref{thm_admissible_iff_adjoint_exact}, a torsion pair $(\tT, \fF)$ is perpendicular if and only if one of functors $\infl^! \colon \eE\to \tT$, $\defl^* \colon \eE\to \fF$ is exact, hence both. 
 
 For a perpendicular torsion pair  $(\tT,\fF)$, we introduce the subcategory of \emph{$\tT$-projective objects}. We prove that it is a square zero extension of $\fF$ in the case when $\tT\simeq k \textrm{--vect}$, see Theorem \ref{thm_Y_to_F_is_square_zero}.

The bounded derived category $\dD^b(\eE)$ is well-behaved for \emph{weakly idempotent split} exact categories
 (see Section \ref{ssec_weak_idm_split_and_tor_p}).
An exact category $\eE$ with a perpendicular torsion pair $(\tT, \fF)$ is weakly idempotent split if and only if $\tT$ and $\fF$ are (see Proposition \ref{prop_idem_split_iff_T_and_F}). 
If this is the case, a perpendicular torsion pair $(\tT, \fF)$ on $\eE$ yields a semi-orthogonal decomposition $\dD^b(\eE) = \langle \dD^b(\fF), \dD^b(\tT) \rangle$, see Theorem \ref{thm_SOD}. 
The change of order of $\tT$ and $\fF$ appears here just to match the traditional conventions. 

We define a \emph{strict admissible left/right filtration} on an exact category $\eE$ (see Section \ref{ssec_str_adm_filt}). It is a counterpart in exact categories of a longer semi-orthogonal decomposition and it induces such on $\dD^b(\eE)$,  see Theorem \ref{thm_srict_filt_on_derived}.

In Section \ref{sec_thin_cat} we introduce a class of exact categories which is our main attraction throughout the paper.
A \emph{thin exact category} is a  $\Hom$ and $\Ext^1$-finite $k$-linear exact category which admits a right admissible filtration $0= \tT_0 \subset \tT_1 \subset \ldots \subset \tT_n = \eE$ with graded factors 
equivalent to $k\textrm{--vect}$. 
The derived category of a thin exact category admits a full exceptional sequence \cite{B}, see Proposition \ref{prop_full_exc_coll_in_D_thin}. Hence, $\dD^b(\eE)$ admits the \emph{Serre functor} $\mathbb{S}$ \cite{BK1}.

Using the theory of universal extensions developed in Section \ref{ssec_univ_ext_as_square_zero}, we prove that any thin exact category admits a projectively generating subcategory $\pP$, see Proposition \ref{prop_proj_in_thin}. This is a categorical version of the construction of tilting objects for standarizable collections by V. Dlab and C. Ringel \cite{DR}.

We endow $\Lambda$, the set of irreducible objects in a thin category $\eE$,  with a canonical partial order, see Section \ref{ssec_can_str_filt_on_thin}. Principal lower ideals in this poset allow us to assign a right admissible subcategory $\tT_\lambda \subset \eE$ to any $\la \in \Lambda$, see Proposition \ref{prop_strict_filtr_on_thin}.

In Section \ref{sec_abel_env} we introduce the main tool of our paper, the machinery of \emph{abelian envelopes}. The \emph{right abelian envelope} for an exact category $\eE$ is an abelian category $\aA_r(\eE)$ together with a right exact functor $i_R \colon \eE\to \aA_r(\eE)$ which is  universal for  right exact functors $\eE\to \aA$, with $\aA$ abelian. The \emph{left abelian envelope} $\aA_l(\eE)$ is defined analogously. Since $\aA_l(\eE^{\opp}) \simeq \aA_r(\eE)^{\opp}$, see (\ref{rem_envs_for_opp}), we mostly restrict our attention to right abelian envelopes.

If the right abelian envelope of $\eE$ exists then it is equivalent to the subcategory of compact objects in $\Lex(\eE^{\opp}, \Ab)$ and the functor $i_R\colon \eE\to\aA_r(\eE) $ is faithful (see Lemma \ref{lem_envelope_fully_exact_in_Lex}). We don't know whether it is always fully faithful, but when it is so, $\eE$ is a fully exact subcategory of $\aA_r(\eE)$ (see Lemma \ref{lem_if_I_R_full_then_exact}).

The right abelian envelope extends to a monad on the 2-category of exact categories admitting the right abelian envelopes, see Section \ref{ssec_monad}.
Algebras over this monad are identified with abelian categories, see Proposition \ref{prop_mondad_algebras_over}. 

Using a statement of \cite{KasSch2}, we prove that an abelian category $\aA$ is the right abelian envelope of a full subcategory $\eE\subset \aA$ closed under extensions and kernels of epimorphisms, provided every object of $\aA$ is a quotient of an object of $\eE$, see Theorem \ref{thm_if_quotient_then_envelope}.
This implies that the right envelope of the category of vector bundles is the category of coherent sheaves (see Corollary \ref{cor_Coh_as_env}).

If an exact category $\eE$ admits a projectively generating subcategory $\pP$ with weak kernels, then the right abelian envelope $\aA_r(\eE)$ is equivalent to the category $\textrm{fp}(\pP)$ of finitely presented functors $\pP^{\opp} \to \Ab$. 
Moreover, $\eE\subset \textrm{fp}(\pP)$ is a fully exact subcategory, 
see Theorem \ref{thm_right_env_if_proj_gen}.
In particular, a thin category admits the right and the left envelope, see Corollary \ref{cor_envel_for_thin}.

Under the same conditions we prove that $\dD^-(\eE) \simeq \dD^-(\aA_r(\eE))$, Proposition  \ref{prop_derived_equiv}.
If further $\aA_r(\eE)$ has finite global dimension then $\dD^b(\eE) \simeq \dD^b(\aA_r(\eE))$, see Theorem \ref{thm_bounded_derived_equi}. 

A perpendicular torsion pair $(\tT, \fF)$ on an exact category $\eE$  induces a colocalising subcategory of the right envelope $\aA_r(\fF) \subset \aA_r(\eE)$ if $\tT$ and $\eE$ have right envelopes, see Proposition \ref{prop_envelope_becomes_Serre_subcat}.
More generally, a strict admissible filtration on $\eE$ with graded factors admitting right envelopes induces a colocalising filtration on $\aA_r(\eE)$, see Theorem \ref{thm_filtr_on_envelope}.

In Section \ref{sec_high_weigh_as_env} we focus on  highest weight categories and show that they are abelian envelopes of thin categories.

Given a highest weight category $(\aA, \Lambda)$, its subcategories $\fF(\Delta_{\Lambda})$ and $\fF(\nabla_{\Lambda})$, of objects with filtration with
standard, respectively costandard, factors, are thin. This gives an essentially one-to-one correspondence between highest weight abelian categories and thin exact categories, the inverse operation given by taking  the right or left envelope (see Propositions \ref{prop_F_Delta_thin}, \ref{prop_highest_weuight_is_right_env} and Theorem \ref{thm_right_env_is_h_w}).

An abelian category $\aA$ might have several highest weight structures which depend on the choice of partial orders on the set of isomorphism classes of simple objects. We consider the equivalence relation on the partial orders defined by equality of the corresponding thin subcategories $\fF(\Delta)\subset \aA$, see (\ref{eqtn_equiv_hw_st}). We prove that posets $\Lambda$ and $\Lambda'$ are equivalent if and only if $\Lambda'$ dominates the canonical poset of the thin category $\fF(\Delta_\Lambda)$, see Proposition \ref{prop_hw_iff_dominated}.

We characterise a thin category $\eE$ inside $\aA_r(\eE)$ using the strict filtration on $\aA_r(\eE)$ induced by the canonical filtration on $\eE$ (see Theorem \ref{thm_filtr_on_envelope}).
Objects of $\eE$
are precisely  those $A\in \aA_r(\eE)$ for which the adjunction counits ${\beta_{\la}}_!\beta_\la^*A \to A$ are monomorphisms, where ${\beta_{\la}}_!\colon \aA_r(\tT_\lambda) \to \aA_r(\eE)$ is the functor induced by the inclusion $\tT_\la \subset \eE$, 
see Theorem \ref{thm_thin_in_its_envelope}.

A highest weight category $(\aA, \Lambda)$ admits \emph{Ringel dual} $\mathbf{RD}(\aA)$ \cite{Rin}, originally defined as the category of modules over the endomorphism algebra of the characteristic tilting object $T \in \aA$.  In Section \ref{ssec_ring_dual_for_hw} we prove
(see Theorem \ref{thm_Ringel_dual_for_hwc})
$$
\mathbf{RD}(\aA) \simeq \aA_l(\fF(\Delta_\Lambda)) \simeq \aA_r(\fF(\nabla_\Lambda)).
$$
Since $\aA \simeq \aA_r(\fF(\Delta_\Lambda))$, Ringel duality is the transfer from the right to the left abelian envelope of a thin category.

We show that a thin category $\eE$ is the intersection of its left and right envelopes in $\dD^b(\eE)$, see Theorem \ref{thm_E_is_the_int_of_hearts}. 
We define the \emph{Ringel dual} $\mathbf{RDT}(\eE)$ of a thin exact category $\eE$,
$$
\mathbf{RDT}(\eE):=\aA_l(\eE)\cap \mathbb{S}^{-1}\aA_r(\eE).
$$
The category  $\bf{RDT}(\eE)$ is again thin exact and $\bf{RDT}(\bf{RDT}(\eE))$ is canonically equivalent to $\eE$, see Proposition \ref{prop_RIn_dual_is_thin}.
Given a highest weight category $(\aA, \Lambda)$, the thin exact categories $\fF(\Delta_\Lambda)$ and $\fF(\nabla_\Lambda)$ are Ringel dual to each other.

In Appendix \ref{sec_abelian_cat} we show that the data of an abelian recollement is equivalent to a
bi-localising subcategory (see Proposition \ref{prop_charact_of_abel_recol}). We establish a one-to-one correspondence between Serre subcategories in a finite length abelian category and subsets of the set of simple objects in it (this should be well-known to experts), see Lemma  \ref{lem_biloc_in_fin_len}. 
For Deligne finite categories (i.e. finite length categories which admit projective generators), we prove that bi-localising subcategories are exactly Serre subcategories, see Proposition \ref{prop_biloc_in_Df_n}.

\textbf{Notation} For a field $k$ we denote by $k\textrm{-vect}$ the category of finite dimensional $k$-vector spaces.  To simplify notation we denote by $\otimes$ the tensor product over the ground field $k$.

\textbf{Acknowledgements}
It should be clear from the exposition how much this paper is  spiritually  indebted to Claus Michael Ringel. It is our pleasure to dedicate it to this wonderful mathematician whose ideas continue to inspire generations of researchers.
We are grateful to A. Beligiannis, A. Efimov and D. Orlov for useful discussions.
The first named author would like to thank Kavli IPMU and Higher School of Economics for their hospitality. The first named author was partially supported by Polish National Science Centre grants No. 2018/31/D/ST1/03375 and 2021/41/B/ST1/03741. 
The work of A. I. Bondal was performed at the Steklov International Mathematical Center and supported by the Ministry of Science and Higher Education of the Russian Federation (agreement no. 075-15-2022-265).
This research was partially supported by World Premier International Research Center Initiative (WPI Initiative), MEXT, Japan.
This work was supported by JSPS KAKENHI Grant Number JP20H01794.

\section{Exact categories}\label{sec_exact_cat}

We recall the definition of an exact category and introduce the crucial notion of left and right admissible subcategories. We characterise them in terms of exactness of appropriate adjoint functors. Then, we discuss strict admissible filtrations on exact categories and the induced filtrations on the derived categories.

\vspace{0.3cm}
\subsection{Preliminaries}\label{ssec_prelim}~\\

Exact categories were defined by N. Yoneda \cite{Yoneda}, A.Heller \cite{Heller} and D.Quillen in  \cite{Quillen}. Below we reproduce a version of the definition as in \cite{Kel4}.	

An \emph{exact category} is an additive category $\eE$ together with a fixed class $\sS$ of \emph{conflations}, i.e. pairs of composable morphisms
\begin{equation}\label{eqtn_confl}
X \xrightarrow{i} Y \xrightarrow{d} Z
\end{equation} 
such that $i$ is the kernel of $d$ and $d$ is the cokernel of $i$. We shall say that $i$ is an \emph{inflation} and $d$ a \emph{deflation}. The class $\sS$ is closed under isomorphisms and the pair $(\eE, \sS)$ is to satisfy the following axioms:
\begin{itemize}
	\item[(Ex 0)] $0 \to X \xrightarrow{\Id_X} X$ is a conflation,
	\item[(Ex 1)] the composite of two deflations is a deflation,
	\item[(Ex 2)] the pullback of a deflation against an arbitrary morphism exists and is a deflation,
	\item[(Ex 2')] the pushout of an inflation along an arbitrary morphism exists and is an inflation.
\end{itemize}

For objects $E$ and $F$ in an exact category $\eE$, we denote by $\Ext^1_\eE(E,F)$ the abelian group whose elements are classes of conflations
$F\to X \to E$ considered up to isomorphisms which are identical on $F$ and $E$.

An additive functor between exact categories is \emph{exact} if the image of every conflation is a conflation. We denote by $\Ex(\eE,\eE')$ the category whose objects are exact functors $\eE \to \eE'$ and morphisms are natural transformations.

A full subcategory $\xX \subset \eE$ of an exact category is said to be \emph{closed under extensions} if for any conflation (\ref{eqtn_confl}) with $X, Z \in \xX$ object $Y$ also belongs to $\xX$. If $\xX$ satisfies this condition then the conflations in $\eE$ whose all three terms lie in $\xX$ yield an exact structure on $\xX$. By \emph{fully exact subcategory of} $\eE$ we mean a full subcategory closed under extensions and endowed with the induced exact structure. The embedding functor of a fully exact subcategory is clearly exact and reflects exactness. 

An additive category $\bB$ has the \emph{split exact structure} with deflations split epimorphisms and inflations split monomorphisms.

An abelian category $\aA$ is regarded as an exact category $(\aA, \sS)$ with the class $\sS$ of conflations being all short exact sequences in $\aA$. In particular exact functors $\eE\to \aA$ from an exact category to an abelian category make sense.

A rich source of examples of exact categories is given by extension closed subcategories of abelian categories with the induced exact structure. According to the Gabriel-Quillen theorem \cite[II.3]{Gabriel}, \cite[§2]{Quillen}, \textit{cf.} \cite[Theorem 1.0.3]{Laumon}, \cite[Ex. 4.7.3.G]{Freyd}, every exact category $\eE$ is a fully exact subcategory of an abelian category $\aA$. 

Under this embedding $\eE \to \aA$ into an abelian category,  $\Hom_{\eE}(E_1,E_2) \simeq \Hom_{\aA}(E_1,E_2)$ and $\Ext^1_{\eE}(E_1,E_2) \simeq \Ext^1_{\aA}(E_1,E_2)$, for any $E_1,E_2 \in \eE$. It follows that for any conflation (\ref{eqtn_confl}) and any $E \in \eE$ we have the six-terms exact sequences: 
\begin{equation}\label{eqtn_6_term_ex_seq} 
\begin{aligned}
	0 \to \Hom_{\eE}(E,X) \to &\Hom_{\eE}(E,Y) \to \Hom_{\eE}(E,Z) \to\\
	&\to \Ext^1_{\eE}(E,X) \to \Ext^1_{\eE}(E,Y) \to \Ext^1_{\eE}(E,Z),&\\
	0 \to \Hom_{\eE}(Z,E) \to &\Hom_{\eE}(Y,E) \to \Hom_{\eE}(X,E) \to\\
	&\to \Ext^1_{\eE}(Z,E) \to \Ext^1_{\eE}(Y,E) \to \Ext^1_{\eE}(X,E)&
\end{aligned}
\end{equation}

\vspace{0.3cm}
\subsection{Weakly left and right admissible subcategories}\label{ssec_weak_l_and_r_adm}~\\

Let $\xX \subset \eE$ be a subcategory of an additive category $\eE$. We define \emph{the right} (respectively \emph{left}) \emph{orthogonal subcategory} $\xX^{\perp_0}$ (respectively ${}^{\perp_0} \xX$) by:
\begin{equation}
	\begin{aligned}
	&\xX^{\perp_0} = \{E \in \eE\,|\, \Hom(\xX, E) = 0\},&\\
	&{}^{\perp_0}\xX = \{E \in \eE\,|\, \Hom(E, \xX) =0\}.&
	\end{aligned}
\end{equation} 

Category ${}^{\perp_0}\xX$ is closed under quotients, i.e. given an epimorphism $X\to X'$ with $X \in \xX^{\perp_0}$ object $X'$ lies in $\xX^{\perp_0}$. Similarly, $\xX^{\perp_0} $ is closed under subobjects.

\begin{DEF}\label{def_weakly_right_adm} 
We say that a fully exact subcategory $\tT$ of an exact category $ \eE$ is \emph{weakly right admissible} if the inclusion functor $i_*\colon \tT\to \eE$ admits right adjoint $i^! \colon \eE \to \tT$ and the adjunction counit $\varepsilon_E \colon i_*i^!E \to E$ is an inflation, for any $E \in \eE$.
\end{DEF}
\begin{DEF}\label{def_weakly_left_adm} 
We say that a fully exact subcategory $\fF \subset \eE$ is \emph{weakly left admissible} if the inclusion functor $\defl_* \colon \fF \to \eE$ admits a left adjoint $\defl^* \colon \eE\to \fF$ and the adjunction unit $\eta_E \colon E \to \defl_*\defl^*E$ is a deflation, for any $E \in \eE$. 
\end{DEF}

Let $\tT$, $\fF$ be fully exact subcategories of an exact category $\eE$. We say that $(\tT, \fF)$ is a \emph{torsion pair} if $\Hom (\tT , \fF )=0$ and any object $E \in \eE$ fits into a conflation $T \to E \to F$ with $T\in \tT$ and $F \in \fF$.

\begin{PROP}\label{prop_torsion=adjoint}
	A fully exact subcategory $\tT$ of an exact category $\eE$ is weakly right admissible  if and only if $(\tT, \fF)$, where $\fF :=\tT^{\perp_0}$, is a torsion pair in  $\eE$. In this case, every conflation $T\to E\to F$ with $T\in \tT$ and $F\in \fF$ is canonically isomorphic (identically on $E$) to the functorial conflation 
\begin{equation}\label{decomposition}
	\infl_*\infl^!E \to E \to \defl_*\defl^*E,
\end{equation}
	where $\infl_*:\tT \to \eE$, and $\defl_*:\fF \to \eE$ are the embedding functors, $\infl^!$ and $\defl^*$ are their adjoints and morphisms in (\ref{decomposition}) are, respectively, the adjunction counit and unit. An analogous fact holds for weakly left admissible subcategories.
\end{PROP}
\begin{proof}
	Let $(\tT, \fF)$ be a torsion pair in an exact category $\eE$. Consider a conflation $T \xrightarrow{in} E \xrightarrow{d} F$ with $T\in \tT$ and $F\in \fF$.
	
	For $T' \in \tT$, the map $\Hom(T', T) \to \Hom(T', E)$ given by the composition with $in$ is an isomorphism, since its cokernel belongs to $\Hom(T', F)=0$. Similarly, $\Hom(F,F') \simeq \Hom(E, F')$, for any $F' \in \fF$. Hence, we have canonical isomorphisms $T \simeq \infl_*\infl^!E$ and $F \simeq \defl_* \defl^* E$. Under these isomorphisms, the adjunction counit $\infl_*\infl^!E \to E$ is identified with $in$, hence an inflation, and the adjunction unit $E \to \defl_*\defl^*E$ is identified with $d$, hence a deflation. Thus, $\tT$ is weakly right admissible and $\fF$ is weakly left admissible.
	
	Let now $\infl_* \colon \tT \to  \eE$ be a weakly right admissible fully exact subcategory. We aim to show that $F$ defined as the cokernel of $T:=\infl_*\infl^!E \to E$ belongs to $\tT^{\perp_0}$. For $T' \in \tT$ the first map in the exact sequence $\Hom_{\eE}(T',T)\to \Hom_{\eE}(T',E) \to \Hom_{\eE}(T', F) \to \Ext^1_{\eE}(T',T)$ (see (\ref{eqtn_6_term_ex_seq})) is an isomorphism. Hence, in order to show that $\Hom(T',F) =0$ it suffices to check that for any $\f \colon T' \to F$ the induced extension of $T'$ by $T$ splits:
	\[
	\xymatrix{T \ar[r]^{in} & E \ar[r]& F\\
	T \ar[r]^{\ol{in}} \ar[u]^{\Id} & T'' \ar[u]^{\ol{\f}} \ar@{-->}[ul]|{\psi} \ar[r] & T' \ar[u]^{\f} }
	\] 
	Object $T''$ belongs to $\tT$, because $\tT$ is closed under extensions. Since $T = \infl_*\infl^!E$ morphism $\ol{\f}$ factors via $\psi$, i.e. $\ol{\f} = in \circ \psi$. Then $in \circ \psi \circ \ol{in} = \ol{\f} \circ \ol{in} = in$ which implies that $\psi \circ \ol{in} = \Id$, as $in$ is a monomorphism. Hence the induced extension of $T'$ by $T$ splits and $\f  =0$. As $\f$ was arbitrary, $F \in \tT^{\perp_0}$, i.e. $(\tT, \tT^{\perp_0})$ is a torsion pair. 
	
	We have shown that torsion pairs are in one to one correspondence with weakly right admissible subcategories. By considering the opposite categories we obtain an analogous result for weakly left admissible subcategories.
\end{proof}

Note that functors $\infl^!$ and $\defl^*$ are not necessarily exact. In the next subsection, we put stronger conditions on the subcategory $\tT \subset \eE$ that ensure exactness of these functors.

\begin{PROP}\label{lem_right-left}
	We have an involutive duality between weakly right admissible subcategories $\tT$ and weakly left admissible subcategories $\fF$ in an exact category $\eE$:
\begin{equation}
\begin{aligned}
\tT \mapsto \fF :=\tT^{\perp_0}, &\\
\fF \mapsto \tT :={}^{\perp_0}\fF.
\end{aligned}
\end{equation}
In particular, $\tT$ is closed under quotients in $\eE$, and $\fF$ is closed under subobjects in $\eE$.
\end{PROP}
\begin{proof}
Let $\tT \subset \eE$ be a weakly right admissible subcategory and $\fF :=\tT^{\perp_0}$. Proposition \ref{prop_torsion=adjoint} implies that $\eE = (\tT, \fF)$ is a torsion pair and $\fF\subset \eE$ is weakly left admissible. The decomposition (\ref{decomposition}) of any $T'\in {}^{\perp_0} \fF$ with respect to the torsion pair $(\tT, \fF)$ is a conflation $T \to T' \to 0$, with $T\in \tT$. Hence ${}^{\perp_0}\fF = {}^{\perp_0}(\tT^{\perp_0})\subset \tT$. As the inverse inclusion is clear, we have: $\tT = {}^{\perp_0}(\tT^{\perp_0})$.

By transferring to the opposite categories, we conclude that if $\fF \subset \eE$ is weakly left admissible then ${}^{\perp_0} \fF$ is weakly right admissible and $\fF = (^{\perp_0}\fF)^{\perp_0}$.
\end{proof}

\vspace{0.3cm}
\subsection{Left and right admissible subcategories}\label{ssec_l_and_r_adm}~\\

Let $\xX \subset \eE$ be a subcategory of an exact category $\eE$. We define  \emph{the right} (respectively \emph{left}), \emph{ perpendicular subcategory} $\xX^{\perp}$ (respectively ${}^\perp \xX$) by:
\begin{equation}\label{eqtn_perp_cat}
\begin{aligned} 
&\xX^\perp = \{E \in \eE\,|\, \Hom(\xX, E) = 0 = \Ext^1(\xX, E)\},&\\
&{}^\perp \xX = \{E \in \eE\,|\, \Hom(E,\xX) = 0 = \Ext^1(E,\xX)\}.&
\end{aligned}
\end{equation} 
Category $\xX^\perp$ is closed under kernels of deflations, i.e. given a conflation $E' \to E \to E''$ with $E, E'' \in \xX^\perp$ object $E'$ lies in $\xX^\perp$. Similarly, ${}^\perp \xX$ is closed under cokernels of inflations.

Perpendicular categories in \emph{abelian} categories were introduced by W. Geigle and H. Lenzing \cite{GeiLen}, who emphasized their importance in the representation theory of associative algebras. In this paper, we use a more general notion of perpendicular categories in exact categories as a basic tool for studying exact categories.

\begin{DEF}\label{def_right_and_left_adm} 
	We say that a fully exact subcategory $\xX \subset \eE$ is \emph{right}  (respectively, \emph{left}) \emph{admissible} if $\eE$ admits a torsion pair $(\xX, \xX^{\perp})$, (respectively $({}^\perp \xX, \xX)$). 
\end{DEF}
In view of Proposition \ref{prop_torsion=adjoint}, $\xX \subset \eE$ is right (respectively, left) admissible if and only if $\xX$ is weakly right (respectively, left) admissibe and $\xX^{\perp}=\xX^{\perp_0}$ (respectively,  ${}^{\perp_0} \xX={}^\perp \xX$).
In particular, by Proposition \ref{lem_right-left}, any right admissible subcategory is closed under quotients, and any left admissible subcategory is closed under subobjects.

\begin{PROP}\label{right-left}
	We have an involutive duality between right admissible subcategories $\tT$ and left admissible subcategories $\fF$ in an exact category $\eE$:
\begin{equation}
\begin{aligned}
\tT \mapsto \fF :=\tT^{\perp}, &\\
\fF \mapsto \tT :={}^{\perp}\fF.
\end{aligned}
\end{equation}
\end{PROP}
\begin{proof}
Let $\tT \subset \eE$ be a right admissible subcategory and $\fF :=\tT^{\perp}=\tT^{\perp_0}$. By Proposition \ref{lem_right-left}, $\tT = {}^{\perp_0}\fF$. By our assumptions, $\tT \subset {}^{\perp}\fF$. Since ${}^{\perp}\fF \subset {}^{\perp_0}\fF=\tT$, we have the equality: $\tT = {}^\perp\fF$, and $\fF$ is left admissible. Similarly for a left admissible subcategories $\fF \subset \eE$, we have $\fF = ({}^\perp \fF)^\perp$.
\end{proof}

We say that a \emph{torsion pair} $(\tT, \fF)$ is \emph{perpendicular} if $\tT = {}^\perp \fF$ (equivalently, $\fF = \tT^\perp$).

By \emph{filtration on an object} $E\in \eE$ we mean a sequence of inflations $0=E_0\to E_1\to \dots \to E_n=E$.
Since the composition of inflations is an inflation, we have an induced conflation $E_i\to E\to Q_{i+1}$, for every $i$. Objects $Q_i$ fit into a sequence of deflations $E=Q_1\to \dots \to Q_n\to Q_{n+1}=0$, which equivalently define the filtration on $E$.

The \emph{graded factors} $G_i$ \emph{of the filtration}, are defined via conflations $E_{i-1}\to E_i\to G_i$. Also they fit conflations $G_i\to Q_i\to Q_{i+1}$. All this can be deduced from the similar facts for abelian categories via the Gabriel-Quillen embedding theorem or proven directly.

\begin{LEM}\label{lem_filtration_torsion}
	 Let $\tT\subset \eE$ be a fully exact subcategory and $\fF\subset\tT^{\perp}$ a fully exact subcategory. If every object in $\eE$ has a filtration with all graded factors being either in $\tT$ or in $\fF$, then $\tT$ is right admissible and $\fF=\tT^{\perp}$.
\end{LEM}
\begin{proof}
	We prove by induction on the length of the filtration on a given object $E$ that $E$ fits into a conflation $T\to E\to F$ with $T\in \tT$ and $F\in \fF$. Such a conflation for any object in $\tT^{\perp}$ implies that it is in $\fF$, hence $\fF=\tT^{\perp}$.
	
	Since the quotient $Q_2$ has a shorter filtration, it fits, by induction hypothesis, into a conflation $T_2\to Q_2\to F_2$, with $T_2\in \tT$ and $F_2\in \fF$. If $E_1\in \tT$, then the pullback of $E_1\to E\to Q_2$ along $T_2\to Q_2$ is a conflation $\sigma : E_1\to E'\to T_2$. Hence $E'\in \tT$, and $E'\to E\to F_2$ is the sought conflation.
	
	If $E_1\in \fF$, then applying the second of sequences (\ref{eqtn_6_term_ex_seq}) to the conflation $T_2\to Q_2\to F_2$ and object $E_1$ gives an isomorphism $\Ext^1(F_2, E_1)\simeq \Ext^1(Q_2, E_1)$. This means that our conflation $E_1\to E \to Q_2$ is the pullback of a conflation $E_1\to F\to F_2$ along $Q_2\to F_2$:
	\[
	\xymatrix{E_1 \ar[r]& F \ar[r] & F_2\\
	E_1\ar[r]\ar[u]^{\simeq}& E \ar[r]\ar[u] & Q_2 \ar[u]\\
	& T_2\ar[r]^{\simeq} \ar[u] & T_2 \ar[u]}
	\]
	It follows that $F\in \fF$ and the middle column is the sought conflation.  
\end{proof}

\vspace{0.3cm}
\subsection{Exactness of adjoints for embeddings of admissible subcategories}\label{ssec_ex_of_adj_for_emb}~\\

The following property of morphisms in an exact category was proved by Yoneda in \cite{Yoneda}, then reintroduced by D. Quillen as a part of axiomatics of exact categories and later again reproved by B. Keller.
\begin{LEM}[Quillen's obscure axiom]\cite{Yoneda} \cite{Kel4}\label{lem_obscure_axiom}
	Let $d\colon X\to Y$ be a morphism in an exact category $\eE$ which admits a kernel. If there exists $\f \colon X' \to X$ such that $d\circ \f$ 
is a deflation then $d$ is a deflation.
\end{LEM}

We will use the following lemma in the proof of Theorem \ref{thm_admissible_iff_adjoint_exact} on exactness of adjoints.

\begin{LEM}\label{lem_j*_of_a_conf}
	Let $\defl_* \colon \fF \to \eE$ be a left admissible subcategory and $\sigma$ a conflation $F \xrightarrow{in} E \xrightarrow{d} E'$ in $\eE$ with $F\in \fF$. Then $\defl^*(\sigma )$ is a conflation in $\fF$ and $\sigma$ is the pull-back of $\defl_*\defl^*(\sigma )$ along the adjunction unit $\eta_{E'}:E' \to \defl_*\defl^*E'$.
\end{LEM}
\begin{proof}
	Since $\defl_*$ reflects exactness, $\defl^*(\sigma )$ is a conflation in $\fF$ if and only if $\defl_*\defl^*(\sigma )$ is a conflation in $\eE$. It is enough to check that $\defl_*\defl^*(in)$ is the kernel of $\defl_*\defl^*(d)$. Indeed, as $\eta_{E'}$ is a deflation by Definition \ref{def_weakly_left_adm}, the composite $\defl_*\defl^*(d) \circ \eta_E = \eta_{E'} \circ d$ is the result of composition of deflations, hence a deflation. By Quillen's obscure axiom, Lemma \ref{lem_obscure_axiom}, if $\defl_*\defl^*(d)$ has a kernel, it is a deflation. Then, $\defl_*\defl^*F \xrightarrow{\defl_*\defl^*(in)} \defl_*\defl^*E \xrightarrow{\defl_*\defl^*(d)} \defl_*\defl^* E'$ is a conflation.
	
	Let $\tT := {}^\perp \fF$ and let $i^!$ be the right adjoint to the inclusion $\infl_*\colon \tT\to \fF$. Objects $E$ and $E'$ admit the decompositions (\ref{decomposition}) with respect to the torsion pair $(\tT, \fF)$, which occur as rows in the following diagram: 
	\begin{equation}\label{eqtn_diagram_2}
	\xymatrix{&\infl_*\infl^!E' \ar[r]^{\varepsilon_{E'}}& E' \ar[r]^{\eta_{E'}}& \defl_*\defl^*E' \\
	&\infl_*\infl^!E \ar[r]^{\varepsilon_E} \ar[u]|{\infl_*\infl^!(d)} & E \ar[u]|d \ar[r]^{\eta_E} & \defl_*\defl^*E \ar[u]|{\defl_*\defl^*(d)} \\
& &F \ar[u]|{in} \ar[r]^{\eta_F}_\simeq & \defl_*\defl^*F \ar[u]|{\defl_*\defl^*(in)} }
	\end{equation}
	
	We will show that $\infl_*\infl^!(d)\colon \infl_*\infl^!E \to \infl_*\infl^!E'$ is an isomorphism. 
	
	Let $T\in \tT$ be any object. In the exact sequence $\Hom_{\eE}(T,F) \to \Hom_{\eE}(T,E) \xrightarrow{d\circ (-)} \Hom_{\eE}(T, E') \to \Ext^1_{\eE}(T,F)$ (see (\ref{eqtn_6_term_ex_seq})), the first and the last group vanish because $(\tT, \fF)$ is a perpendicular torsion pair. In other words, $d$ induces an isomorphism of functors on $\tT$: $\Hom_{\eE}(-,E) \xrightarrow{d\circ (-)} \Hom_{\eE}(-, E')$. Since $\infl^!E$ and $\infl^!E'$ are objects representing these functors, $\infl^!(d)$ is an isomorphism. Hence, so is $\infl_*\infl^!(d)$.
	
	Since $\eta_E$ and $\eta_{E'}$ have isomorphic kernels, the standard diagram chasing shows that $\defl_*\defl^*(d)$ has the kernel isomorphic to that of $d$. As $\eta_F$ is an isomorphism, the kernel of $\defl_*\defl^*(d)$ 	is identified with $\defl^*\defl_*(in)$.

	Let $\tilde{E}$ be the fibre product of $E'$ and $\defl_*\defl^*E$ over $\defl_*\defl^*E'$. The canonical map $E\to \tilde{E}$ extends to a morphism of conflation $\sigma$ to the pull-back conflation $\defl_*\defl^*F\to \tilde{E}\to E'$,
	which has isomorphisms on the ends of conflations: $F\xrightarrow{\simeq} \defl_*\defl^*F$ and $E'\xrightarrow{\rm{id}_{E'}} E'$. Hence the conflations are canonically isomorphic.
\end{proof}

\begin{LEM}\cite[Proposition 2.15]{Buehl}\label{lem_push_def}
	The pullback of an inflation along a deflation exists and is an inflation. The pushout of a deflation along an inflation exists and is a deflation.
\end{LEM}

\begin{THM}\label{thm_admissible_iff_adjoint_exact}
	A weakly left admissible subcategory $\defl_*\colon \fF \to \eE$ of an exact category $\eE$ is left admissible if and only if $\defl^* \colon \eE\to \fF$ is exact.
	A weakly right admissible subcategory $\infl_*\colon \tT \to \eE$  of an exact category $\eE$ is right admissible if and only if $\infl^! \colon \eE\to \tT$ is exact. 
\end{THM}
\begin{proof}
	Since the transfer to the opposite categories $\eE^{\opp}$, $\tT^{\opp}$, $\fF^{\opp}$ exchanges the left and right (weakly) admissible subcategories and preserves the exactness of functors, it suffices to prove the equivalence for weakly left admissible subcategories.
	
	Consider a weakly left admissible subcategory $\fF\subset \eE$ with an exact functor $\defl^*\colon \eE\to \fF$.  Let $\sigma$ be a conflation $F \xrightarrow{in} E \xrightarrow{d} T$ in $\eE$ with $F \in \fF$ and $T \in {}^{\perp_0} \fF$. Vanishing of $\defl^*T $ and exactness of $\defl^*$ imply that $j=\defl^*(\sigma )$ is a conflation $\defl^*F \xrightarrow{\simeq} \defl^*E \to 0$, meaning that $\defl_*\defl^*(in)$ is an isomorphism $F = \defl_*\defl^*F \xrightarrow{\simeq} \defl_*\defl^*E$. Then $\eta_E \colon E \to \defl_*\defl^*E\simeq F$ yields a splitting of $\sigma$. Hence $\Ext^1({}^{\perp_0} \fF, \fF) = 0$, which implies that $\fF$ is left admissible in $\eE$.
	
	Let now $\fF$ be a left admissible subcategory and $\sigma$ a conflation $A \xrightarrow{in} B \xrightarrow{d} C$ in $\eE$. We aim at showing that $\defl^*(\sigma )$ is a conflation in $\fF$. As $\fF$ is a fully exact subcategory, it suffices to check that $\defl_*\defl^*(\sigma )$ is a conflation in $\eE$.
	
	Let $T_A \xrightarrow{\varepsilon_A} A \xrightarrow{\eta_A} \defl_*\defl^* A$ be the decomposition (\ref{decomposition}) for $A$ with respect to the torsion pair $({}^\perp \fF,\fF)$. The pushout of $\sigma$ along the adjunction unit $\eta_A\colon A \to \defl_*\defl^*A$ yields an extension of $C$ by $\defl_*\defl^*A$:
	\begin{equation}\label{eqtn_diag_1}
	\xymatrix{\defl_*\defl^*A \ar[r]^{\alpha} & \wt{B} \ar[r]^{\beta} & C \\
	A \ar[u]^{\eta_A} \ar[r]^{in} & B \ar[r]^d \ar[u]^{\gamma} & C \ar[u]^{\Id} \\ 
T_A \ar[r]^{\Id} \ar[u]^{\varepsilon_A} & T_A \ar[u]^{in\circ \varepsilon_A}}
	\end{equation}
	By Lemma \ref{lem_push_def}, $\gamma$ is a deflation, hence the middle column in the above diagram
	$T_A \xrightarrow{in\circ \varepsilon_A} B \xrightarrow{\gamma} \wt{B}$
	 is a conflation in $\eE$. 
	
	Lemma \ref{lem_j*_of_a_conf} applied to the conflation $\defl_*\defl^*A \xrightarrow{\alpha} \wt{B} \xrightarrow{\beta} C$ yields a conflation  $\defl_*\defl^*A \xrightarrow{\defl_*\defl^*(\alpha)} \defl_*\defl^*\wt{B} \xrightarrow{\defl_*\defl^*(\beta)} \defl_*\defl^*C$:
	\begin{equation}\label{eqtn_diag_3}
	\xymatrix{\defl_*\defl^*A\ar[r]^{\defl_*\defl^*(\alpha)} &\defl_*\defl^*\wt{B} \ar[r]^{\defl_*\defl^*(\beta)} & \defl_*\defl^*C\\
		\defl_*\defl^*A \ar[u]^{\Id} \ar[r]^{\alpha} & \wt{B} \ar[u]^{\eta_{\wt{B}}} \ar[r]^\beta & C \ar[u]^{\eta_C}\\
	& T_C \ar[r]^{\Id} \ar[u]^{\varepsilon_{\wt{B}}} & T_C \ar[u]^{\varepsilon_C}}
	\end{equation}
	Moreover, we can apply Lemma \ref{lem_j*_of_a_conf} to the middle horizontal conflation in (\ref{eqtn_diag_3}) and conclude that $\eta_{\wt{B}}$ is a pullback of $\eta_C$, hence it is a deflation. Since $T_C \in {}^\perp \fF$ and $\defl_*\defl^*\wt{B} \in \fF$, the conflation $T_C \xrightarrow{\varepsilon_{\wt{B}}} \wt{B} \xrightarrow{\eta_{\wt{B}}} \defl_*\defl^* \wt{B}$ is the decomposition (\ref{decomposition}) of $\wt{B}$ with respect to the torsion pair $({}^{\perp}\fF, \fF)$ (hence the notation for the morphisms).

	Note that commutativity of (\ref{eqtn_diag_1}) implies that $\defl^*(\alpha )= \defl^*(\gamma )\circ \defl^*(in) $ and $\defl^*(d) = \defl^*(\beta )\circ \defl^*(\gamma )$. To prove that $\defl_*\defl^*(\sigma)$ is a conflation it remains to show that $\defl^*(\gamma )\colon \defl^*B \to \defl^*\wt{B}$ is an isomorphism. Indeed, by applying functor $\defl_*\defl^*$ to (\ref{eqtn_diag_1}), we see that the conflation in the top row of (\ref{eqtn_diag_3}) is then isomorphic to $\defl_*\defl^*(\sigma)$:
	\[
	\xymatrix{\defl_*\defl^*A\ar[r]^{\defl_*\defl^*(\alpha)} &\defl_*\defl^*\wt{B} \ar[r]^{\defl_*\defl^*(\beta)} & \defl_*\defl^*C\\
	\defl_*\defl^*A\ar[r]^{\defl_*\defl^*(in)} \ar[u]|{\Id} &\defl_*\defl^*B\ar[u]|{\defl_*\defl^*(\gamma)} \ar[r]^{\defl_*\defl^*(d)} & \defl_*\defl^*C\ar[u]|\Id}
	\]

	Consider the pullback of the conflation in the middle column in (\ref{eqtn_diag_1}) along $\varepsilon_{\wt{B}} \colon T_C \to \wt{B}$: 
\begin{equation}\label{mattress4}
	\xymatrix{& \defl_*\defl^*\wt{B} \ar[r]^{\Id} & \defl_*\defl^* \wt{B}\\ 
		T_A \ar[r]^{in \circ \varepsilon_A} & B \ar[r]^{\gamma} \ar[u]^{\theta} & \wt{B} \ar[u]^{\eta_{\wt{B}}} \\ T_A \ar[u]^{\Id} \ar[r] & \wt{T} \ar[r] \ar[u]^{\zeta} & T_C \ar[u]^{\varepsilon_{\wt{B}}} }
\end{equation}
	By Lemma \ref{lem_push_def}, $\zeta$ is an inflation, hence $\wt{T} \xrightarrow{\zeta} B \xrightarrow{\theta} \defl_*\defl^*\wt{B}$ 
	is a conflation. Since $\wt{T}$ is an extension of $T_C$ by $T_A$, it is an object of ${}^\perp\fF$. Hence, the above conflation is of the form (\ref{decomposition}), and we have canonical isomorphisms $\wt{T} \simeq \infl_*\infl^!B$, $\defl_*\defl^*\wt{B} \simeq \defl_*\defl^*B$, and $\theta =\eta_B$. By applying $\defl^*$ to (\ref{mattress4}), we see that  $\defl^*(\gamma )$ is an isomorphism, which finishes the proof.
\end{proof}

Now we describe admissible subcategories in the case $\eE$ is an abelian category. We recall the notion of a (co)localising subcategory of an abelian category in Appendix \ref{sec_abelian_cat}.
\begin{PROP}
	A fully exact subcategory $\tT$ (resp. $\fF$) of an abelian category $\aA$ is right (left) admissible if and only if $\tT \subset \aA$ (resp. $\fF\subset \aA$) is a colocalising (localising) subcategory and $\tT^{\perp_0} = \tT^{\perp}$ (resp. ${}^{\perp_0}\fF = {}^{\perp} \fF$).
\end{PROP}
\begin{proof}
	Assume that $(\tT, \fF)$ is a perpendicular torsion pair in $\aA$. 
	Let $0 \to A' \to A \to A'' \to 0$ be a short exact sequence in $\aA$. Since $\defl^*$ is exact, $\defl^*(A) =0$ if and only if $\defl^*(A') =0$ and $\defl^*(A'')=0$. Hence, $\tT$ is a Serre subcategory.
	Since the torsion pair is perpendicular, $\tT^{\perp_0} \simeq \tT^{\perp}= \fF$. In particular, $\aA/\tT \simeq \tT^{\perp}= \fF$ \cite{Gabriel}. As the embedding $\fF\to \aA$ is right adjoint to the quotient $\aA \to \fF \simeq \aA/\tT$, category $\tT$ is colocalising. 
	
	Now let $\tT\subset \aA$ be a colocalising subcategory such that $\fF:=\tT^{\perp_0} \simeq \tT^{\perp}$. The quotient functor $\defl^*\colon \aA \to \aA/\tT \simeq \tT^{\perp}=\fF$ has a fully faithful right adjoint $\defl_* \colon \fF\to\aA$. The subcategory  $\fF = \tT^{\perp_0} \subset \aA$ is closed under extensions and subobjects. 
	As the inclusion functor $\defl_*\colon \fF \to \aA$ has left adjoint $\defl^*$, $\aA$ admits a torsion pair $({}^{\perp_0}\fF, \fF)$ \cite[Proposition 1.2]{BelRei}. Let $X$ be an object of $ {}^{\perp_0}\fF$. Vanishing of $\defl^*(X)$ implies that $X\in \tT$, hence ${}^{\perp_0}\fF \subset \tT$. The inverse inclusion is clear, i.e. $\aA$ admits a torsion pair $(\tT, \fF)$. As $\tT^{\perp_0} \simeq \tT^{\perp}$ the torsion pair is perpendicular and $\tT \subset \aA$ is right admissible.
\end{proof}

\vspace{0.3cm}
\subsection{The category of universal extensions as a square-zero extension category}\label{ssec_univ_ext_as_square_zero}~\\

Let $\aA$ be an additive category. An $\aA$ \emph{bimodule} is an additive bifunctor $\mM \colon \aA^{\opp} \times \aA \to \textrm{Ab}$.

Let $q\colon \aA \to \bB$ be a functor. A \emph{lifting} of an object $B\in \bB$ along $q$ is an object $A\in \aA$ together with an isomorphism $\zeta \colon q(A) \xrightarrow{\simeq} B$.

Let $\aA$ and $\bB$ be additive categories and $q\colon \aA\to \bB$ an additive functor. We define the \emph{kernel bimodule} $\kK\colon \aA^{\opp} \times \aA \to \aA b$ of $q$ via: 
$$
\kK(A, A') = \{\f\in \Hom_\aA(A,A')\,|\, q(\f)  = 0\}.
$$

\begin{DEF} We say that an additive functor $q\colon \aA \to \bB$ of additive categories is a \emph{square-zero extension} of $\bB$ if $q$ is essentially surjective, conservative and full, and the composition of any two morphisms in the kernel $\kK$ of $q$ is zero.
\end{DEF}
\begin{LEM}\label{lem_kernel_of_square_zer_ext}
	Let $q\colon \aA \to \bB$ be a square-zero extension with kernel $\kK$. There is unique up to unique isomorphism $\bB$ bimodule $\mM$ such that $\kK \simeq \mM \circ (q^{\opp} \times q)$.
\end{LEM}
\begin{proof}
 	Let $B_1, B_2$ be objects of $\bB$ and $(A_2, \xi)$ a lifting of $B_2$. We check that for two liftings $(A_1, \zeta)$, $(A'_1, \zeta')$ of $B_1$ the groups $\kK(A_1,A_2)$ and $\kK(A'_1, A_2)$ are canonically isomorphic.
	
	Consider the composite $q(A'_1) \xrightarrow{\zeta'} B_1 \xrightarrow{\zeta^{-1}}q(A_1)$. Since $q$ is full, there exists $\tau \in \Hom_{\aA}(A'_1,A_1)$ such that $q(\tau) = \zeta^{-1} \circ \zeta'$. As $q$ is conservative, $\tau$ is an isomorphism. Morphism $\tau$ induces a commutative diagram 
	\[
	\xymatrix{0 \ar[r] & \kK(A'_1,A_2) \ar[r] & \Hom_{\aA}(A'_1,A_2) \ar[r]^(0.4)q & \Hom_{\bB}(q(A'_1), q(A_2)) \ar[r] & 0\\
	0 \ar[r] & \kK(A_1,A_2) \ar[r] \ar[u]_{\simeq}^\mu& \Hom_{\aA}(A_1,A_2) \ar[r]^(0.4)q \ar[u]^{(-)\circ \tau}_{\simeq}& \Hom_{\bB}(q(A_1), q(A_2))\ar[u]^{(-) \circ \zeta^{-1} \circ \zeta'}_{\simeq} \ar[r] & 0 }
	\]
	Let $\tau' \colon A'_1 \to A_1$ be another morphism such that $q(\tau') = \zeta^{-1} \circ \zeta'$. Then $\tau -\tau' \in \kK(A'_1, A_1)$. As the composition of any two morphism in $\kK$ is zero we have $\f\circ(\tau-\tau') = 0$, for any $\f \in \kK(A_1,A_2)$, i.e. the isomorphism $\mu \colon \kK(A_1,A_2) \xrightarrow{\simeq} \kK(A'_1,A_2)$ does not depend on the choice of $\tau$. Similarly for two liftings $(A_2,\xi)$, $(A'_2,\xi')$ for $B_2$ one constructs a canonical isomorphism $\nu \colon \kK(A_1,A_2)\to \kK(A_1,A'_2)$. Given $f_0\colon B_0 \to B_1$, $f_2 \colon B_2\to B_3$ one can also check that the homomorphism $\kK(g_0\times g_2) \colon \kK(A_1,A_2) \to \kK(A_0, A_3)$ is independent of the choice of $g_0\in \Hom(A_0, A_1)$, $g_2\in \Hom(A_2,A_3)$ such that $q(g_0) =f_0$, $q(g_2)= f_2$.
	
	Therefore $\mM$ defined by $\mM(B_1, B_2):= \kK(A_1, A_2)$, for some fixed choice of liftings $(A_1,\zeta_1)$, $(A_2, \zeta_2)$ for any $B_1, B_2\in \bB$, is a $\bB$ bimodule.
	Isomorphisms $\mu$ and $\nu$ yield an isomorphism 
	$\kK(A'_1, A'_2)\xrightarrow{\simeq} \mM (q(A'_1), q(A'_2))$  for any pair $(A'_1, A'_2)$ of objects of $\aA$.
\end{proof}

Let $(Q,T)$ be a pair of objects in a $k$-linear, $\Ext^1$-finite exact category $\eE$. \emph{The universal extension} of $Q$ by $T$ is a unique up to isomorphism object $R$ that fits into a conflation:
\begin{align}\label{univ_extens}
T\otimes \Ext^1_{\eE}(Q, T)^\vee \to R \xrightarrow{\pi} Q
\end{align}  
given by the canonical element in $\Ext^1(Q, T \otimes \Ext^1(Q,T)^\vee)$ corresponding to $\Id_{\Ext^1(Q,T)}$ under the isomorphism $\Ext^1(Q, T \otimes \Ext^1(Q,T)^\vee)  \simeq \Ext^1(Q, T) \otimes \Ext^1(Q, T)^\vee$ .

The \emph{universal coextension} is a unique up to isomorphism object $U$ that fits into a conflation
\begin{equation*} 
T \to U \to Q \otimes \Ext^1(Q,T)
\end{equation*}
given by the canonical element in $\Ext^1(Q \otimes \Ext^1(Q,T), T)$ corresponding to $\Id_{\Ext^1(Q,T)}$ under the isomorphism $\Ext^1(Q \otimes \Ext^1(Q,T), T)\simeq \Ext^1(Q, T) \otimes \Ext^1(Q, T)^\vee$ .

Let $(\tT, \fF)$ be a perpendicular torsion pair in an exact category $\eE$ with functors $\defl_* \colon \fF \to \eE$, $\defl^* \colon \eE\to \fF$, $\infl_*\colon \tT\to \eE$, $\infl^!\colon \eE\to \tT$ as in Proposition \ref{prop_torsion=adjoint}. 

We define the \emph{$\tT$-projective subcategory} as a full subcategory of $\eE$ with objects

\begin{equation}\label{eqtn_categ_Y}
\begin{aligned} 
	\yY = \{E \in \eE\,|\, \forall T \in \tT\, \Ext^1_{\eE}(E,T) =0, \, 
	(-)\circ \eta_E \colon \Hom_{\eE}(\defl_*\defl^*E, T) \xrightarrow{\simeq} \Hom(E,T)\},
\end{aligned} 
\end{equation} 
where $\eta_E\colon E\to \defl_*\defl^*E $ is the adjunction unit.

We assume that $\tT \simeq k\textrm{-vect}$ and $T\in \tT$ is an indecomposable object. Then
the category  of $\tT$-projective objects is the full subcategory of $\eE$ whose objects are universal extensions of objects in $\fF$:

\begin{THM}\label{thm_Y_to_F_is_square_zero}
	Let $(\tT, \fF)$ be a perpendicular torsion pair in a $k$-linear, $\Ext^1$-finite exact category $\eE$ such that $\tT\simeq k\textrm{-vect}$. Let further $\yY\subset \eE$ be the $\tT$-projective subcategory. Then functor $q = \defl^*|_{\yY} \colon \yY\to \fF$ is a square-zero extension. Its kernel $\kK(F_1,F_2)$ is $\Hom_{\eE}(\defl_*F_1,T) \otimes \Ext^1_{\eE}(\defl_*F_2,T)^\vee$ with the natural $\fF$ bimodule structure  (see Lemma \ref{lem_kernel_of_square_zer_ext}).
\end{THM}
\begin{proof}
	First we check that $q$ is essentially surjective. Let $F$ be an object of $\fF$ and $Y_F$ the universal extension of $F$ by $T$. Then the conflation (\ref{univ_extens}) with $Q =F$ and $R = Y_F$ is the decomposition of $Y_F$ with respect to the torsion pair $(\tT, \fF)$. Applying $\Hom(-,T)$ yields an exact sequence (see (\ref{eqtn_6_term_ex_seq})):
	\begin{align*}
	0 \to \Hom_\eE(F,T) \to \Hom_\eE(Y_F,T) \to \Ext^1_\eE(F,T) \xrightarrow{\Id} \Ext^1_\eE(F,T) \to \Ext^1_\eE(Y_F,T) \to 0.
	\end{align*}
	It follows that $\Ext^1_\eE(Y_F, T) = 0$ and $\Hom_\eE(F,T) \simeq \Hom_\eE(Y_F,T)$, hence $Y_F\in \yY$. 
	
	Conversely, let 
	\begin{equation}\label{eqtn_decom_for_Y}
	T\otimes V \to Y \to \defl_*\defl^*Y
	\end{equation} 
	be a decomposition of an object $Y \in \yY$ with respect to the torsion pair (see Proposition \ref{prop_torsion=adjoint}). The conditions $\Ext^1_\eE(Y,T) = 0$ and $\Hom_\eE(\defl_*\defl^*Y,T) \simeq\Hom_\eE(Y,T)$ imply that the map $\Hom_\eE(T\otimes V,T) \simeq V^\vee \to \Ext^1_\eE(\defl_*\defl^*Y,T)$ obtained by applying $\Hom_\eE(-,T)$ to (\ref{eqtn_decom_for_Y}) is an isomorphism. 
	The element in $\Ext^1_\eE(\defl_*\defl^*Y, T\otimes V)$ corresponding to (\ref{eqtn_decom_for_Y}) is then the image of $\Id_{T\otimes V} $ under the morphism $\Hom_{\eE}(T\otimes V, T\otimes V)\to \Ext^1(\defl_*\defl^*Y, T \otimes V)$ (which we proved to be an isomorphism).  As $\Id_{T\otimes V}$ corresponds to $\Id_{V}$ under the isomorphism
	$$
\Hom_{\eE}(T\otimes V, T\otimes V) \simeq \Hom_{\eE}(T, T) \otimes \Hom_k(V, V) \simeq k \otimes \Hom_k(V,V) \simeq \Hom_k(V,V),
$$ 
(\ref{eqtn_decom_for_Y}) corresponds to the identity on $V\simeq \Ext^1_{\eE}(\defl_*\defl^*Y,T)^\vee$, i.e. to the canonical element in $\Ext^1_{\eE}(\defl_*\defl^*Y,T) \otimes \Ext^1_{\eE}(\defl_*\defl^*Y, T)^\vee$.
	This implies that $Y$ is isomorphic to the universal extension of $\defl_*\defl^*Y$ by $T$.
	
	Let $Y, Y'$ be objects of $\yY$ and $\f\colon Y' \to Y$ a morphism. 
	As discussed above $\infl^!(Y) \simeq T\otimes \Ext^1_\eE(\defl_*\defl^*Y,T)^\vee$, and analogously for $Y'$. Let $\alpha \colon \Ext^1_\eE(\defl_*\defl^*Y',T)^\vee \to \Ext^1_\eE(\defl_*\defl^*Y,T)^\vee$ be the transpose to $(-)\circ \defl_*\defl^* \f \colon \Ext^1_{\eE}(\defl_*\defl^*Y, T) \to \Ext^1_{\eE}(\defl_*\defl^*Y', T)$, i.e. $\alpha(\zeta(-)) = \zeta(- \circ \defl_*\defl^*\f)$.
	Then diagram 
	\[
\xymatrix{T \otimes \Ext^1_\eE(\defl_*\defl^*Y, T)^\vee \ar[r]& Y \ar[r]& \defl_*\defl^*Y\\
	T \otimes \Ext^1_\eE(\defl_*\defl^*Y', T)^\vee \ar[r] \ar[u]^{\Id_T\otimes \alpha} & Y' \ar[r] \ar[u]^{\f} & \defl_*\defl^*Y'\ar[u]^{\defl_*\defl^*\f} }
	\]
	commutes, i.e. $\Id_T\otimes \alpha= \infl^!(\f)$.

	If
	 $\defl^*\f$ is an isomorphism, then so is $\Id_T\otimes \alpha$. Then the five lemma \cite[Corollary 3.2]{Buehl} implies that $\f$ is an isomorphism, i.e. functor $q$ is conservative.
	
	Since $\Ext^1_\eE(Y',T) =0$, applying $\Hom_\eE(Y',-)$ to (\ref{eqtn_decom_for_Y}) yields a short exact sequence
	\begin{equation}\label{eqtn_some_ses}
	0 \to \Ext^1_\eE(\defl_*\defl^*F,T)^\vee \otimes \Hom_\eE(Y', T) \to \Hom_\eE(Y',Y) \xrightarrow{\Phi} \Hom_\eE(Y', \defl_*\defl^*Y) \to 0.
	\end{equation} 
	The composition of $\Phi$ with the adjunction isomorphism $\Hom_\eE(Y', \defl_*\defl^*Y) \simeq \Hom_{\fF}(\defl^*Y', \defl^*Y)$ is the map induced by $\defl^*$, hence $q$ is full. 
	
	It follows from (\ref{eqtn_some_ses}) that the kernel of $q$ is 
	\begin{equation}\label{eqtn_form_of_kern}
	\kK(Y',Y) =\Ext^1_\eE(\defl_*\defl^*Y,T)^\vee  \otimes \Hom_\eE(Y',T)  \simeq \Ext^1_\eE(\defl_*\defl^*Y,T)^\vee\otimes \Hom_\eE(\defl_*\defl^*Y', T),
	\end{equation} 
	where the isomorphism follows from the fact that $Y'\in \yY$.
	
	Note that, the right hand side of \eqref{eqtn_form_of_kern} has the $\yY$ bimodule structure; the right $\yY$ module structure comes from the isomorphism $\Ext^1_\eE(\defl_*\defl^*Y,T)^\vee\otimes \Hom_\eE(\defl_*\defl^*Y', T)\simeq \Hom_{\eE}(\defl_*\defl^*Y', T\otimes \Ext^1_\eE(\defl_*\defl^*Y,T)^\vee )$, hence it is given by $(\zeta \otimes f)  g = \zeta \otimes f\circ \defl_*\defl^*g$. The left $\yY$ module structure, induced by the right $\yY$ module structure of $\Ext^1_{\eE}$, reads $h(\zeta \otimes f) = \zeta(-\circ \defl_*\defl^*h) \otimes f$.
	
	We check that \eqref{eqtn_form_of_kern} is an isomorphism of $\yY$- bimodules. Let $\zeta \otimes f \in \Hom(\defl_*\defl^*Y',T \otimes \Ext_{\eE}^1(\defl_*\defl^*Y,T)^{\vee})$ be a morphism in $\kK(Y', Y)$ and let $g\colon Y_0 \to Y'$ 
		be a morphism in $\yY$. The composite $Y_0 \xrightarrow{g} Y' \xrightarrow{\eta_{Y'}} \defl_*\defl^*Y' \xrightarrow{\zeta \otimes f} T \otimes \Ext^1_{\eE}(\defl_*\defl^*Y,T)^\vee$ factors via $\defl_*\defl^*Y_0$ as $Y_0$ is an object in $\yY$. Hence, 
		$(\zeta \otimes f) \circ \eta_{Y'} \circ g = (\zeta \otimes f)\circ \defl_*\defl^*g \circ \eta_{Y_0} = (\zeta \otimes f\circ \defl_*\defl^*g )\circ \eta_{Y_0}$. It follows that $\zeta \otimes f \circ \defl_*\defl^*g$ is the corresponding element in $\kK(Y_0,Y)$. 
	
	Let now $h\colon Y \to Y_1$ be a morphism in $\yY$ and $\zeta \otimes f$ as above. It follows from \eqref{eqtn_some_ses} that $\zeta \otimes f\in \kK(Y',Y)$ should be considered as morphism $Y' \to T \otimes \Ext^1(\defl_*\defl^*Y, T)^\vee$. Then, the left $\yY$ module structure is given by the composition with $\infl^!(h)$. Hence, by the above discussion on $\infl^!\colon \yY\to \tT$, we have $h \circ(\zeta \otimes f) = \zeta(-\circ \defl_*\defl^*h) \otimes f$.
	
	Let $\f\colon Y_1 \to Y_2$ be in $\kK(Y_1,Y_2)$ and $\psi \colon Y_2 \to Y_3$ in $\kK(Y_2,Y_3)$. As $Y_2$ is an object in $\yY$,  any morphism $Y_2 \to T$ factors via $Y_2\to \defl_*\defl^*Y_2$.Hence, the composite $\psi \circ \f$ is 
	$$
	Y_1 \to T\otimes \Ext^1_\eE(\defl_*\defl^*Y_2,T)^\vee \to Y_2 \to \defl_*\defl^*Y_2 \to T \otimes \Ext^1_\eE(\defl_*\defl^*Y_3,T)^\vee \to Y_3.
	$$
	As $(\tT, \fF)$ is a torsion pair, there are no non-zero maps $T\otimes \Ext^1_\eE(\defl_*\defl^*Y_2,T)^\vee \to \defl_*\defl^*Y_2$, hence $\psi \circ \f =0$ which proves that $q$ is a square-zero extension.
	
	By Lemma \ref{lem_kernel_of_square_zer_ext} the kernel $\kK$ of $q$ is an $\fF$ bimodule. For $F\in \fF$ the universal extension $Y_F$ of $F$ by $T$ is a lifting of $F$ to $\yY$. Any lift of a morphism $f\colon F\to F'$ to a morphism $g\colon Y_F \to Y_{F'}$ of universal extensions fits into a commutative diagram:
	\[
	\xymatrix{T\otimes \Ext^1_\eE(\defl_*F',T)^\vee \ar[r] & Y_{F'} \ar[r] & \defl_* F'\\
	T\otimes \Ext^1_\eE(\defl_*F,T)^\vee \ar[r] \ar[u]^{\Id_T\otimes \alpha} & Y_{F} \ar[r] \ar[u]^g & \defl_* F\ar[u]^f} 
	\]
	with $\Id_T \otimes \alpha = \infl^!(g)$ induced by the composition with $f$ (see above).
	It follows from (\ref{eqtn_form_of_kern}) that
	$$
	\kK(F_1,F_2) \simeq \Hom_{\eE}(\defl_*F_1,T) \otimes \Ext^1_\eE(\defl_*F_2,T)^\vee. 
	$$
	As \eqref{eqtn_form_of_kern} is an isomorphism of $\yY$ bimodules, the above is an isomorphism of $\fF$ bimodules for the natural $\fF$ bimodule structure of the right hand side.
\end{proof}

\vspace{0.3cm}
\subsection{Strict admissible filtrations in exact categories}\label{ssec_str_adm_filt}~\\

We define the \emph{right admissible poset} $\textrm{rAdm}(\eE)$ of an exact category $\eE$ as the poset of right admissible subcategories with the inclusion order. Similarly, we consider the \emph{left admissible poset} $\textrm{lAdm}(\eE)$. Note that two elements in $\textrm{rAdm}(\eE)$ and in $\textrm{lAdm}(\eE)$ in general have neither union nor intersection.

Consider a finite lattice $\lL$ with the maximal element $1$ and the minimal element $0$.

\begin{DEF}\label{def_right_amd_filtr}
	A \emph{right admissible $\lL$-filtration} on $\eE$ is a map of posets $\lL \to \textrm{rAdm}(\eE)$, $I \mapsto \tT_I$, such that 
\begin{itemize}
	\item[(Ri)] $\tT_0 = 0$, $\tT_1 = \eE$,
	\item[(Rii)] for any $I, J \in \lL$, $\tT_{I \cap J} = \tT_I \cap \tT_J$, $\tT_{I\cup J}^\perp = \tT_I^\perp \cap \tT_J^\perp$.
\end{itemize}
\end{DEF}
\begin{DEF}\label{def_left_amd_filtr}
A \emph{left admissible $\lL$-filtration} on $\eE$ is a map of posets $\lL\to \textrm{lAdm}(\eE)$, $I \mapsto \fF_I$, such that
\begin{itemize}
	\item[(Li)] $\fF_0 = 0$, $\fF_1 = \eE$,
	\item[(Lii)] for any $I,J \in \lL$, $\fF_{I \cap J} = \fF_I \cap \fF_J$, ${}^\perp\fF_{I \cup J} = {}^\perp \fF_I \cap {}^\perp \fF_J$.
\end{itemize}
\end{DEF}

\begin{LEM}\label{lem_right_admis_in_subcat}
Let $\sS \subset \tT$ be a pair of right admissible subcategories in an exact category $\eE$. Then $\sS$ is right admissible in $\tT$. Similarly for left admissible subcategories.
\end{LEM}
\begin{proof}
Since $\sS$ is right admissible in $\eE$, for any object $T \in \tT$, there exists a conflation $S \to T \to T'$ with $S \in \sS$ and $T'\in \sS^\perp$, where the perpendicular is taken in $\eE$. Since $T'$ is the cokernel of an inflation in $\tT$ and $\tT$ is right admissible, $T' \in \tT$. 
\end{proof} 

The following definition for exact categories is a counterpart of the similar definition for triangulated categories in \cite{BodBon2} (see also subsection \ref{ssec_der_admiss}).

\begin{DEF}\label{def_strict_right_filtr} 
A \emph{right admissible} $\lL$-\emph{filtration} on $\eE$ is \emph{strict} if 
\begin{itemize}
	\item[(Riii)] for any $I, J \in \lL$, we have $\tT_{I \cup J} \cap \tT_{I}^\perp = \tT_J \cap \tT_{I\cap J}^\perp$ as full subcategories of $\eE$.
\end{itemize}
A \emph{left admissible} $\lL$-\emph{filtration} on $\eE$ is \emph{strict} if 
\begin{itemize}
	\item[(Liii)]  for any $I, J \in \lL$, we have $\fF_{I \cup J} \cap ( ^\perp\fF_{J}) = \fF_I \cap ( ^\perp\fF_{I\cap J})$ as full subcategories of $\eE$.	
\end{itemize}
\end{DEF}

The conditions of strictness are trivially satisfied when $I\preceq J$ or $J\preceq I$, hence they always hold for a full order $\lL$.

For a lattice $\lL$, denote by $\lL^{op}$ the lattice obtained by inversion of the order of elements in $\lL$. For $I\in \lL$, denote by $I^{o}$ the corresponding element of $\lL^{\opp}$. We have: $(I \cap J)^{o} = I^{o} \cup J^{o}$ and $(I \cup J)^{o}  = I^{o} \cap J^{o}$.

For a right admissible $\lL$-filtration $\{\tT_{I}\}$ define the \emph{right dual} $\lL^{op}$-\emph{filtration} $\{ \fF_{I^{o}}\}$ to consist of perpendicular categories: 
$$
\fF_{I^{o}}:= \tT_{I}^{\perp}
$$
In view of Proposition \ref{right-left}, this is an involutive duality between right admissible $\lL$-filtrations and left admissible $\lL^{\opp}$-filtrations in an exact category. 

\begin{LEM}\label{lem_dual_strict}
	A right admissible $\lL$-filtration on an exact category $\eE$ is strict if and only if its right dual $\lL^{\opp}$-filtration is strict.
\end{LEM}
\begin{proof}		
	This follows from equalities:
	\begin{align*}
	&\fF_{I^{o} \cup J^{o}} \cap (^\perp \fF_{J^{o}}) = \fF_{(I \cap J)^{o}} \cap (^\perp \fF_{J^{o}}) = \tT_{I\cap J}^\perp \cap \tT_{J} \\
	& \fF_{I^{o}} \cap (^\perp \fF_{I^{o} \cap J^{o}}) =\fF_{I^{o}} \cap ({}^\perp\fF_{(I \cup J)^{o}})  =\tT_I^\perp \cap \tT_{I \cup J}.
	\end{align*}
\end{proof}

\begin{PROP}\label{prop_direct_sum}
	A right admissible $\lL$-filtration $\{\tT_I\}$ on an exact category $\eE$ is strict if and only if, for any pair $I,J\in \lL$, the category $\xX = \tT_{I\cup J} \cap \tT_{I\cap J}^\perp$  is a direct sum of mutually perpendicular subcategories $\xX_J=(\tT_I \cap \tT_{I\cap J}^\perp)$ and $\xX_I=(\tT_J\cap \tT_{I \cap J}^\perp)$.
\end{PROP}
\begin{proof}
	Assume that the $\lL$-filtration is strict. Lemma \ref{lem_right_admis_in_subcat} implies existence of perpendicular torsion pairs $(\tT_I, \xX_I)$, $(\tT_J, \xX_J)$, $(\tT_{I\cap J}, \xX)$ in $\tT_{I\cup J}$, $(\tT_{I \cap J}, \xX_J)$ in $\tT_I$, and $(\tT_{I\cap J}, \xX_I)$ in $\tT_J$. 
	As $\xX_J \subset \tT_I$, and $\xX_I \subset \tT_J$ 
categories $\xX_I$ and $\xX_J$ are mutually perpendicular, i.e. $\Hom_{\eE}(X_I, X_J)$ and $\Hom_{\eE}(X_J, X_I)$ vanish for any $X_I \in \xX_I$ and $X_J \in \xX_J$.

Consider $X \in \xX$. Decomposition (\ref{decomposition}) with respect to torsion pair $(\tT_I, \xX_I)$ is a conflation 
\begin{equation}\label{conftixi}
T_I \to X \to X_I
\end{equation}
with $T_I\in \tT_I$ and $X_I\in \xX_I$. Denote by $\infl^!_K$ the right adjoint to the inclusion functor $\tT_K\ \to \eE$, for any $K\in \lL$. Functor $\infl_{I\cap J}^!$ vanishes on $X$ in view of torsion pair $(\tT_{I\cap J}, \xX)$. Moreover, 
$\infl_{I\cap J}^!(X) = \infl_{I\cap J}^! \circ \infl_{I}^!(X) = \infl_{I\cap J}^!(T_I)=0$. Therefore, in view of the torsion pair $(\tT_{I\cap J}, \xX_J)$ in $\tT_I$, we have: $T_I\in \xX_J$. Then perpendicularity of $\xX_I$ and $\xX_J$ implies that conflation (\ref{conftixi}) splits, hence $\xX$ is a direct sum of $\xX_I$ and $\xX_J$.
	
Conversely, assume that $\xX= \xX_I \oplus \xX_J$, and $\xX_I$ and $\xX_J$ are mutually perpendicular. We need to show that $\xX_I= \tT_{I\cup J}\cap \tT^{\perp}_I$.
By Lemma \ref{lem_right_admis_in_subcat}, $\tT_{I \cap J}$ is admissible in $\tT_I$. Then in view of definition of $\xX_J$, category $\tT_I$ admits a perpendicular torsion pair $(\tT_{I \cap J},\xX_J)$. Since both components of the torsion pair are (left) perpendicular to $\xX_I$, we have: $\xX_I\subset \tT_I^{\perp}$. The other way around, any $X\in \tT_{I\cup J}\cap \tT^{\perp}_I$ is in $\tT_{I \cap J}^{\perp}$, hence $X\in \xX=\xX_I\oplus \xX_J$. As $X$ is right perpendicular to $\tT_I$ its $\xX_J$-component should be trivial, hence $X\in \xX_I$.
\end{proof}

\vspace{0.3cm}
\subsection{Weakly idempotent split categories and torsion pairs}\label{ssec_weak_idm_split_and_tor_p}~\\

Recall \cite{ThoTro}, \cite{Buehl} that a  category $\eE$ is \emph{weakly idempotent split} if any retraction $r \colon E \to E'$, i.e. a morphism admitting $s\colon E' \to E$ such that $rs = \Id_{E'}$ has a kernel. Equivalently, any co-retraction $c\colon E \to E''$, i.e. a morphism admitting $s\colon E''\to E$ with $sc = \Id_E$, has a cokernel. Note that $\eE$ is weakly idempotent split if and only if $\eE^{\opp}$ is.

\begin{LEM}\cite[Corollary 7.7]{Buehl}\label{lem_Buehler_weakly_idem_comp}
	An exact category $\eE$ is weakly idempotent split if and only if it has the following property: if $f\colon A \to B$ is a deflation and $g\colon B \to C$ a morphism such that $gf\colon A \to B$ is a delfation, then $g$ is a deflation. 
\end{LEM}

We will show that an exact category possessing a perpendicular torsion pair $(\tT, \fF)$ is weakly idempotent split if and only if $\tT$ and $\fF$ are.

\begin{LEM}\label{lem_defl_in_the_middle}
	Consider a diagram in an exact category $\eE$ with conflations as rows:
	\begin{equation}\label{eqtn_diag_2}
	\xymatrix{C' \ar[r] & C \ar[r] & C''\\
	B' \ar[r] \ar[u]^{g'} & B \ar[r] \ar[u]^g & B'' \ar[u]^{g''}}
	\end{equation}
	 If $g'$ and $g''$ are deflations, so is $g$.
\end{LEM}
\begin{proof}
	Consider $\eE$ as a fully exact subcategory in an abelian category $\aA$. The snake lemma for the diagram (\ref{eqtn_diag_2}) considered as a diagram in $\aA$ yields a short exact sequence $0 \to A' \to A \to A'' \to 0$ of the kernels of the vertical maps. Since $g'$ and $g''$ are deflations, $A', A'' \in \eE$. As $\eE \subset \aA$ is closed under extensions, $A \in \eE$. The snake lemma implies that cokernel of $g$ in $\aA$ is trivial. We conclude that $A \to B \xrightarrow{g} C$ is a conflation in $\eE$, as the embedding $\eE\to \aA$ reflects exactness.
\end{proof}

\begin{PROP}\label{prop_idem_split_iff_T_and_F}
	Consider an exact category $\eE$ with a perpendicular torsion pair $(\tT, \fF)$. Then $\eE$ is weakly idempotent split if and only if $\tT$ and $\fF$ are.
\end{PROP}
\begin{proof}
	Subcategory $\tT$ is closed under cokernels and $\fF$ is closed under kernels. Hence, if $\eE$ is weakly idempotent split, then so are $\tT$ and $\fF$.

	Now let $\tT$ and $\fF$ be weakly idempotent split. Let $g\colon B \to C$ be a morphism and $f\colon A \to B$ a deflation such that the composite $gf$ is a deflation. As adjoint-to-embedding functors $\infl^!: \eE \to \tT$ and $\defl^*:\eE \to \fF$ are exact, morphisms $\infl^!(f)$, $\infl^!(gf)$, $\defl^*(f)$, $\defl^*(gf)$ are deflations. Hence, by Lemma \ref{lem_Buehler_weakly_idem_comp}, $\infl^!g$ and $\defl^*g$ are deflations.  Then $g$ is a deflation by Lemma \ref{lem_defl_in_the_middle}. We conclude by Lemma \ref{lem_Buehler_weakly_idem_comp} that $\eE$ is weakly idempotent split. 
\end{proof}

\vspace{0.3cm}
\subsection{Deriving admissibility}\label{ssec_der_admiss}~\\

Given an exact category $\eE$, its bounded (resp. bounded above, bounded below, unbounded) derived category $\dD^b(\eE)$, (resp. $\dD^{-}(\eE)$, $\dD^+(\eE)$, $\dD(\eE)$) is the Verdier quotient of the homotopy category $\kK^b(\eE)$ (resp. $\kK^-(\eE)$, $\kK^+(\eE)$, $\kK(\eE)$) of bounded (bounded above, bounded below, unbounded) complexes in $\eE$ by the thick subcategory $\aA^b(\eE)$, (resp. $\aA^{-}(\eE)$, $\aA^+(\eE)$, $\aA(\eE)$) generated by acyclic complexes \cite{ThoTro},\cite{Nee1}. If $\eE$ is a weakly idempotent split exact category, then $\aA^b(\eE)$, $\aA^-(\eE)$, $\aA^+(\eE)$ are already thick subcategories ((partial) boundedness of complexes is important here, $\aA(\eE)$ is thick if $\eE$ is idempotent split, see \cite{ Nee1}).
The derived categories of exact categories are triangulated.

A full triangulated subcategory $\dD_0 \subset \dD$ in a triangulated category $\dD$ is said to be \emph{right admissible} \cite{B} if the inclusion functor $\iota_{0*} \colon \dD_0 \to \dD$ admits a right adjoint $\iota_0^! \colon \dD\to \dD_0$. Equivalently, $\dD$ admits a semi-orthogonal decomposition (or SOD) $\dD= \langle \dD_0^{\perp_0}, \dD_0\rangle$, see \cite[Lemma 3.1]{B}. Similarly, $\dD_0 \subset \dD$ is \emph{left admissible} if the inclusion functor admits a left adjoint $i_0^* \colon \dD \to \dD_0$. Then $\dD$ admits an SOD $\dD= \langle \dD_0, {}^{\perp_0} \dD_0\rangle$. 

\begin{THM}\label{thm_SOD}
	Let $\defl_* \colon\fF \to \eE$ be a left admissible subcategory of a weakly idempotent split exact category $\eE$. Then $\dD^b(\fF)$ is a left admissible subcategory of $\dD^b(\eE)$ and $\dD^b(\eE)$ admits a semi-orthogonal decomposition $\dD^b(\eE) = \langle \dD^b(\fF), \dD^b({}^\perp \fF)\rangle$.
\end{THM}
\begin{proof}
Since $\defl_*$ is exact, its derived functor $R\defl_*$ exists.
	\cite[Theorem 12.1.b]{Kel5} states that $R\defl_*$ is fully faithful if for any conflation $\defl_* F' \to E \to E''$ in $\eE$ there exists a conflation $F' \to F \to F''$ in $\fF$ and a commutative diagram with identical the left-most vertical morphism:
	\[
	\xymatrix{\defl_* F' \ar[r] & \defl_*F \ar[r] & \defl_*F''\\
	\defl_*F' \ar[r] \ar[u]& E \ar[r] \ar[u] & E'' \ar[u]}
	\]
	By Lemma \ref{lem_j*_of_a_conf} the conflation $F' \to \defl^*E \to \defl^* E''$ fits, hence $R\defl_*$ is fully faithful.
	
	Functor $\defl^*$ is exact by Theorem \ref{thm_admissible_iff_adjoint_exact}, hence its derived functor $L\defl^* \colon \dD^b(\eE) \to \dD^b(\fF)$ exists and is left adjoint to $R\defl_* \colon \dD^b(\fF) \to \dD^b(\eE)$ (\textit{cf.} \cite[Lemma 13.6]{Kel5}). Therefore, $\dD^b(\fF)$ is a left admissible in $\dD^b(\eE)$.
	
	Category ${}^{\perp_0}\dD^b(\fF)$ is the kernel of $L\defl^*$. It remains to show that if  $E^\bcdot$ is a complex in $\dD^b(\eE)$ such that $L\defl^*E^\bcdot$ is acyclic then $E^\bcdot$ is quasi-isomorphic to a complex in $\dD^b({}^\perp \fF)$.
	
	Let $E^\bcdot$ be any complex in $\dD^b(\eE)$. The decomposition (\ref{decomposition}) with respect to the torsion pair $({}^\perp \fF, \fF)$ applied term-wise to $E^\bcdot$ yields a sequence $\infl_*\infl^!E^\bcdot \xrightarrow{\varepsilon^\bcdot} E^\bcdot \xrightarrow{\eta^\bcdot} \defl_*\defl^*E^\bcdot$ of complexes over $\eE$. As $\infl_*\infl^!E^n \xrightarrow{\varepsilon^n} E^n \xrightarrow{\eta^n}\defl_*\defl^*E^n$ is a conflation for every $n$, \cite[Lemma 11.6]{Kel5} implies that $\infl_*\infl^!E^\bcdot \to E^\bcdot \to \defl_*\defl^*E^\bcdot \to \infl_*\infl^!E^\bcdot[1]$ is an exact triangle in $\dD^b(\eE)$.
	Since functors $\defl_*$ and $\defl^*$ are exact, $\defl_*\defl^*E^\bcdot \simeq R\defl_*L\defl^*(E^\bcdot)$. If $E^\bcdot$ lies in the kernel of $L\defl^*$ complex $\defl_*\defl^*E^\bcdot$ is acyclic. Therefore, $\varepsilon^\bcdot$ is a quasi-isomorphism, i.e. $E^\bcdot \simeq \infl_*\infl^*E^\bcdot \in \dD^b({}^\perp \fF)$. 
\end{proof}

Next we will show that a strict right/left admissible filtration on an exact category induces a strict right/left admissible filtration on its derived category as a triangulated category.

Recall the definitions after \cite{BodBon2}.
A triangulated subcategory $\dD_0$ in a triangulated category  $\dD$  is \emph{admissible} if it is both right and left admissible, if and only if $\dD$ admits SODs: $\dD=\langle \dD_0^{\perp_0}, \dD_0\rangle = \langle \dD_0, {}^{\perp_0} \dD_0 \rangle$. This condition implies existence of a recollement \cite{BBD}:
\begin{equation}\label{eqtn_recoll}
\xymatrix{\dD_0 \ar[rr]|{\iota_{0*}}&& \dD \ar@<-2ex>[ll]|{\iota_0^*} \ar@<2ex>[ll]|{\iota_0^!} \ar[rr]|{j^*} && \dD_1 \ar@<-2ex>[ll]|{j_!} \ar@<2ex>[ll]|{j_*}} 
\end{equation}
where $\dD_1 = \dD/\dD_0$, $j^*$ is the quotient functor, $j_!$ embeds $\dD_1$ as ${}^{\perp_0} \dD_0$ while $j_*$ embeds $\dD_1$ as $\dD_0^{\perp_0}$. Conversly, for any recollement (\ref{eqtn_recoll}), category $\dD_0$ is admissible. We say that (\ref{eqtn_recoll}) is a \emph{recollement with respect to subcategory} $\dD_0$.

Given \tr es $(\dD_0^{\leq 0},\dD_0^{\geq 1})$ on $\dD_0$ and $(\dD_1^{\leq 0},\dD_1^{\geq 1})$ on $\dD_1$, there exists a unique \tr e $(\dD^{\leq 0},\dD^{\geq 1})$ on $\dD$ such that functors $i_{0*}$ and $j^*$ are $t$-exact,  i.e. $i_* \dD_0^{\leq 0} \subset \dD^{\leq 0}$, $i_* \dD_0^{\geq 1} \subset \dD^{\geq 1}$, $j^* \dD^{\leq 0} \subset \dD_1^{\leq 0}$, and $j^* \dD^{\geq 1} \subset \dD_1^{\geq 1}$, see \cite{BBD}. The \tr e $(\dD^{\leq 0},\dD^{\geq 1})$ is said to be \emph{glued via the recollement (\ref{eqtn_recoll})}.

We define the \emph{right admissible poset} $\textrm{rAdm}(\dD)$ of a triangulated category $\dD$ as the poset of right admissible subcategories with the inclusion order. Similarly, we consider the \emph{left admissible poset} $\textrm{lAdm}(\dD)$ and the \emph{admissible poset} $\textrm{Adm}(\dD)$. 

For a finite lattice $\lL$ with the maximal element $1$ and the minimal element $0$ a \emph{right admissible $\lL$-filtration} on $\dD$ is a map of posets $\lL \to \textrm{rAdm}(\dD)$, $I \mapsto \dD_I$, such that 
\begin{itemize}
	\item[(Rti)] $\dD_0 = 0$, $\dD_1 = \dD$,
	\item[(Rtii)] for any $I, J \in \lL$, $\dD_{I \cap J} = \dD_I \cap \dD_J$, $\dD_{I\cup J}^{\perp_0} = \dD_I^{\perp_0} \cap \dD_J^{\perp_0}$.
\end{itemize}
A \emph{left admissible $\lL$-filtration} on $\dD$ is a map of posets $\lL\to \textrm{lAdm}(\dD)$, $I \mapsto \dD_I$, such that
\begin{itemize}
	\item[(Lti)] $\dD_0 = 0$, $\dD_1 = \dD$,
	\item[(Ltii)] for any $I,J \in \lL$, $\dD_{I \cap J} = \dD_I \cap \dD_J$, ${}^{\perp_0}\dD_{I \cup J} = {}^{\perp_0} \dD_I \cap {}^{\perp_0} \dD_J$.
\end{itemize}
An \emph{admissible $\lL$-filtration} on $\dD$ is a map of posets $\lL \to \textrm{Adm}(\dD)$ which defines both left and right admissible $\lL$-filtration on $\dD$. 

Denote by $\aA / \bB$ the Verdier quotient. We say that a (right or left) admissible $\lL$-filtration on $\dD$ is \emph{strict} if 
\begin{itemize}
	\item[(iii)] for any $I, J \in \lL$, $\Hom_{\dD_{I\cup J}/\dD_{I \cap J}}(\dD_I/\dD_{I \cap J}, \dD_J/\dD_{I \cap J}) =0$.
\end{itemize}

\begin{THM}\label{thm_srict_filt_on_derived}
	Consider a finite lattice $\lL$ and a strict right admissible $\lL$-filtration $\{\tT_I\}$ on a weakly idempotent split exact category $\eE$. Then $\{\dD^b(\tT_I)\}$ is a strict right admissible $\lL$-filtration on $\dD^b(\eE)$.
\end{THM}
\begin{proof}
	Consider $I \in \lL$. Theorem \ref{thm_SOD} implies that $\dD^b(\tT_I)\subset \dD^b(\eE)$ is right admissible. 
Since condition (Rti) is obviously satisfied, it suffices to check conditions (Rtii) and (iii). We note that 
	condition (iii) follows immediately from the direct sum decomposition $\tT_{I\cap J}^\perp \cap \tT_{I\cup J} = (\tT_{I\cap J}^\perp \cap \tT_I) \oplus (\tT_{I\cap J}^\perp \cap \tT_J)$ of Proposition \ref{prop_direct_sum}.
	
	Let $I, J$ be any pair of elements of $\lL$. Lemma  \ref{lem_right_admis_in_subcat} and Proposition \ref{prop_direct_sum} imply the existence of the perpendicular torsion pairs $(\tT_I, \xX_I)$, $(\tT_J,\xX_J)$, $(\tT_{I\cap J}, \xX_I\oplus \xX_J)$ in $\tT_{I\cup J}$, $(\tT_{I \cap J}, \xX_J)$ in $\tT_I$, and $ (\tT_{I \cap J}, \xX_I)$ in $\tT_J$. 
	
 	Clearly $\dD^b(\tT_{I\cap J}) \subset \dD^b(\tT_I) \cap \dD^b(\tT_J)$. Let now $X$ be an object of $\dD^b(\tT_I) \cap \dD^b(\tT_J)$. The decompositions of $X$ with respect to SOD's $\dD^b(\tT_I) = \langle \dD^b(\xX_J), \dD^b(\tT_{I\cap J}) \rangle$, $\dD^b(\tT_J) = \langle \dD^b(\xX_I), \dD^b(\tT_{I\cap J})\rangle$ (given by Theorem \ref{thm_SOD}) yield isomorphic exact triangles $T_{IJ} \to X \to Y$, $T'_{IJ} \to X \to Y'$, because morphisms $T_{IJ} \to X$, $T'_{IJ} \to X$ are the adjunction counit for the embedding $\dD^b(\tT_{I\cap J}) \to \dD^b(\tT_I) \cap \dD^b(\tT_J)$. Then 
$Y\simeq Y' \in \dD^b(\xX_I) \cap \dD^b(\xX_J)$=0. 
Hence $X \simeq T_{IJ} \in \dD^b(\tT_{I \cap J})$, which proves the first part of (Rtii).
 	
 	By Lemma \ref{lem_dual_strict} the right dual $\lL^{\opp}$-filtration on $\eE$ is strict. Then the same argument for the opposite category $\eE^{op}$ implies that $\dD^b(\fF_{I^o}) \cap \dD^b(\fF_{J^o}) = \dD^b(\fF_{I^o \cap J^o})$, where $\fF_{K^o} := \tT_K^\perp$.
By Theorem \ref{thm_SOD}, $\dD^b(\fF_{K^o})={}^{\perp_0}\dD(\tT_K )$ and ${I^o \cap J^o}=(I\cup J)^o$, hence the second part of (Rtii) follows.
\end{proof}

\section{Thin categories}\label{sec_thin_cat}
In this section we introduce the basic notion for our approach to highest weight categories, thin exact categories. We define a canonical partial order on the set of isomorphism classes of irreducible objects and construct a strict admissible filtration related to the poset. Then we find projective generators in thin categories. This is based on the construction of 'universal extensions' of standarizable collections by Dlab and Ringel \cite{DR}. 

\vspace{0.3cm}
\subsection{Canonical strict filtrations on thin categories}\label{ssec_can_str_filt_on_thin}~\\

Let $k$ be an algebraically closed field and $k\rm{-vect}$ the category of finite dimensional $k$-vector spaces.
\begin{DEF}
We say that a $k$-linear Hom and $\Ext^1$-finite \emph{exact category} $\eE$ is \emph{thin} if it has a right admissible filtration $0=\tT_0\subset \tT_1\subset \ldots \subset \tT_n = \eE$ with graded factors $\eE_i:=\tT_{i-1}^{\perp}\cap \tT_{i} $ equivalent to
$k\rm{-vect}$, for all $i\in [1, n]$. 
\end{DEF}
Any filtration with above properties in a thin category is called \emph{thin filtration}.

The notion of thin filtration, hence of thin category, is self-dual: if we transfer to the opposite category, then $\tT_i^{\perp}$ make a right admissible filtration with graded factors $k\rm{-vect}$.
\begin{LEM}\label{lem_thin_weakly_idem_split}
	A thin category $\eE$ is weakly idempotent split.
\end{LEM}
\begin{proof}
	Since $k\rm{-vect}$ is idempotent split, the statement follows from Proposition \ref{prop_idem_split_iff_T_and_F}.
\end{proof}

Object $E$ in an exact category is called \emph{irreducible} if any inflation $A\to E$ is either zero or an isomorphism.
\begin{LEM}\label{lem_irr_in_thin}
Let $\eE$ be a thin category and $E_i\in \eE$ the object corresponding to 
$k$ under the equivalence $\eE_i\simeq k\textrm{-vect}$ for a thin filtration. Then $\Lambda =\{E_1,\ldots, E_n\}$ is the set of isomorphism classes of irreducible objects in $\eE$.
\end{LEM}
\begin{proof}
This follows easily by induction on the length of a thin filtration in view of decompositions (\ref{decomposition}) for right admissible subcategories in the filtration.
\end{proof}

Recall \cite{B} that an ordered collection  $( E_n, \ldots, E_1 )$ is a \emph{full exceptional sequence} in a $k$-linear triangulated category $\tT$ if
\begin{itemize}
	\item $\Hom_{\tT}(E_i, E_i[l])$ equals $k$ if $l =0$ and vanishes otherwise,
	\item $\Hom_{\tT}(E_i, E_j[l])$ vanish for $i$ less than $j$ and any $l \in \mathbb{Z}$,
	\item $\tT$ is equivalent to its smallest triangulated subcategory containing $E_n, \ldots, E_1$.
\end{itemize} 

\begin{PROP}\label{prop_full_exc_coll_in_D_thin}
	Let $\{E_1,\ldots,E_n\}$ be the ordered collection  of representatives of  isomorphisms classes of irreducible objects in a thin category $\eE$. Then $( E_n,\ldots,E_1 )$ is a full exceptional sequence in $\dD^b(\eE)$.
\end{PROP}
\begin{proof}
	Follows from Theorem \ref{thm_srict_filt_on_derived}.
\end{proof}

 We endow the set $\Lambda$ of isomorphism classes of irreducible objects in a thin category with a \emph{canonical poset structure}. Namely, we consider the minimal partial order such that 
\begin{equation}\label{def-poset}
{\rm if}\
E_i \notin E_j^{\perp},\ {\rm for\ some}\ i,j\in \Lambda,
{\rm then}\ i\preceq j. 
\end{equation}
The existence of a thin filtration implies no cycle 
$(i_1, i_2, \dots , i_k=i_1)$, such that $E_{i_s}\notin E_{i_t}^{\perp}$ for $i_s\le i_t$.
Therefore, the partial order is well-defined. The obvious full order on $\Lambda$ induced by any thin filtration is a linearization of the canonical partial order.
We write $E(\la)$ for the irreducible object of $\eE$ corresponding to $\la \in \Lambda$.

If $\lL$ is a finite distributive lattice then by Birkhoff's theorem \cite{Birk}, $\lL$ is isomorphic to the lattice of lower ideals $\iI(\Lambda)$ in a finite poset $\Lambda$. A \emph{lower (upper) ideal} in a poset $\Lambda$ is a subset $I\subset \Lambda$ (resp.  $U\subset \Lambda$) which together with every $\la\in I$ (resp. $\la\in U$) contains all $\la'\in \Lambda$ such that $\la'\preceq \la$ (resp. $\la'\succeq \la$). 
The poset $\Lambda$ in Birkhoff's theorem is identified with the subposet of join-prime elements in $\lL$. An element $s\in \lL$ is \emph{join-prime} if the fact that $s\preceq J_1\cup J_2$, for some $J_1,J_2\in \lL$, implies that $s\preceq J_1$ or $s\preceq J_2$. 

For a finite poset $\Lambda$ and  $\lambda \in \Lambda$, the corresponding join-prime element in the distributive lattice $\lL =\iI(\Lambda)$
is the principal lower ideal $I_{\la}\subset \Lambda$, i.e. the ideal of all elements $\mu \in \Lambda$, such that $\mu \preceq \la $. Then $I_{<\la} :=I_{\la}\setminus \{\la\}$ is an ideal in $I_{\la}$ which contains all ideals $J\subsetneqq I_{\la}$.

Let $\{\tT_I\}$ be a right admissible $\iI(\Lambda)$-filtration. To $\la \in \Lambda$ we assign the subcategory 
$$
\eE_{\la} := \tT_{I_{<\la}}^\perp \cap \tT_{I_{\la}},
$$
the $\la$-graded factor of the filtration. By Lemma \ref{lem_right_admis_in_subcat}, $\eE_{\la}$ is left admissible in $\tT_{I_{\la}}$. Similarly, given a left admissible $\iI(\Lambda)$-filtration $\{\fF_I\}$, we define the $\la$-graded factor by:
$$
\eE_{\la} := {}^\perp \fF_{I_{<\lambda}} \cap \fF_{I_{\la}}.
$$

For a class of objects $\{X_i\}$ in an exact category $\eE$, we denote by $\fF(\{X_i\})$ the smallest fully exact subcategory in $\eE$ that contains all $X_i$'s.

\begin{PROP}\label{prop_strict_filtr_on_thin}
	Let $\eE$ be a thin category and $\Lambda$ the canonical poset of isomorphism classes of irreducible objects in $\eE$. Then $\tT_I:=\fF(\{E({\la})\}_{\la\in I})$, where $I$ runs over the set of lower ideals in $\Lambda$,
defines a strict right admissible $\iI(\Lambda)$-filtration on $\eE$.
\end{PROP}
\begin{proof} Category $\fF_I :=\fF(\{E( \la)\}_{\la \notin I}\})$ belongs to $\tT_I^{\perp}$ and any thin filtration on $\eE$ gives a filtration on every object of $\eE$ with all graded factors sums of irreducible objects, i.e. either in $\tT_I$ or in $\fF_I$. By Lemma (\ref{lem_filtration_torsion}), $\tT_I$ is right admissible and $\fF_I=\tT_I^{\perp}$.

	Subcategory $\tT_I \cap \tT_J$ is contained in the left perpendicular to both $\fF_I$ and $\fF_J$, hence, by the definition of $\fF_{I\cap J}$, it is contained in the left perpendicular to $\fF_{I\cap J}$, i.e. in $\tT_{I\cap J}$. Since the opposite inclusion is obvious, $\tT_{I\cap J} = \tT_I \cap \tT_J$ and, similarly, $\fF_{I \cup J} = \fF_I \cap \fF_J$, i.e. $I\mapsto \tT_I$ defines a right admissible filtration. 
	
	Finally, as $\fF(\{E( \la )\}_{\la \in J\setminus I})$ is the right perpendicular category to $\tT_{I\cap J}$ in $\tT_J$ and to $\tT_I$ in $\tT_{I \cup J}$, for any lower ideals $I\subset J$, the filtration is strict. 
\end{proof}

\begin{REM}\label{rem_thin_not_alg_cl}
	In the case when the field $k$ is not algebraically closed, we generalize the definition of a thin category as follows: we say that a $k$-linear Hom and $\Ext^1$-finite \emph{exact category} $\eE$ is \emph{thin} if it has a right admissible filtration $0=\tT_0\subset \tT_1\subset \ldots \subset \tT_n = \eE$ such that, for all $i\in [1,n]$, the graded factor $\eE_i:=\tT_{i-1}^{\perp}\cap \tT_{i} $ is equivalent to $\textrm{mod-}\Gamma_i$, for a finite dimensional division $k$-algebra $\Gamma_i$. The construction of the partially ordered set $\Lambda$ and the $\iI(\Lambda)$-filtration on $\eE$ are basically the same.	
\end{REM}

\vspace{0.3cm}
\subsection{Projective generators in thin categories}\label{ssec_proj_gen_in_thin}~\\

An object $P \in \eE$ in an exact category $\eE$ is \emph{projective} if functor $\Hom_{\eE}(P, -)\colon \eE\to \aA b$ is exact. It is equivalent to requiring $\Ext^1(P, -)=0$. A full additive subcategory $\pP \subset \eE$ is \emph{projectively generating} if all objects of $\pP$ are projective and, for any object $E\in \eE$, there exists a deflation $P \to E$ with $P \in \pP$. 

For a collection of objects $\{X_i\}_{i\in I}\subset \eE$, denote by $\textrm{add }\{X_i\}\subset \eE$ the full subcategory of finite direct sums of copies of $X_i$'s. 
A set of projective object $\{P_i\}_{i\in I}\subset \eE$ is a \emph{set of projective generators} if  $\textrm{add }\{P_i\}\subset \eE$ is projectively generating. If the set consists of one element $P$, we say that $P$ is a \emph{projective generator}. A \emph{set of injective generators} is a set of objects in $\eE$ which is a set of projective generators in the opposite category $\eE^{op}$.

The following proposition allows us to construct a projective generator in an exact category, given an admissible subcategory isomorphic to $k$-vect and a projective generator in its perpendicular.
\begin{PROP}\label{prop_univ_ext_in_proj}
	Let $\tT\simeq k-{\rm vect}$ be a right admissible subcategory of $\eE$. Let $Q$ be a projective object in $ \tT^{\perp}$. Then the universal extension $R$ of $Q$ by the irreducible object $T$ in $\tT$, is projective in $\eE$. If $\{Q_{i}\}$ is a set of projective generators in $\tT^{\perp}$, then $T\cup \{R_i\}$ is a set of projective generators in $\eE$, where $R_i$ is the universal extension of $Q_i$ by $T$.
\end{PROP}
\begin{proof} 
	By applying the second of the sequences (\ref{eqtn_6_term_ex_seq}) to the conflation (\ref{univ_extens}) and object $F\in \tT^{\perp}$ we get: $\Ext^1(R,F)=0$.
	
	Now applying (\ref{eqtn_6_term_ex_seq}) to the same conflation and object $T^{\oplus a}\in \tT$ gives an exact sequence:
	\begin{equation}\label{four-term}
	\Ext^1(Q,T)\otimes\Hom(T,T^{\oplus a})\to \Ext^1(Q,T^{\oplus a})\to \Ext^1(R,T^{\oplus a})\to 0
	\end{equation}
	Since the first morphism is an isomorphism, $\Ext^1(R,T^{\oplus a})$ vanishes.

	Now any object $E\in \eE$ has a conflation of the form 
	\begin{equation}\label{taef}
	T^{\oplus a}\to E\to F 
	\end{equation}
	with $T^{\oplus a}\in \tT$ and $F\in \tT^{\perp}$.
	 Applying $\Ext^1(R, -)$ to this conflation gives that $\Ext^1(R,E)=0$, for any $E\in \eE$, i.e. $R$ is projective in $\eE$.
	
	Applying $\Ext^1(T,-)$ to the same conflation implies that $T$ is projective in $\eE$.
	
	Now let $\{Q_{i}\}$ be a set of projective generators in $\tT^{\perp}$. Consider conflation (\ref{taef}) for an arbitrary object $E\in \eE$. We have a deflation $\oplus Q_i^{\oplus b_i}\to F$. Since $R_i$ is a universal extension of $Q_i$, we have a deflation $R_i\to Q_i$. As $\oplus R_i^{\oplus b_i}$ is a projective object in $\eE$, the composite deflation $\oplus R_i^{\oplus b_i}\to \oplus Q_i^{\oplus b_i}\to F$ lifts along the deflation $E\to F$ to a morphism $\oplus R_i^{\oplus b_i}\to E$. We get a diagram:
	\[
	\xymatrix{T^{\oplus a} \ar[r] & E \ar[r] & F\\
	T^{\oplus a}\ar[r] \ar[u]^{\Id}& T^{\oplus a}\oplus  R_i^{\oplus b_i} \ar[r] \ar[u] & \bigoplus R_i^{\oplus b_i}\ar[u]}
	\]
	that satisfies the condition of Lemma \ref{lem_defl_in_the_middle}. Hence the morphism $T^{\oplus a}\oplus R_i^{\oplus b_i}\to E$ is a deflation.
\end{proof}

\begin{PROP}\label{prop_proj_in_thin}
	Let $\eE$ be a thin category with the canonical poset $\Lambda$. There exist pairwise non-isomorphic, unique (up to a non-unique isomorphism) projective objects $\{P(\la)\}_{\la \in \Lambda}\subset\eE$ with local endomorphism rings which admit deflations $d_{\la} \colon P(\la)\to E(\la)$. The kernel of $d_{\la}$ is in $\tT_{I_{<\la}}$ 
	and the subcategory $\pP := \fF (\{P(\la)\}_{\la \in \Lambda})$ is projectively generating.  	
\end{PROP}
\begin{proof}
	We construct pairwise non-isomorphic $P(\la)$ by induction on $|\Lambda|$. The case $|\Lambda|=1$ is clear. Consider $\la \in \Lambda$. If $\la$ is minimal then $P(\la) = E(\la)$. Otherwise, let $\nu\in \Lambda$ be a minimal element. Category $\fF:=\fF(\{ E(\la )\}_{\la \neq \nu})$ is thin with canonical poset $\Lambda \setminus \{\nu\}$, hence by inductive hypothesis, there exists $Q(\la )\in \fF$ with the required properties. By Proposition \ref{prop_strict_filtr_on_thin}, category $\eE$ admits a perpendicular torsion pair $(\tT, \fF)$ where $\tT = \fF(E (\nu ))$. 
	
	We define $P(\la)$ as the universal extension of $Q(\la)$ by $E(\nu)$. By Proposition \ref{prop_univ_ext_in_proj} object $P(\la)$ is projective in $\eE$. 
		Let $\yY \subset \eE$ be the category of $\tT$-projective objects. By Theorem \ref{thm_Y_to_F_is_square_zero} functor $\defl^* = q\colon \yY \to \fF$ is a square-zero extension. 
		As it was shown in the beginning of the proof of this theorem, $P(\la) \in \yY$ and $q(P(\la)) = Q(\la)$. Hence, the kernel of the surjective morphism $\End(P(\la)) \to \End(Q(\la))$ induced by $q$ is a square-zero ideal.
As $\End(Q(\la))$ is local, so is $\End(P(\la))$. 
	By inductive hypothesis $Q(\la)$ admits a deflation $Q(\la) \to E(\la)$ with kernel in $\tT_{I_{<\la}} \cap \fF$.
	Then the composite of deflations $P(\la) \to Q(\la) \to E(\la)$ is a deflation whose kernel lies in $\tT_{I_{< \la}}$. This implies $\Hom(P(\la), E(\la)) \simeq \End(E(\la))$.

	For $\la \notin I_{\mu}$, $\Hom(P(\mu), E(\la)) =0$, as $E(\la) \in \fF_{I_\mu}$, $P(\mu)\in \tT_{I_{\mu}}$. Hence, if $P(\la) \simeq P(\mu)$, deflations $P(\la) \to E(\mu)$, $P(\mu) \to E(\la)$ imply that $\mu \in I_{\la}$ and $\la \in I_{\mu}$, i.e. $\la = \mu$.

	Proposition \ref{prop_univ_ext_in_proj} implies by induction   that $\pP$ is a projectively generating subcategory.
	
	Now we prove the uniqueness of $P(\la)$. Let $P(\la)$ and $P'(\la)$ be two 
	projective objects 
	with local endomorphism rings
	and $d_{\la}\colon P(\la)\to E(\la)$, $d'_{\la}\colon P'(\la)\to E(\la)$ deflations. Morphism $d_{\la}$ admits a lift to $\alpha \colon P(\la) \to P'(\la)$. Similarly, we have $\beta\colon P'(\la) \to P(\la)$. The composite $ x =\beta \circ \alpha \in \End(P(\la))$ is an element whose image under the surjective map $\End(P(\la)) \to \Hom(P(\la),E(\la)) \simeq \End(E(\la)) \simeq k$ is the identity morphism. 
	As $\End(P(\la))$ is a local ring, 
	$x$ is invertible. Similarly, $\alpha \circ \beta$ is invertible, which implies that $\alpha $ is an isomorphism.
\end{proof}

\begin{REM}\label{rem_injective_in_thin}
	Since the opposite of a thin category is thin, Proposition \ref{prop_proj_in_thin} implies that objects $E(\la)$  admit also injective hulls $I(\la)$ and  $\iI:=\fF (\{I(\la)\}_{\la \in \Lambda})\subset \eE$ is an injectively generating subcategory.
\end{REM}

\begin{REM}\label{rem_proj_in_thin_non_alg_cl}
	Theorem \ref{thm_Y_to_F_is_square_zero} remains true for a torsion pair $(\tT, \fF)$ with $\tT \simeq \textrm{mod-}\Gamma$, for a division $k$-algebra $\Gamma$. The current proof works if the tensor product over $k$ in the definition of the universal extension (\ref{univ_extens}) and in Theorem \ref{thm_Y_to_F_is_square_zero} 
	is replaced by the tensor product over $\Gamma$. 
	Hence, thin categories over arbitrary fields (see Remark \ref{rem_thin_not_alg_cl}) admit projective covers $P(\lambda)$ of $E(\lambda)$ with local endomorphisms algebras.
\end{REM}

\section{Abelian envelopes of exact categories}\label{sec_abel_env}

We shall introduce the left and right abelian envelopes of an exact category and discuss their basic properties. We construct such envelopes for exact categories with enough projectives and injectives. Under further finiteness conditions we prove the derived equivalence of an exact category with its envelopes.
Also, we discuss the structures which admissible subcategories and strict filtrations in exact categories induce on envelopes, and apply these techniques to  thin categories.

For additive categories $\cC$ and  $\bB$ with $\cC$ essentially small, we denote by $\Fun(\cC, \bB)$ the category of additive functors $\cC \to \bB$ with natural transformations as morphisms.

\vspace{0.3cm}
\subsection{Definition of the right and  left envelope
}\label{ssec_first_prop}~\\

There is a meaningful notion of righ/left exact functors between exact categories. We introduce a bit of their theory in the form required for our purposes.

Let $\eE$ and $\eE'$ be exact categories. We define an additive functor $\Psi \colon \eE\to \eE'$ to be \emph{right exact} if for any conflation $X \xrightarrow{i} Y \xrightarrow{d} Z$ in $\eE$ morphism $\Psi(d)$ is a deflation with kernel $i'\colon X' \to \Psi(Y)$ and $\Psi(i) = i' \circ d'$, for some deflation $d'\colon \Psi(X) \to X'$.  A \emph{left exact} functor is defined similarly. We denote by $\textrm{Rex}(\eE, \eE')$  and $\Lex(\eE, \eE')$ the categories of right/left exact functors $\eE \to \eE'$ with natural transformations as morphisms.

\begin{LEM}\label{lem_comp_of_righ_ex_is_right_ex}
	The composite of two right exact functors is right exact, and similarly for left exact functors.
\end{LEM}
\begin{proof}
	Let $\Psi \colon \eE\to \eE'$, $\Phi \colon \eE' \to \eE''$ be right exact functors and $X\xrightarrow{i}Y \xrightarrow{d}Z$ a conflation in $\eE$. Let $i'$ be the kernel of $\Psi(d)$ and $\Psi(i) = i' \circ d'$. As $\Psi(d)$ and $d'$ are deflations, so are $\Phi \Psi(d)$ and $\Phi(d')$. Let $i''$ be the kernel of $\Phi \Psi(d)$. The right exactness of $\Phi$ implies $\Phi(i') = i'' \circ d''$, for some deflation $d''$. Then $\Phi \Psi(i) = \Phi(i') \circ \Phi(d') = i'' \circ d'' \circ \Phi(d')$ is the composite of $i''$ with the deflation $d'' \circ \Phi(d')$, i.e. $\Phi \Psi$ is right exact.
\end{proof}

Note that, for an abelian category $\aA$, an additive functor $\Psi\colon \eE\to \aA$ is right exact if for any conflation (\ref{eqtn_confl}) the sequence $\Psi(X) \to \Psi(Y) \to \Psi(Z) \to 0$ is exact in $\aA$.

The category $\Lex(\eE^{\opp},\Ab)$ of contravariant left exact functors  with values in abelian groups is the quotient of the category $\textrm{Fun}(\eE^{\opp}, \Ab)$ of additive functors by the Serre subcategory $\textrm{Efc}(\eE^{\opp}, \Ab)$ of effaceable functors (\textit{cf.} \cite{Gabriel, Kel4}), hence it is abelian. The composite 
\begin{equation}\label{eqtn_Yoneda_for_E}
h\colon \eE\to \textrm{Fun}(\eE^{\opp},\Ab) \to \Lex(\eE^{\opp}, \Ab)
\end{equation} 
of the Yoneda embedding with the quotient functor 
induces an equivalence of $\eE$ with a fully exact subcategory of $\Lex(\eE^{\opp}, \Ab)$. Actually,
 $h$ (\ref{eqtn_Yoneda_for_E}) is the embedding functor of the Gabriel-Quillen theorem that we mentioned in Section \ref{ssec_prelim}.

\begin{DEF}\label{def_rig_ab_env}
	A \emph{right abelian envelope} for an exact category $\eE$ is an abelian category $\aA_r(\eE)$ together with a right exact functor $i_R \colon \eE \to \aA_r(\eE)$ which, for any abelian category $\bB$, yields an equivalence of categories $(-)\circ i_R\colon \Rex(\aA_r(\eE), \bB) \xrightarrow{\simeq} \Rex(\eE, \bB)$. 
\end{DEF}

If we choose a quasi-inverse to this equivalence then every right exact functor $F \colon \eE \to \bB$ has a unique, up to a canonical functorial isomorphism, lifting to a functor $\tilde{F} \colon \aA_r(\eE) \to \bB$ together with a canonical isomorphism $\tilde{F} \circ i_R \simeq F$:
\begin{equation}\label{triangle}
\xymatrix{\eE \ar[dr]_{F} \ar[r]^{i_R} & \aA_r(\eE) \ar[d]^{\tilde{F}} \\ & \bB}
\end{equation}
Similarly, a \emph{left abelian envelope} of $\eE$ is an abelian category $\aA_l(\eE)$ together with a left exact functor $i_L\colon \eE\to \aA_l(\eE)$ which, for any abelian category $\bB$, induces an equivalence $\Lex(\aA_l(\eE), \bB) \xrightarrow{\simeq} \Lex(\eE, \bB)$.

It easily follows from the definition that if a right (left) abelian envelope $\aA_r(\eE)$ ($\aA_l(\eE)$) exists for a given exact category $\eE$ then it is unique up to a canonical equivalence, i.e. given two right abelian envelopes $\aA_r(\eE)$ and $\aA'_r(\eE)$ of an exact category $\eE$, there is a unique, up to a unique isomorphism of functors, equivalence  $\Phi \colon \aA_r(\eE) \to \aA'_r(\eE)$.

We say that $\eE$ is an exact category \emph{with the right/left envelope} if $\aA_r(\eE)$/$\aA_l(\eE)$ exists. 

The opposite to an exact (abelian) category is again an exact (abelian) category. The opposite to the right exact functor is left exact. Then transfer to the opposite categories and functors in (\ref{triangle}) implies that if $\aA_r(\eE)$ exists then so does $\aA_l(\eE^{\opp})$ and
\begin{equation}\label{rem_envs_for_opp}
	 \aA_l(\eE^{\opp})\simeq \aA_r(\eE)^{\opp}.
\end{equation}

The envelopes do not exist in general, see Example \ref{exm_no_env} below. We will discuss their existence under some conditions later in this section. 
Historically, the first instance 
when right and left abelian envelopes exist are 
quasi-abelian categories. 

Let $\cC$ be an additive category with kernels and cokernels. A morphism $f\colon C\to C'$ in $\cC$ is \emph{strict} if the canonical map $\textrm{Coim }f \to \textrm{Im }f$ is an isomorphism. Category $\cC$ is \emph{quasi-abelian} if a pull back of a strict epimorphism is a strict epimorphism and a pushout of a strict monomorphism is a strict monomorphism. A quasi-abelian category $\cC$ carries an intrinsic exact category structure with strict monomorphisms as inflations and strict epimorphisms as deflations.

With \cite[Proposition 1.2.33]{Schne} J-P. Schneiders constructed the right  abelian envelope of a quasi-abelian category, which he denoted by $\lL\hH(\cC)$. He proved that the corresponding functor $I\colon \cC \to \lL\hH(\cC)$ is fully faithful, exact, and reflects exactness \cite[Corollary 1.2.27]{Schne}. Moreover, $I$ induces an equivalence of unbounded derived categories $\dD(\cC) \xrightarrow{\simeq} \dD(\lL\hH(\cC))$  \cite[Proposition 1.2.31]{Schne}.

It make sense to compare the left/right envelopes with the exact abelian hull introduced by M. Adelaman. 

In \cite{Adel}, M.~Adelman  defined and proved the existence of an \emph{exact abelian hull} of an exact category $\eE$,
i.e. an abelian category $\mathcal{U}(\eE)$ together with an 
exact functor $i_e\colon \eE \to \mathcal{U}(\eE)$ which induces an equivalence of the category of exact functors $\mathcal{U}(\eE) \to \bB$ with the category of exact functors $\eE\to \bB$, for any abelian category $\bB$, \textit{cf.} \cite{Stein}. 
Moreover, $i_e \colon \eE\to \aA_e(\eE)$ is fully faithful and reflects exactness, see \cite[Theorem 4.4.4]{Stein}. It would be interesting to investigate the relation of the exact abelian hull of $\eE$ with its abelian envelopes.

The rest of this subsection was added after 
	 the reviewer's comments on possible relation of Rump's recent construction of left/right quotient categories, which appeared at about the same time as the first version of our paper.

 Recall that an additive category $\aA$ is \emph{left abelian} \cite{Rump2} if it has cokernels and for any $f\colon A\to B$ and $g\colon D\to B$ in $\aA$ with $c = \textrm{cok }f$ and $cg=0$, there is a cokernel $d\colon E\to D$ and a morphism $e\colon E\to A$ such that $gd = fe$:
	\[
	\xymatrix{E\ar@{-->}[r]^d \ar@{-->}[d]_e & D\ar[r] \ar[d]^g & 0 & \\
	A \ar[r]^f & B \ar[r]^c & C \ar[r] & 0.}
	\]

	In \cite{Rump2}  Rump introduced the left/right quotient categories, respectively $Q_l(\eE)$ and $Q_r(\eE)$, for any left/right exact category $\eE$.  
	For a left exact category $\eE$, the category $Q_l(\eE)$ is defined as the quotient of the left abelian category $\textrm{fp}(\eE)$ of finitely presented functors $\eE^{\opp}\to \Ab$ by the Serre subcategory of \emph{defects}, i.e. functors $F\colon \eE^{\opp} \to \Ab$ which admit a presentation $\Hom(-,Y) \to \Hom(-, Z) \to F(-) \to 0$, with a deflation $Y\to Z$.
	The category $Q_l(\eE)$ is itself left abelian. The Yoneda functor $\eE\to \textrm{fp}(\eE)$ composed with the quotient functor gives a fully faithful exact functor $i \colon \eE\to Q_l(\eE)$, \cite[Proposition 3.5]{Rump2}. Category $Q_l(\eE)$
	  has the following universal property for right exact functors from $\eE$:
	\begin{PROP}\cite[Corollary 3.1]{Rump2}\label{prop_Rump}
		Let $\eE$ be a left exact category. For any left abelian category $\aA$, the composition with $i$ induces an equivalence of the category of additive functors $Q_l(\eE) \to \aA$ which respect cokernels of morphisms and right exact functors $\eE \to \aA$.
	\end{PROP}
	Any exact category $\eE$ is both left and right exact. The maximal exact structures on the quotient categories $Q_l(\eE)$, $Q_r(\eE)$ consist of all kernel-cokernel pairs.  Rump introduced also a weaker exact structure on them, denoted by $Q_l^\vee(\eE)$ and $Q_r^\vee(\eE)$, and proved that the exact abelian hull $\uU(\eE)$ of $\eE$ is equivalent to 
	 $$
	 \uU(\eE) \simeq Q_r(Q_l^\vee(\eE)) \simeq Q_l(Q_r^\vee(\eE)).
	 $$

\begin{THM}
	If $Q_l(\eE)$ is abelian then $\aA_r(\eE)$ exists and is equivalent to $Q_l(\eE)$.
\end{THM}
\begin{proof}
	We need to check that if $Q_l(\eE)$ is abelian it has the universal property of the right abelian envelope. It is well-known that a functor between abelian categories is right exact if and only if it respects cokernels (note that it is not so when categories are only left abelian).
	Hence, 
	according to Proposition \ref{prop_Rump}, we have an equivalence $\Rex(Q_l(\eE), \aA) \xrightarrow{\simeq}\Rex(\eE, \aA)$, for any abelian category $\aA$.
\end{proof}

\vspace{0.3cm}
\subsection{Faithfulness and fully faithfulness of the functor $ \eE\to \aA_r(\eE)$}~\\

Let $\eE$ be an essentially small exact category with the right envelope $\aA_r(\eE)$. Since the category $\Ab$ of abelian groups is abelian, the composition with $i_R$ yields an equivalence
\begin{equation}\label{eqtn_equiv_of_Lex}
\Lex(\aA_r(\eE)^{\opp}, \Ab) \simeq \Rex(\aA_r(\eE), \Ab^{\opp}) \simeq \Rex(\eE, \Ab^{\opp}) \simeq \Lex(\eE^{\opp},\Ab).
\end{equation}
The composite of the Yoneda embedding $h_{\aA} \colon \aA_r(\eE) \to \Lex(\aA_r(\eE)^{\opp}, \Ab)$ with the  above equivalence is the fully faithful exact functor
\begin{align}\label{eqtn_func_Xi}
&\Xi\colon \aA_r(\eE) \to \Lex(\eE^{\opp}, \Ab),& &\Xi(A)(-) = \Hom_{\aA_r(\eE)}(i_R(-), A).&
\end{align}
Recall after \cite[Definition 6.3.3]{KasSch2} that an object $C$ in a category $\cC$ with small filtrant colimits is \emph{compact} if for any $\alpha \colon I \to \cC$ with $I$ small and filtrant the natural morphism $\varinjlim \Hom_{\cC}(C, \alpha) \to \Hom_{\cC}(C, \varinjlim \alpha)$ is an isomorphism. We denote by $\cC^c$ the full subcategory of compact objects in $\cC$.

\begin{LEM}\label{lem_envelope_fully_exact_in_Lex}
	Let $\eE$ be an essentially small exact category with the right abelian envelope. Then functor $i_R \colon \eE\to \aA_r(\eE)$ is faithful. Moreover, $\Xi$ identifies $\aA_r(\eE)$ with the subcategory $\Lex(\eE^{\opp}, \aA b)^c $ of compact objects in $\Lex(\eE^{\opp}, \aA b)$.
\end{LEM}
\begin{proof}
	Functor $h$ in (\ref{eqtn_Yoneda_for_E})
	is exact, in particular right exact, hence there exists right exact functor $q\colon \aA_r(\eE) \to \Lex(\eE^{\opp}, \aA b)$ such that $q\circ i_R = h$. As $h$ is fully faithful, functor $i_R$ is faithful.
	
	Functor $\Xi\colon \aA_r(\eE) \to \Lex(\eE^{\opp}, \Ab)$ in (\ref{eqtn_func_Xi}) is exact and fully faithful. As category $\aA_r(\eE)$ is abelian, $\Lex(\aA_r(\eE)^{\opp}, \Ab)$ is equivalent to the category $\textrm{Ind}(\aA_r(\eE))$ of ind-objects over $\aA_r(\eE)$, see \cite[Corollary 8.6.3]{KasSch2}. The category $\aA_r(\eE)$ is idempotent complete, hence it is equivalent to the category of compact objects in $\textrm{Ind}(\aA_r(\eE))\simeq \Lex(\aA_r(\eE)^{\opp}, \Ab)$, see \cite[Exercise 6.1]{KasSch2}.	
\end{proof}

\begin{REM}\label{rem_not_yoneda}
	By Lemma \ref{lem_envelope_fully_exact_in_Lex},  if $\aA_r(\eE)$ exists then it is equivalent to the full subcategory of compact objects in $\Lex(\eE^{\opp}, \Ab)$. However, it is not clear if the composite of the universal functor $i_R\colon \eE\to \aA_r(\eE)$ with the canonical embedding $\Xi\colon \aA_r(\eE) \simeq \Lex(\eE^{\opp}, \Ab)^c \to \Lex(\eE^{\opp}, \Ab)$ is the Yoneda embedding  $h$ (\ref{eqtn_Yoneda_for_E}).
\end{REM}

The following example of an exact category without the right abelian envelope was communicated to us by A. Efimov.
\begin{EXM}\label{exm_no_env}
	Let $R$ be a non-coherent ring and $\eE$ the category of finitely presented projective right $R$-modules. Then $\eE$ does not have the right abelian envelope. Indeed, by Lemma \ref{lem_envelope_fully_exact_in_Lex}, if $\aA_r(\eE)$ exists it is equivalent to the category of compact objects in $\textrm{Mod-}R$, i.e. to the category of finitely presented right $R$- modules  (see \cite[Exercise 6.8]{KasSch2}) which, by the assumption on $R$, is not abelian.
\end{EXM} 

By Lemma \ref{lem_envelope_fully_exact_in_Lex}, the functor $i_R\colon \eE\to \aA_r(\eE)$ is always faithful.  It would be interesting to know whether functor it is also always full.

\begin{LEM}\label{lem_if_I_R_full_then_exact}
	Let $\eE$ be an essentially small exact category. If the right abelian envelope $\aA_r(\eE)$ exists and $i_R \colon \eE\to \aA_r(\eE)$ is full then $i_R$ 
	induces an equivalence of $\eE$ with a fully exact subcategory of $\aA_r(\eE)$.
\end{LEM}
\begin{proof}
	Since $i_R$ is full, it is fully faithful by Lemma \ref{lem_envelope_fully_exact_in_Lex}. Then formula (\ref{eqtn_func_Xi}) implies that the composite of $\Xi$ and $i_R$ is the Yoneda embedding $h$ (\ref{eqtn_Yoneda_for_E}): $\Xi\circ i_R\simeq h$. 
	As $i_R$ is right exact, $\Xi$ and $h$ are exact, and $\Xi$ is fully faithful, functor $i_R$ is exact. Moreover, since by Gabriel-Quillen theorem $h$
	is an inclusion of a fully exact subcategory and $\Xi$ is exact and fully faithful, $i_R$ induces an equivalence of $\eE$ with a fully exact subcategory of $\aA_r(\eE)$.
\end{proof}

\vspace{0.3cm}
\subsection{The monad associated to the right abelian envelope}\label{ssec_monad}~\\

We consider the 2-category $\mathscr{E}x_r$ whose objects are exact categories admitting  right abelian envelopes. The 1-morphisms in $\mathscr{E}x_r$ are right exact functors and 2-morphisms are natural transformations.

\begin{PROP}\label{prop_envelope_functorial}
	The 2-category $\mathscr{E}x_r$ admits a pseudo-endofunctor $\aA_r\colon \mathscr{E}x_r \to \mathscr{E}x_r$ and a pseudo-natural transformation $\mathscr{I}\colon \Id_{\mathscr{E}x_r} \to \aA_r$ such that, for any $\eE\in \mathscr{E}x_r$, the functor $\mathscr{I}_\eE \colon \eE\to \aA_r(\eE)$ is the right abelian envelope of $\eE$. 
\end{PROP}
\begin{proof}
	We shall define the functor $\aA_r$ and the natural transformation $\mathscr{I}$. For $\eE\in \mathscr{E}x_r$, $\aA_r(\eE)$ is the right abelian envelope of $\eE$ and $\mathscr{I}_\eE :=i_R\colon \eE\to \aA_r(\eE)$.
	
	Let $\Phi \in \Hom_{\mathscr{E}x_r}(\eE,\eE')$. By Lemma \ref{lem_comp_of_righ_ex_is_right_ex}, the composite $\mathscr{I}_{\eE'}\circ \Phi\colon \eE\to \aA_r(\eE')$ is right exact, hence there exists $\Psi\in \Hom_{\mathscr{E}x_r}(\aA_r(\eE), \aA_r(\eE'))$, unique up to a canonical functorial isomorphism, such that $\Psi \circ \mathscr{I}_{\eE} \simeq \mathscr{I}_{\eE'} \circ \Phi$. We put $\aA_r(\Phi):= \Psi$. Clearly, $\mathscr{I}$ is a pseudo-natural transformation.
 
 	A 2-morphism in $\mathscr{E}x_r$, i.e. a natural transformation $\tau \colon \Phi_1 \to \Phi_2$ of right exact functors $\eE\to \eE'$, yields a natural transformation $\mathscr{I}_{\eE'} \circ \tau \colon \mathscr{I}_{\eE'} \circ \Phi_1 \to \mathscr{I}_{\eE'} \circ \Phi_2$ of functors in $\Rex(\eE, \aA_r(\eE'))$. Equivalence $\Rex(\eE, \aA_r(\eE')) \simeq \Rex(\aA_r(\eE), \aA_r(\eE'))$ implies existence of a unique $\sigma \colon \Psi_1\to \Psi_2$, where $\Psi_i \circ \mathscr{I}_{\eE} \simeq \mathscr{I}_{\eE'} \circ \Phi_i$, for $i=1,2$:
 	\[
 	\xymatrix{\aA_r(\eE) \ar@<1ex>[r]^{\Psi_1} \ar[r]_{\Psi_2} & \aA_r(\eE')\\
 		\eE \ar[u]^{\mathscr{I}_{\eE}} \ar[r]^{\Phi_1} \ar@<-1ex>[r]_{\Phi_2} & \eE' \ar[u]_{\mathscr{I}_{\eE'}}}
 	\] We put $\aA_r(\tau) := \sigma$.
 	
 	The uniqueness of $\Psi\colon \aA_r(\eE) \to \aA_r(\eE')$ such that $\Psi \circ \mathscr{I}_{\eE} \simeq \mathscr{I}_{\eE'} \circ \Phi$ implies canonical isomorphisms $\aA_r(\Id) \simeq \Id$ and $\aA_r(\Phi' \circ \Phi)\simeq \aA_r(\Phi') \circ \aA_r(\Phi)$, for $\Phi \in \Hom_{\mathscr{E}x_r}(\eE, \eE')$, $\Phi' \in \Hom_{\mathscr{E}x_r}(\eE', \eE'')$. It follows that $\aA_r$ is indeed a pseudo-functor.
\end{proof}

	Recall that a \emph{monad} on a category $\cC$ is an endofunctor $F\colon \cC \to \cC$ together with natural transformations $\eta \colon \Id_{\cC} \to F$, $\mu \colon F^2 \to F$ which satisfy the standard conditions.

	An \emph{algebra over monad} $(F, \eta, \mu)$ is a pair $(C,T_C)$ of an object $C\in \cC$ and $T_C \in \Hom_{\cC}(F(C), C)$ which fit into the standard commutative diagrams. 
	A morphism of algebras $(C,T_C) \to (C', T_{C'})$  is $\f \colon C \to C'$ such that $T_{C'} \circ F(\f) =\f \circ T_C$. We denote by $\cC^F$ the (Eilenberg-Moore) category of algebras over $F$ and algebra homomorphisms.
	If $\cC$ is a 2-category, all the equalities should be replaced by isomorphisms.

	We denote by $\mathscr{A}b_r\subset \mathscr{E}x_r$ the full 2-subcategory whose objects are abelian categories endowed with the exact structure in which conflations are the short exact sequences.
	
		If $\aA\in \mathscr{E}x_r$ is abelian, then $\Id_{\aA} \colon \aA \to \aA$ is its right abelian envelope. In particular, $\mathscr{I}_{\aA_r(\eE)} \colon \aA_r(\eE) \xrightarrow{\simeq} \aA_r^2(\eE)$ is an equivalence, for any $\eE\in \mathscr{E}x_r$. We define isomorphism 
	$$
	\mu:=(\mathscr{I}|_{\aA_r})^{-1}\colon \aA_r^2 \xrightarrow{\simeq} \aA_r.
	$$

\begin{PROP}\label{prop_mondad_algebras_over}
	The pseudo-functor $\aA_r$ together with the pseudo-natural transformations $\mathscr{I}\colon \Id_{\mathscr{E}x_r} \to \aA_r$, $\mu \colon \aA_r^2\to \aA_r$ defines a monad $(\aA_r, \mathscr{I}, \mu)$ on the category $\mathscr{E}x_r$. 
	The subcategory $\mathscr{A}b_r\subset \mathscr{E}x_r$ is equivalent to the category $\mathscr{E}x_r^{\aA_r}$ of algebras over this monad.
\end{PROP}
\begin{proof}
	It is straightforward to verify that $(\aA_r, \mathscr{I}, \mu)$ is a monad.
	
	For $\aA\in \mathscr{A}b_r$, define $T_\aA\colon \aA_r(\aA) \to \aA$ as the quasi-inverse of $\mathscr{I}_{\aA} \colon \aA\xrightarrow{\simeq} \aA_r(\aA)$. Then the pair $(\aA, T_\aA)$ is an algebra over $(\aA_r, \mathscr{I}, \mu)$ and, for $\aA$, $\aA'$ in $\mathscr{A}b_r$, any $\Phi \in \Hom_{\mathscr{A}b_r}(\aA, \aA')$ satisfies $T_{\aA'} \circ \aA_r(\Phi) \simeq \Phi \circ T_{\aA}$, i.e. $\mathscr{A}b_r$ is a subcategory of $\mathscr{E}x_r^{\aA_r}$. 
	
	We check that the inclusion $\mathscr{A}b_r \to \mathscr{E}x_r^{\aA_r}$ is an equivalence. By definition of an algebra over a monad, for $(\eE, T_{\eE}\colon \aA_r(\eE) \to \eE)\in \mathscr{E}x_r^{\aA_r}$, the composite $\eE \xrightarrow{\mathscr{I}_\eE} \aA_r(\eE) \xrightarrow{T_{\eE}}\eE$ is isomorphic to $\Id_{\eE}$. It follows that $\mathscr{I}_\eE \circ T_\eE\circ \mathscr{I}_\eE \simeq \mathscr{I}_{\eE} \simeq \Id_{\aA_r(\eE)} \circ \mathscr{I}_{\eE}$. As the composition with $\mathscr{I}_\eE$ yields an equivalence $\Rex(\aA_r(\eE), \aA_r(\eE)) \simeq \Rex(\eE, \aA_r(\eE))$, we conclude that $\mathscr{I}_\eE \circ T_\eE\simeq \Id_{\aA_r(\eE)}$, i.e. $T_\eE$ is an equivalence. In particular, $\eE \simeq \aA_r(\eE) \in \mathscr{A}b_r$ is abelian.
\end{proof}

\vspace{0.3cm}
\subsection{Abelian categories as envelopes for their exact subcategories}~\\

The following theorem, based on a result from \cite{KasSch2}, is an useful source of examples of abelian envelopes.

\begin{THM}\label{thm_if_quotient_then_envelope}
	Let $\aA$ be an abelian category and $\eE\subset \aA$ a fully exact subcategory closed under kernels of epimorphisms. If every object of $\aA$ is a quotient of an object of $\eE$ then $\aA$ is the right abelian envelope for $\eE$.
\end{THM}
\begin{proof}
	Let $\bB$ be an abelian category.
	According to \cite[Theorem 8.7.2]{KasSch2}, if $\eE$ is a full subcategory of an abelian category $\aA$ and  every object of $\aA$ is a quotient of an object of $\eE$, then the embedding functor $\eE\to \aA$ induces an equivalence of  $\Rex(\aA, \bB)$ with the category of functors $\eE \to \bB$ which take any one-sided exact sequence $E' \to E \xrightarrow{d} E'' \to 0$ with all terms in $\eE$ to a one-sided exact sequence in $\bB$. We shall check that the latter category is equivalent to $\Rex(\eE, \bB)$. By definition, then $\aA =\aA_r(\eE )$. 
	
	Let $\Phi \colon  \eE\to \bB$ be a right exact functor. By assumption the kernel $\wt{E}$ of $d$ is an object of $\eE$. Moreover, the surjective map $E' \to \wt{E}$ has a kernel $E_1$ which is also an object of $\eE$. The right exactness of $\Phi$ implies that sequences $\Phi(E_1) \to \Phi(E') \to \Phi(\wt{E}) \to 0$, $\Phi(\wt{E}) \to \Phi(E) \to \Phi(E'') \to 0$ are exact. Hence so is $\Phi(E') \to \Phi(E) \to \Phi(E'') \to 0$. 
\end{proof}

 This theorem has a nice application to the category of vector bundles. Recall that a noetherian scheme $X$ has a \emph{resolution property} if every coherent sheaf on $X$ is a quotient of a locally free sheaf. 
\begin{COR}\label{cor_Coh_as_env}
	Let $X$ be a noetherian scheme with a resolution property. Then the category $\Coh(X)$ of coherent sheaves is the right abelian envelope for its fully exact subcategory $\textrm{Bun}(X)$ of locally free sheaves.
\end{COR}
\begin{proof}
	The subcategory $\textrm{Bun}(X) \subset \Coh(X)$ is closed under extensions and kernels of epimorphisms. By assumption, every $F\in \Coh(X)$ is a quotient of some $E\in \textrm{Bun}(X)$. We conclude by Theorem \ref{thm_if_quotient_then_envelope}.
\end{proof}

\vspace{0.3cm}
\subsection{Abelian envelopes for additive categories with weak (co)kernels}\label{ssec_ab_env_for_add}~\\

For an additive category $\cC$ the category $\textrm{fp}(\cC)$ of \emph{finitely presented contravariant functors} is the full subcategory of $\Fun(\cC^{\opp}, \aA b)$ whose objects are functors $\Phi$ admitting a presentation $\Hom_{\cC}(-,C_1) \to \Hom_{\cC}(-,C_2) \to \Phi \to 0$. The Yoneda embedding $\cC \to \Fun(\cC^{\opp}, \aA b)$ restricts to a fully faithful functor $h \colon \cC \to \textrm{fp}(\cC)$, $C\mapsto h^C(-) = \Hom_{\cC}(-,C)$.

We say that an additive category $\cC$ is \emph{additively generated} by objects $C_1,\ldots,C_n$, if any $C\in \cC$ is isomorphic to a finite direct sum of copies of $C_i$. 
\begin{LEM}\label{lem_fp_as_modules}
	Let $\cC$ be a category additively generated by $\{C_i\}_{i=1}^n$. Then $\textrm{fp}(\cC)$ is equivalent to the category of finitely presented right modules over the ring $A:=\End_{\cC}(\bigoplus_{i=1}^nC_i)$.
\end{LEM}
\begin{proof}
	Consider the full $\Ab$-enriched subcategory $\cC_0 \subset \cC$ on objects $C_1,\ldots, C_n$, i.e. the subcategory together with the structure of abelian groups on $\Hom$-spaces. 	By \cite[Exercise VIII.2.6(a)]{MacLane2}, there exists a universal additive category $\textrm{Add}(\cC_0)$.
	Functor $\f\colon \textrm{Add}(\cC_0) \to \cC$, extending the embedding $\cC_0 \to \cC$, is essentially surjective and fully faithful, because $\Hom$-groups are canonically extended to direct sums. Therefore, $\f$ is an equivalence. 
	The statement follows from the equivalences $\textrm{Fun}(\textrm{Add}(\cC_0), \aA b) \simeq \mathscr{F}(\cC_0, \aA b) \simeq \textrm{Mod-}A$, for the category $\mathscr{F}(\cC_0, \aA b)$ of $\aA b$-enriched functors $\cC_0 \to \aA b$.
\end{proof}

Recall that $K \xrightarrow{i} X$ is a \emph{weak kernel} of $X \xrightarrow{\f}Y$ if $\f\circ i =0$ and for any $f\colon Z\to X$ such that $\f \circ f =0$ there exist (not necessarily unique) $g\colon Z\to K$ such that $i\circ g = f$. \emph{Weak cokernels} are defined dually.

\begin{PROP}\label{prop_envelope_if_weak_kernels}
	Let $\cC$ be an additive category with weak kernels. Endow $\cC$ with the split exact structure. Then $h\colon \cC \to \textrm{fp}(\cC)$ is the right abelian envelope.
\end{PROP}	 
\begin{proof}
	The category $\textrm{fp}(\cC)$ is abelian if and only if $\cC$ has weak kernels. This was essentially proved in \cite[Theorem 1.4]{Freyd1} and later formulated in \cite[Proposition 4.5]{Bel}.
	
	Clearly, $\Rex(\cC, \bB)$ is equivalent to $\textrm{Fun}(\cC, \bB)$, for any abelian $\bB$. As $\cC\subset \textrm{fp}(\cC)$ is projectively generating, the fact follows from Proposition \ref{prop_right_ext_fun_and_A_obj} below, formulated in a more general set-up.	
\end{proof}

\begin{PROP}\cite[Proposition 3.6]{Bel} \label{prop_proj_in_fp}
	Let $\cC$ be an idempotent split additive category with weak kernels. Then $\cC \subset \textrm{fp}(\cC)$ is the category of projective objects.
\end{PROP}
 
By considering the opposite categories, we get
\begin{PROP}
	Let $\cC$ be an additive category with weak cokernels. Endow $\cC$ with the split exact structure. Then $\cC \to \textrm{fp}(\cC^{\opp})^{\opp}$ is the left abelian envelope.
\end{PROP}

\vspace{0.3cm}
\subsection{Envelopes in the presence of projective/injective generators}\label{ssec_env_if_proj_inj_gen}~\\

Let $\pP \subset \eE$ be  a projectively generating subcategory of an exact category. Then every $E\in \eE$ admits a \emph{projective presentation}, by which we mean a pair of composable morphisms $P_1\xrightarrow{i\circ q'} P_0 \xrightarrow{q} E$ given by a choice of two conflations $E'\xrightarrow{i} P_0 \xrightarrow{q} E$ and $E'' \xrightarrow{i'} P_1 \xrightarrow{q'} E'$. 
Let $P_1 \xrightarrow{i\circ q'} P_0 \xrightarrow{q} E$, $P'_1 \xrightarrow{j\circ p'} P'_0 \xrightarrow{p} F$ be projective presentations of objects $E , F\in \eE$. A \emph{morphism of projective presentations} is a commutative diagram
\begin{equation}\label{eqtn_morph_of_proj_pres}
\xymatrix{P'_1\ar[r]^{j\circ p'} & P'_0 \ar[r]^{p} & F \\
	P_1 \ar[r]^{i\circ q'} \ar[u]^{f_1} & P_0 \ar[r]^q \ar[u]^{f_0} & E \ar[u]^f}
\end{equation}
If $\bB$ is an abelian category and $F \colon \eE\to \bB$ a right exact functor, then $F$ takes a projective presentation of $E$ to an exact sequence $F(P_1) \to F(P_0) \to F(E) \to 0$ in $\bB$.

\begin{LEM}cf. \cite[Lemma 3.6.1]{Pop}\label{lem_proj_resol}
	Let $\pP\subset \eE$ be a projectively generating subcategory of an exact category.  
	\begin{itemize} 
		\item[(i)] Any morphism $f\colon E \to F$ in $\eE$ admits a lift to a morphism of projective presentations (\ref{eqtn_morph_of_proj_pres}).
		\item[(ii)] For an additive functor  $\Phi \colon \pP\to \bB$ to an abelian category $\bB$, the induced morphism $\textrm{coker }\Phi(i\circ q') \to \textrm{coker }\Phi(j \circ p')$ does not depend on the choice of $f_0$ and $f_1$.
	\end{itemize} 
\end{LEM} 
\begin{proof}
	The standard proof for the case when $\eE$ is abelian (\textit{cf.} \cite[Lemma 3.6.1]{Pop}) works for the more general case of an exact $\eE$ as well.
\end{proof}

\begin{PROP}\label{prop_right_ext_fun_and_A_obj}
	Let $\mathcal{P} \subset \eE$ be a projectively generating subcategory of an exact category $\eE$. Then $\Gamma\colon \Rex(\eE, \bB) \to \Fun(\mathcal{P},\bB)$, $\Gamma(F) =F|_{\mathcal{P}}$, is an equivalence for any abelian category $\bB$.
\end{PROP}
\begin{proof}
	Let $\Phi \colon \pP\to \bB$ be an additive functor. The quasi-inverse of $\Gamma$ is defined by extending $\Phi$ to a functor $\wt{\Phi} \colon \eE\to \bB$ defined on an object $E\in \eE$ via $\wt{\Phi}(E) := \textrm{coker } \Phi(i \circ q')$, for a projective presentation $P_1 \xrightarrow{i \circ q'} P_0 \xrightarrow{q} E$. One easily checks that $\wt{\Phi}$ is a right exact extension of $\Phi$. The uniqueness is guaranteed by Lemma \ref{lem_proj_resol}.
\end{proof}

We use Propositions \ref{prop_envelope_if_weak_kernels} and \ref{prop_right_ext_fun_and_A_obj} to construct the right abelian envelope of an exact category with a projectively generating subcategory $\pP$. 

\begin{THM}\label{thm_right_env_if_proj_gen}
	Assume that a projectively generating subcategory $\pP$ of an exact category $\eE$ has weak kernels.
	Then the functor $h^{\wedge} \colon \eE\to \textrm{fp}(\pP)$, $h^{\wedge}(E)(-) = \Hom(-,E)$ is the right abelian envelope for $\eE$. Moreover, $\eE$ is a fully exact subcategory of $\textrm{fp}(\pP)$.
\end{THM}
\begin{proof}
	Proposition \ref{prop_envelope_if_weak_kernels} implies that the Yoneda embedding $h \colon \pP \to \textrm{fp}(\pP)$ is the right abelian envelope. 
	
	Fix an abelian category $\bB$ and consider the composite $\Upsilon \colon \Rex(\textrm{fp}(\pP), \bB) \xrightarrow{(-)\circ h^\wedge} \Rex(\eE, \bB) \xrightarrow{(-)\circ i} \Fun(\pP, \bB)$, where $i\colon \pP\to \eE$ is the inclusion functor. By Proposition \ref{prop_right_ext_fun_and_A_obj}, the second functor is an equivalence. The composite $h^\wedge \circ i \colon \pP \to \textrm{fp}(\pP)$ is the Yoneda embedding $h$, hence $\Upsilon(-) = (-) \circ h^\wedge \circ i = (-) \circ h$ is an equivalence. It follows that the precomposition with $h^\wedge$ is an equivalence $\Rex(\textrm{fp}(\pP), \bB) \simeq \Rex(\eE, \bB)$, i.e. $h^\wedge$ is the right abelian envelope of $\eE$.
		
	By Lemma \ref{lem_if_I_R_full_then_exact}, to show that $h^{\wedge} \colon \eE \to \textrm{fp}(\pP)$ is an embedding of a fully exact subcategory, it suffices to check that $h^{\wedge}$ is full.
	Since $\pP$ is projectively generating, the composite of Yoneda embedding $h\colon \eE\to \Rex(\eE, \Ab^{\opp})$ with the restriction functor $(-)\circ i \colon \Rex(\eE, \Ab^{\opp}) \to \Fun(\pP, \Ab^{\opp})$ takes values in $\textrm{fp}(\pP)$. The induced functor $\eE \to \textrm{fp}(\pP)$ is $h^\wedge$.
	Functor $h$ is fully faithful by the Yoneda lemma and so is $(-)\circ i$ by Proposition \ref{prop_right_ext_fun_and_A_obj}.
Hence, $h^\wedge$ is fully faithful too.
\end{proof}

Dually, we have
\begin{THM}\label{thm_left_env_if_inj_gen}
	Let $\mathcal{I} \subset \eE$ be an injectively generating subcategory of an
	exact category $\eE$. Assume that $\mathcal{I}$ has weak cokernels. Then the functor $h^\vee\colon \eE\to \textrm{fp}(\mathcal{I}^{\opp})^{\opp}$, $h^\vee(E)(-) = \Hom(E,-)$ is the left abelian envelope for $\eE$. Moreover, $\eE$ is a fully exact subcategory of $\textrm{fp}(\mathcal{I}^{\opp})^{\opp}$.
\end{THM}
\begin{proof}
	Injectively generating subcategory $\mathcal{I}\subset \eE$ yields a projectively generating subcategory in $\eE^{\opp}$. The statement follows from Theorem \ref{thm_right_env_if_proj_gen} and formula (\ref{rem_envs_for_opp}).
\end{proof}

The following lemma gives a criterion for existence of weak kernels which we shall use for projectively generating subcategories. 
\begin{LEM}\label{lem_P_has_weak_kernels}
	Let $\cC$ be a $k$-linear additive category generated by finitely many objects $\{C_i\}_{i=1}^n$. If the algebra $A_C:=\End_{\cC}(\bigoplus_{i=1}^n C_i)$ is right coherent,  then $\cC$ has weak kernels.
\end{LEM}
\begin{proof}
	By Lemma \ref{lem_fp_as_modules}, $\textrm{fp}(\cC) \simeq \textrm{mod-}A_C$. As algebra $A_C$ is right coherent, $\textrm{fp}(\cC)$ is abelian. 
	
	The essential image of $h\colon \cC\to \textrm{fp}(\cC)$ consists of projective $A_C$-modules and every $A_C$-module is a quotient of $h^C$, for some $C\in \cC$. 	
	Consider $f\in \Hom_{\cC}(C, C')$. The kernel $K$ of $h^f \colon h^{C} \to h^{C'}$ is a quotient of $h^{C_K}$, for some $C_K \in \cC$. Then morphism $g\colon C_K \to C$ whose image under $h$ is the composite $h^{C_K} \to K \to h^C$ is a weak kernel of $f$.
\end{proof}

\begin{COR}\label{cor_envel_for_thin}
	Consider a thin category $\eE$. Denote by $\pP\subset \eE$, respectively $\iI \subset \eE$,  the subcategory of projective, respectively injective, objects. Then $\textrm{fp}(\pP)$, respectively $\textrm{fp}(\iI^{\opp})^{\opp}$, is the right, respectively left, abelian envelope for $\eE$.
\end{COR}
\begin{proof}
	By definition a thin category $\eE$ is Hom-finite.
	By Proposition \ref{prop_proj_in_thin} and Remark \ref{rem_injective_in_thin}, category $\eE$ has a projective generator $\bigoplus_{\la \in \Lambda}P(\la)$ and an injective generator $\bigoplus_{\la \in \Lambda} I(\la)$. Algebras  $\End(\bigoplus_{\la \in \Lambda}P(\la))$ and $\End(\bigoplus_{\la \in \Lambda}I(\la))^{\opp}$ are finite dimensional, hence coherent. Then	Lemma \ref{lem_P_has_weak_kernels} implies that $\pP = \textrm{add}\{P(\la)\}$ has weak kernels and $\iI = \textrm{add}\{I(\la)\}$ has weak cokernels. 
	The statement follows from Theorems \ref{thm_right_env_if_proj_gen} and \ref{thm_left_env_if_inj_gen}.
\end{proof}

\vspace{0.5cm}
\subsection{Derived equivalence of $\eE$ and $\aA_r(\eE)$}\label{ssec_der_equiv}~\\

Recall that the derived category $\dD^-(\eE)$ of a weakly idempotent split exact category $\eE$ is the quotient of the homotopy category of complexes $\kK^-(\eE)$ by the full subcategory of acyclic complexes. The following fact is well-known.
\begin{LEM}\cite[Example 12.2]{Kel5}\label{lem_homotopy_cat_of_proj}
	Let $\pP \subset \eE$ be a projectively generating subcategory in a  weakly idempotent split exact category $\eE$. Then $\dD^-(\eE) \simeq \kK^-(\pP)$.
\end{LEM}

For an exact category $\eE$ with a projectively generating subcategory $\pP$ with weak kernels, denote by $Li_R\colon \dD^\sharp(\eE) \to \dD^\sharp(\aA_r(\eE))$ the functor of the bounded (bounded above) derived categories induced by the exact functor $i_R\colon \eE\to \aA_r(\eE)$ (see Theorem \ref{thm_right_env_if_proj_gen}).

\begin{PROP}\label{prop_derived_equiv}
	Let $\eE$ be a weakly idempotent split exact category admitting a projectively generating subcategory $\pP$ with weak kernels. Then  $L i_R\colon \dD^-(\eE) \xrightarrow{\simeq} \dD^-(\aA_r(\eE))$ is an equivalence.
\end{PROP}
\begin{proof}
	Theorem \ref{thm_right_env_if_proj_gen} implies that $\aA_r(\eE) \simeq \textrm{fp}(\pP)$.
	Since $\pP$ is projectively generating both in  $\eE$ and in $\textrm{fp}(\pP) \simeq \aA_r(\eE)$, Lemma \ref{lem_homotopy_cat_of_proj} implies that $\dD^-(\eE) \simeq \kK^-(\pP) \simeq \dD^-(\aA_r(\eE))$.
\end{proof}

We leave to the reader the formulation of an analogous statement for exact categories with injectively generating subcategories and bounded below derived categories.

\begin{THM}\label{thm_bounded_derived_equi}
	Let $\eE$ be a weakly idempotent split exact category admitting a projectively generating subcategory $\pP$ with weak kernels. Assume that $\pP$ is idempotent split and $\aA_r(\eE)$  is of finite global dimension. Then $Li_R \colon \dD^b(\eE) \to \dD^b(\aA_r(\eE))$ is an equivalence.
\end{THM} 
\begin{proof}
	By Proposition \ref{prop_proj_in_fp}, the idempotent split category $\pP$ is equivalent to the category of projective objects in $\textrm{fp}(\pP)$, which is equivalent to $\aA_r(\eE)$ by Theorem \ref{thm_right_env_if_proj_gen}. The finite global dimension of $\aA_r(\eE)$ implies that $\dD^b(\aA_r(\eE)) \simeq \kK^b(\pP)$.
	
	Let $\textrm{gldim} \aA_r(\eE) = n$.
	We check that any $E\in \eE$ admits a finite projective resolution, i.e. there exists an acyclic complex $0 \to P_{n+1} \to P_{n} \to \ldots \to P_0 \to E \to 0$ in $\kK(\eE)$ with $P_i \in \pP$. This is not obvious because the acyclicity of a complex $i_R(E_\bcdot) \in \kK(\aA_r(\eE))$ does not necessarily imply the acyclicity of $E_\bcdot \in \kK(\eE)$. Since $\pP\subset \eE$ is projectively generating, object $E\in \eE$ admits a resolution $P_\bcdot = \{P_{n}\to \ldots \to P_0\to E \to 0\}$ with $P_i\in \pP$. By Theorem \ref{thm_right_env_if_proj_gen}, $i_R(\eE) \subset\aA_r(\eE)$ is a fully exact subcategory, hence $i_R(P_\bcdot)\in \kK(\aA_r(\eE))$ is an acyclic complex with $i_R(P_i)\in \aA_r(\eE)$ projective. The assumption on the global dimension of $\aA_r(\eE)$ implies that the kernel $K$ of $i_R(P_{n}) \to i_R(P_{n-1})$ is projective. Hence, $K \simeq i_R(P_{n+1})$,  for some $P_{n+1} \in \pP$ (see Proposition \ref{prop_proj_in_fp}).
	Then $0 \to P_{n+1} \to P_{n} \to \ldots \to P_0 \to E \to 0$ is a finite projective resolution of $E$.
		
	It follows that the functor $\iota \colon \kK^b(\pP) \to \dD^b(\eE)$ is an equivalence. Indeed,  Lemma \ref{lem_homotopy_cat_of_proj} implies that $\iota$ is fully faithful and the essential surjectivity of $\iota$ follows from the existence of finite projective resolutions by the unwinding of complexes.
	Then, the equivalence $Li_R$ of Proposition \ref{prop_derived_equiv} restricts to an equivalence $Li_R \colon \dD^b(\eE) \to \dD^b(\aA_r(\eE))$.
\end{proof}

The following Lemma ensures that the condition in Theorem \ref{thm_bounded_derived_equi} on category $\pP$ to be idempotent split is verified when $\pP$ is Hom-finite and additively generated by a finite number of objects.

\begin{LEM}\label{lem_idempotent_split}
	Let $\cC$ be a $k$-linear, Hom-finite additive category additively generated by objects $C_1,\ldots,C_n$. If, for all $i$, $\End_\cC(C_i)$ is local, then $\cC$ is idempotent split.
\end{LEM}
\begin{proof}
	Algebra $A_C:=\End_{\cC}(\bigoplus_{i=1}^n C_i)$ is finite dimensional, hence right coherent. By Lemma \ref{lem_P_has_weak_kernels}, category $\cC$ has weak kernels. By Theorem \ref{thm_right_env_if_proj_gen}, $\cC$ is a full subcategory of the Deligne finite category $\aA_r(\cC)\simeq \textrm{fp}(\cC)\simeq \textrm{mod-}A_C$ (see Appendix \ref{sec_abelian_cat} and Lemma \ref{lem_fp_as_modules}). For any $i$, object $C_i\in \aA_r(\cC)$ is indecomposable, as $\End_{\cC}(C_i) \simeq \End_{\aA_r(\cC)}(C_i)$ has no non-trivial idempotent. Since $\aA_r(\cC)$ is	 Krull-Schmidt, any direct summand of $C := \bigoplus C_i^{a_i} \in \cC$ is isomorphic to a direct sum of copies of $C_i$'s. As $\textrm{fp}(\cC)$ is idempotent split, $\cC$ is idempotent split too.
\end{proof}

\vspace{0.3cm}
\subsection{Abelian envelopes and admissible subcategories}\label{ssec_ab_env_adm_subcat}~\\

Here we discuss how the idea of left and right admissible subcategories of Section \ref{sec_exact_cat}	match the idea of abelian envelopes of this section. Recall that right and left admissible subcategories are in bijection with perpendicular torsion pairs (see Definition \ref{def_right_and_left_adm}).

\begin{LEM}\label{lem_enveloping_adjointness}
	Let	
	$\Phi \colon \eE\to \eE'$, $\Psi \colon \eE' \to \eE$ be 1-morphisms in $\mathscr{E}x_r$ such that $\Phi$ is the left adjoint functor to $\Psi$.
Then $\aA_r(\Phi)$ is left adjoint to $\aA_r(\Psi)$ and $\aA_r(\Psi)$ is exact.
\end{LEM}
\begin{proof} 
	As $\aA_r$ is a pseudo-functor $\mathscr{E}x_r \to \mathscr{E}x_r$ (see Proposition \ref{prop_envelope_functorial}), it maps adjoint functors to adjoint functors, i.e. $\aA_r(\Phi) \dashv \aA_r(\Psi)$.
Functor $\aA_r(\Psi)$ is right adjoint, hence it is left exact. By definition, it is also right exact, which implies that $\aA_r(\Psi)$ is exact.
\end{proof} 

Let $(\tT, \fF)$ be a perpendicular torsion pair in an exact category $\eE$. Assume that $\aA_r(\eE)$ and $\aA_r(\tT)$ exist. By Lemma \ref{lem_enveloping_adjointness}, the adjoint pair $\infl_*\dashv \infl^!$ of exact functors (see Theorem \ref{thm_admissible_iff_adjoint_exact}) yields an adjoint pair $\aA_r(\infl_*)\dashv \aA_r(\infl^!)$ of functors between the abelian envelopes with $\aA_r(\infl^!)$ exact. We check that $\aA_r(\fF)$ is the kernel of $\aA_r(\infl^!)$:

\begin{PROP}\label{prop_envelope_becomes_Serre_subcat}
		Let $(\tT, \fF)$ be a perpendicular torsion pair in an exact category $\eE$. 
		If
		$\aA_r(\eE)$ and $\aA_r(\tT)$ exist, then so does $\aA_r(\fF)$, and $\aA_r(\defl_*)$ identifies it with the colocalising subcategory $\kK$, the kernel of $ \aA_r(\infl^!)$. 
		The universal functor $i_R^{\fF} \colon \fF\to \aA_r(\fF) \simeq \kK$ is the unique functor which fits into a commutative square
		\[
		\xymatrix{\kK \ar[r]^{\kappa_*} & \aA_r(\eE)\\
	\fF \ar[u]^{i_R^{\fF}} \ar[r]^{\defl_*} & \eE, \ar[u]^{i_R} }
		\] 
		where $\kappa_*\colon \kK\to \aA_r(\eE)$ denotes the inclusion.
\end{PROP} 
\begin{proof}
	Functor $\aA_r(\infl^!)$ is  exact, as it has left adjoint $\aA_r(\infl_*)$ (see Lemma \ref{lem_enveloping_adjointness}). As $\aA_r(\infl_*)$ is fully faithful, category $\kK\subset \aA_r(\eE)$ is a colocalising subcategory (see Proposition \ref{prop_Popescu}). We denote by $\kappa^*\colon \aA_r(\eE) \to \kK$ the functor left adjoint to $\kappa_*$ (see Lemma \ref{lem_adjoint_to_coloc}).
	
	The composite $\fF \xrightarrow{\defl_*} \eE \xrightarrow{i_R} \aA_r(\eE) \xrightarrow{\aA_r(\infl^!)} \aA_r(\tT)$ is zero, as $\infl^! \circ \defl_* \colon \fF \to \tT$ is. Hence, there exists $i_R^{\fF}\colon \fF \to \aA_r(\eE)$ such that $\kappa_*\circ i_R^{\fF}= i_R\circ \defl_*$. Functor $i_R^{\fF}$ is right exact, because $\kappa_*\circ i_R^{\fF}$ is.
	
	Let $\defl^* \colon \eE\to \fF$ be the left adjoint to the inclusion $\defl_*\colon \fF\to \eE$. As functor $\kappa^*\circ i_R \colon \eE \to \kK$ vanishes on $\tT \subset \eE$, decomposition (\ref{decomposition}) implies an isomorphism $\kappa^* \circ  i_R \simeq \kappa^*\circ  i_R\circ  \defl_*\circ  \defl^*$. Then, $\kappa^* \circ  i_R \simeq \kappa^*\circ  i_R\circ  \defl_*\circ  \defl^*\simeq \kappa_*\circ \kappa^*\circ  i_R^{\fF}\circ   \defl^* \simeq i_R^{\fF}\circ  \defl^*$.  
	
	Let $\bB$ be an abelian category. We check that $(-) \circ i_R^{\fF} \colon \Rex(\kK, \bB) \to \Rex(\fF,\bB)$ is essentially surjective and fully faithful.
	
	For $F\in \Rex(\fF, \bB)$, functor $F \circ \defl^*$ is an object of $\Rex(\eE, \bB)$. Hence, there exist $\ol{F}\in \Rex(\aA_r(\eE), \bB)$ such that $\ol{F}\circ i_R \simeq F \circ \defl^*$. Then $\ol{F}\circ \kappa_*$ is a functor in $\Rex(\kK, \bB)$ such that $\ol{F} \circ \kappa_* \circ i_R^{\fF} \simeq \ol{F} \circ i_R \circ \defl_* \simeq F\circ \defl^* \circ \defl_* \simeq F$.
	
	A natural transformation $\nu \colon G_1 \to G_2$ of functors in $\Rex(\kK, \bB)$ induces a natural transformation $\nu \circ \kappa^*$ of functors in $\Rex(\aA_r(\eE), \bB)$. The composite with $i_R$ gives a natural transformation $\nu \circ \kappa^* \circ i_R \simeq \nu \circ i_R^{\fF}\circ \defl^*$ of functors in $\Rex(\eE, \bB)$. If $\nu \circ i_R^{\fF}$ is zero, then so is $\nu \circ \kappa^* \circ i_R$. As $(-)\circ i_R \colon \Rex(\aA_r(\eE), \bB) \xrightarrow{\simeq} \Rex(\eE, \bB)$ is faithful, natural transformation $\nu \circ \kappa^* \circ i_R$ is zero if and only if $\nu \circ \kappa^*$ is. If it is the case, then $\nu \simeq \nu \circ \kappa^* \circ \kappa_*$ is zero, i.e. $(-)\circ i_R^{\fF}\colon \Rex(\kK, \bB) \to \Rex(\fF,\bB)$ is faithful.
	
	Finally, let $\nu \colon G_1\circ i_R^{\fF} \to G_2\circ  i_R^{\fF}$ be a natural transformation. It induces a natural transformation $\nu \circ \defl^*$ of functors in $\Rex(\eE, \bB)$. As $G_j \circ i_R^{\fF} \circ \defl^* \simeq G_j\circ \kappa^* \circ i_R$, for $j=1,2$, fullness of $(-)\circ i_R \colon \Rex(\aA_r(\eE), \bB) \to \Rex(\eE, \bB)$ implies existence of  $\ol{\nu} \colon G_1 \circ \kappa^* \to G_2 \circ \kappa^*$ such that $\ol{\nu} \circ i_R = \nu \circ \defl^*$. Then $\ol{\nu} \circ \kappa_*\colon G_1 \circ \kappa^* \circ \kappa_*\simeq G_1 \to G_2 \simeq G_2 \circ \kappa^* \circ \kappa_*$ is a natural transformation such that $\ol{\nu}\circ \kappa_* \circ i_R^{\fF} \simeq \ol{\nu} \circ i_R \circ \defl_* \simeq \nu \circ \defl^* \circ \defl_*\simeq \nu$.
\end{proof}
Dually we have:
\begin{PROP}
		Let $(\tT, \fF)$ be a perpendicular torsion pair in an exact category $\eE$. 
		If
	$\aA_l(\eE)$ and $\aA_l(\fF)$ exist, then so does $\aA_l(\tT)$ and $\aA_l(\infl_*)$ identifies it with the localising subcategory $\kK$, the kernel of $ \aA_l(\defl^*)$. 
		The universal functor $i_L^{\tT} \colon \tT\to \aA_l(\tT) \simeq \kK$ is the unique functor which fits into a commutative square
	\[
	\xymatrix{\kK \ar[r]^{\kappa_*} & \aA_l(\eE)\\
		\tT \ar[u]^{i_L^{\tT}} \ar[r]^{\infl_*} & \eE, \ar[u]^{i_L} }
	\]
	where $\kappa_*\colon \kK\to \aA_l(\eE)$ denotes the inclusion.
\end{PROP} 

The following proposition gives a suitable characterisation of the exact category with a perpendicular torsion pair as a subcategory in its right abelian envelope.

We say that $(\tT, \fF)$ is a perpendicular \emph{torsion pair with right envelopes} in an exact category $\eE$ if the right abelian envelopes $\aA_r(\eE)$, $\aA_r(\tT)$, and $\aA_r(\fF)$ exist. 

\begin{PROP}\label{prop_E_in_A_r}
	Let $(\tT, \fF)$ be a perpendicular torsion pair with right envelopes in a $k$-linear exact category $\eE$. Assume that $i_R\colon \eE \to \aA_r(\eE)$ is full. 
	Then an object $A \in \aA_r(\eE)$ is isomorphic to $i_R(E)$, for some $E\in \eE$, if and only if 
	\begin{enumerate}
		\item $\aA_r(\defl^*)(A) \simeq i_R(F)$, for some $F\in \fF$,
		\item $\aA_r(\infl^!)(A) \simeq  i_R(T)$, for some $T\in \tT$, and
		\item the adjunction counit $\aA_r(\infl_*) \aA_r(\infl^!)A \to A$ is a monomorphism.
	\end{enumerate}  
\end{PROP}
\begin{proof}
	By Lemma \ref{lem_if_I_R_full_then_exact} functor $i_R$ is exact and its essential image is closed under extensions.
	
	An object $E\in \eE$ fits into a conflation $\infl_*\infl^!E \to E \to \defl_*\defl^*E$. As $i_R$ is exact, 
	$$
	 0 \to \aA_r(\infl_*) \aA_r(\infl^!)i_R(E) \to i_R(E) \to \aA_r(\defl_*) \aA_r(\defl^*)i_R(E) \to 0
	 $$ 
	 is a short exact sequence in $\aA_r(\eE)$.
	
	Now let $A \in \aA_r(\eE)$ satisfy condition $(1)$ and $(2)$.
	Then $\aA_r(\defl_*)\aA_r(\defl^*)A \simeq i_R(\defl_*F)$ and $\aA_r(\infl_*)\aA_r(\infl^!)A \simeq i_R(\infl_*T)$. By Proposition \ref{prop_envelope_becomes_Serre_subcat}, $\aA_r(\fF) \subset \aA_r(\eE)$ is a colocalising subcategory with quotient  functor $\aA_r(\infl^!)$. By Lemma \ref{lem_adjoint_to_coloc}, sequence
	$$
	\aA_r(\infl_*) \aA_r(\infl^!) A \xrightarrow{\varepsilon_A} A \to \aA_r(\defl_*) \aA_r(\defl^*)A \to 0
	$$
	is exact. 
	By condition $(3)$, $\varepsilon_A$ is injective, hence $A$ is an extension of $i_R(\defl_*F)\simeq \aA_r(\defl_*) \aA_r(\defl^*)A$ by $i_R(\infl_* T)\simeq \aA_r(\infl_*) \aA_r(\infl^!) A$. Since $\eE\subset \aA_r(\eE)$ is closed under extensions, we conclude that $A \in \eE$. 
\end{proof}
 By Proposition \ref{prop_strict_filtr_on_thin}, a lower ideal $I$ in the canonical poset of a thin category gives a perpendicular torsion pair $(\tT_I, \fF_I)$.
\begin{PROP}
Let $\eE$ be a thin category and $\Lambda$ its canonical poset. For a lower ideal $I\subset \Lambda$, the perpendicular torsion pair $(\tT_I, \fF_I)$ 
yields an abelian recollement
$$
\aA_r(\fF_I) \xrightarrow{{i_I}_*} \aA_r(\eE) \xrightarrow{j_I^*} \aA_r(\tT_I).
$$
\end{PROP}
\begin{proof} 
By Proposition \ref{prop_envelope_becomes_Serre_subcat}, $\aA_r(\fF_I) \subset \aA_r(\eE)$ is a Serre subcategory. Corollary \ref{cor_envel_for_thin} and Lemma \ref{lem_fp_as_modules} imply that $\aA_r(\eE)$ is a Deligne finite category, see Appendix \ref{ssec_ab_cat_fin_len}.  By Proposition \ref{prop_biloc_in_Df_n}, $\aA_r(\fF_I) \subset \aA_r(\eE)$  is bi-localising, hence it yields an abelian recollement by Proposition \ref{prop_charact_of_abel_recol}.
\end{proof}
We denote by $Li_I^* \colon \dD^-(\aA_r(\eE)) \to \dD^-(\aA_r(\fF_I))$ the derived functor of $i_I^*$. It exists because $\aA_r(\eE)$ has enough projective objects (see Corollary \ref{cor_envel_for_thin}).

\begin{LEM}\label{lem_kern_of_counit}
	For $A\in \aA_r(\eE)$, the kernel $K_A$ of the adjunction unit ${j_I}_!j_I^*A \to A$ is isomorphic to ${i_I}_*L^1i_I^*A$.
\end{LEM}
\begin{proof}
	Let $A\in \aA_r(\eE)$ and $0 \to B \to P \to A \to 0$ an exact sequence with $P\in \aA_r(\eE)$ projective. Object ${i_I}_*L^1i_{I}^*A$ is isomorphic to the kernel of ${i_I}_*i_I^*B \to {i_I}_*i_I^*P$. As, for any $A'\in \aA_r(\eE)$, 
	the cokernel of ${j_I}_!j_I^*A' \to A'$ is isomorphic to  ${i_I}_*i_I^*A'$ (see Lemma \ref{lem_adjoint_to_coloc}),  the snake lemma for
	\[
	\xymatrix{0 \ar[r]& B \ar[r] & P \ar[r] & A \ar[r]& 0\\
	& {j_I}_!j_I^*B \ar[r] \ar[u] & {j_I}_!j_I^*P \ar[r] \ar[u] & {j_I}_!j_I^*A \ar[r] \ar[u] & 0 }
	\]
	yields an exact sequence $K_{B}\to K_P \to K_A \to {i_I}_*{i_I}^*B \to {i_I}_*{i_I}^*P \to {i_I}_*{i_I}^*A \to 0$. Projective $P\in \aA_r(\eE)$ lies in the subcategory $\eE\subset \aA_r(\eE)$, hence Proposition \ref{prop_E_in_A_r} implies that the adjunction counit ${j_I}_!j_I^*P \to P$ is injective, i.e. $K_P \simeq 0$. Hence, $K_A \simeq \textrm{ker }({i_I}_*{i_I}^*B \to {i_I}_*{i_I}^*P)\simeq {i_I}_*L^1i_I^*A$.
\end{proof}

\vspace{0.3cm}
\subsection{Induced filtrations on abelian envelopes}\label{ssec_col_str_filt_on_env}~\\ \label{ssec_stric_filt}

For an abelian category $\aA$, let $\textrm{Serre}(\aA)$ denote the poset of Serre subcategories in $\aA$ with the inclusion order.

Let $\lL$ be  a finite lattice with the minimal element $0$ and the maximal element $1$. A \emph{localising $\lL$-filtration} on an abelian category $\aA$ is a map of posets $\lL \to \textrm{Serre}(\aA)$, $I\mapsto \aA_I$, such that 
\begin{itemize}
	\item[(Si)] for any $I\in \lL$, $\aA_I\subset \aA$ is a localising subcategory, $\aA_0 \simeq 0$, $\aA_1\simeq \aA$, 
	\item[(Sii)] for any $I, J \in \lL$, $\aA_{I\cap J} = \aA_I \cap \aA_J$ and $\aA_{I \cup J}^\perp \simeq \aA_I^\perp \cap \aA_J^\perp$.
\end{itemize}
A \emph{colocalising $\lL$-filtration} on $\aA$ is a map of posets $\lL \to \textrm{Serre}(\aA)$, $I\mapsto \aA_I$, such that 
\begin{itemize}
	\item[(Si)] for any $I\in \lL$, $\aA_I\subset \aA$ is a colocalising subcategory, $\aA_0 \simeq 0$, $\aA_1\simeq \aA$, 
	\item[(Sii)] for any $I, J \in \lL$, $\aA_{I\cap J} = \aA_I \cap \aA_J$ and ${}^\perp\aA_{I \cup J} \simeq {}^\perp\aA_I \cap {}^\perp\aA_J$.
\end{itemize}
We say that a map of posets $\lL \to \textrm{Serre}(\aA)$ is a \emph{bi-localising} $\lL$-filtration if it is both localising and colocalising.

We say that a (co)localising $\lL$-filtration is \emph{strict} if
\begin{enumerate}
	\item[(Siii)] for any $I, J \in \lL$, $\aA_{I\cup J}/\aA_{I \cap J} \simeq \aA_I/\aA_{I\cap J} \oplus \aA_J/\aA_{I \cap J}$.
\end{enumerate}

Unlike for exact categories, condition $\aA_{I\cap J}^\perp \cap \aA_J \simeq \aA_I^\perp \cap \aA_{I\cup J}$ does not imply (Siii):
\begin{EXM}\label{exm_strictness}
	Let $\aA$ be the category of representations of the quiver with relations:
	\[
	\xymatrix{1 \ar@<1ex>[r]^a & 2 \ar@<1ex>[l]^b \ar@<1ex>[r]^c & 3 \ar@<1ex>[l]^d} \quad ba=0,\, bdca = 0,\, cd =0,\, cabd = 0.
	\]
	Consider the partial order $\Lambda$
	with $2\prec 1$ and $2 \prec 3$. The lattice $\lL$ of lower ideals in $\Lambda$ has 5 elements $\emptyset$, $\{2\}$, $\{1, 2\}$, $\{2, 3\}$, $\{1, 2, 3\}$ ordered by inclusion. The map $\lL\to \textrm{Serre}(\aA)$, $I \mapsto \fF(S_i)_{i\in I}$ is a bi-localising $\lL$-filtration on $\aA$ (see Proposition \ref{prop_biloc_in_Df_n}).
	
	If $I = \{1,2\}$, $J = \{2, 3\}$, then $I\cap J  =\{2\}$ and $I \cup J = \{1, 2,3\}$. The category $\aA_{I\cap J}^\perp \cap \aA_{J} \simeq \aA_{23}/\aA_2 \simeq k\textrm{-vect}$ is equivalent to $\aA_{I}^\perp \cap \aA_{I\cup J} \simeq \aA/\aA_{12} \simeq k \textrm{-vect}$, while the quotient $\aA_{I \cup J}/\aA_{I \cap J} = \aA/\aA_2$ is  the category of representations of the quiver 
	\[
	\xymatrix{1 \ar@<1ex>[r]^{ca} & 3 \ar@<1ex>[l]^{bd}}\quad bdca =0,\, cabd =0.
	\]
\end{EXM}

Let $\lL$ be as above. For $I\in \lL$, let $\lL/I$ be the sublattice of $\lL$ of the elements $J$ with $I\prec J$. Then $\lL/I$ is a lattice with the minimal element $I$ and the maximal element $1$.
\begin{LEM}\label{lem_filtr_on_quotient}
	Let $\{\aA_J\}$ be a strict (co)localising $\lL$-filtration on $\aA$. Then $\{\aA_J/\aA_I\}$ is a strict (co)localising $\lL/I$-filtration on $\aA/\aA_I$.
\end{LEM}
\begin{proof}
	If $\aA_I \subset \aA_J$ and $\aA_J\subset \aA$ are colocalising subcategories, then $\aA_J/\aA_I \subset \aA/\aA_I$ is a colocalising subcategory, \textit{cf.} Lemma \ref{lem_shriek_commute}. We have:
	\begin{align*}
	&(\aA_J/\aA_I) \cap (\aA_K/\aA_I) = (\aA_J \cap \aA_K)/\aA_I = \aA_{J \cap K}/\aA_I.&
	\end{align*}
	By Proposition \ref{prop_colocalising}, ${}^\perp\aA_J$ is identified with  $\aA/\aA_J$  via the functor left adjoint to the quotient functor. Since $(\aA/\aA_I)/(\aA_J/\aA_I) \simeq \aA/\aA_J$, we have  ${}^{\perp}(\aA_J/\aA_I) = \aA/\aA_J$. Hence, 
	\begin{align*} 
	&{}^\perp(\aA_{J\cup K}/\aA_I) = \aA/\aA_{J \cup K} = \aA/\aA_J \cap \aA/\aA_K = {}^\perp(\aA_J/\aA_I) \cap {}^\perp(\aA_K/\aA_I).&
	\end{align*} 
		Equivalence $(\aA_{K\cup J}/\aA_I)/(\aA_{K\cap J}/\aA_I) \simeq \aA_{K \cup J}/\aA_{K\cap J}$ implies (Siii).
\end{proof}

\begin{THM}\label{thm_filtr_on_envelope}
	Consider a finite lattice $\lL$ and a strict left admissible $\lL$-filtration $\{\fF_I\}$ on an exact category $\eE$. Assume that the exact category $\fF_I \cap {}^\perp \fF_J$ has right abelian envelope, for any $I, J\in \lL$. Then $\{\aA_r(\fF_I)\}$ is a strict colocalising $\lL$-filtration on $\aA_r(\eE)$.
\end{THM}
\begin{proof}
	For $I\in \lL$, let $\tT_I ={}^\perp \fF_I$. By Proposition \ref{prop_envelope_becomes_Serre_subcat}, $\aA_r(\fF_I) \subset \aA_r(\eE)$ is a colosalising subcategory with $\aA_r(\eE)/\aA_r(\fF_I) \simeq \aA_r(\tT_I)$. By Proposition \ref{prop_colocalising}, ${}^\perp(\aA_r(\fF_I)) \simeq \aA_r(\tT_I)$ considered as a subcategory of $\aA_r(\eE)$ via the functor $\aA_r(\infl_*)$, right adjoint to $\aA_r(\infl^!)$, see Lemma \ref{lem_enveloping_adjointness}.  As $0 \simeq \aA_r(0)$, condition (Si) holds.
	
	Propositions \ref{prop_colocalising} and \ref{prop_envelope_becomes_Serre_subcat} imply that, for a perpendicular torsion pair $(\tT, \fF)$ in an exact category $\eE$, $\aA_r(\fF) = \aA_r(\tT)^\perp$ and $\aA_r(\tT) = {}^\perp \aA_r(\fF)$.
	
	Since the $\lL$-filtration is strict, Proposition \ref{prop_direct_sum} implies existence of the perpendicular torsion pairs $(\xX_I, \fF_I)$, $(\xX_J, \fF_J)$, $(\xX_I \oplus \xX_J, \fF_{I\cap J})$ in $\fF_{I\cup J}$.
	Hence, $\aA_r(\fF_I) = \aA_r(\xX_I)^\perp \cap \aA_r(\fF_{I \cup J})$, similarly for $\aA_r(\fF_J)$ and $\aA_r(\fF_{I\cap J})$. Then:
	\begin{align*} 
	&\aA_r(\fF_I) \cap \aA_r(\fF_J) = 
		\aA_r(\xX_I)^\perp \cap \aA_r(\xX_J)^\perp \cap \aA_r(\fF_{I\cup J}) = &\\
		& =\aA_r(\xX_I \oplus \xX_J)^\perp \cap \aA_r(\fF_{I \cup J}) = 
	\aA_r(\fF_{I \cap J}).&
	\end{align*} 
	By Lemma \ref{lem_dual_strict}, the left dual $\lL^{\opp}$ filtration is also strict. An analogous argument for perpendicular torsion pairs in $\tT_{I \cap J}$, $\tT_I$ and $\tT_J$ analogously implies that ${}^\perp\aA_r(\fF_{I \cup J}) = {}^\perp \aA_r(\fF_I) \cap {}^\perp \aA_r(\fF_J)$.
	
		Consider $I, J \in \lL$.  By Proposition \ref{prop_envelope_becomes_Serre_subcat}, the quotient $\aA_r(\fF_{I \cup J})/\aA_r(\fF_{I \cap J})$ is equivalent to $\aA_r(\fF_{I \cup J} \cap {}^\perp\fF_{I \cap J}) =  \aA_r((\fF_I \cap {}^\perp \fF_{I \cap J}) \oplus (\fF_J \cap {}^\perp \fF_{I\cap J}))$, see Proposition \ref{prop_direct_sum}. As a right exact functor on $(\fF_I \cap {}^\perp \fF_{I \cap J} )\oplus( \fF_J \cap {}^\perp \fF_{I\cap J})$ is a pair of right exact functors, the category $\aA_r(\fF_I \cap {}^\perp \fF_{I \cap J}) \oplus \aA_r(\fF_J \cap {}^\perp \fF_{I\cap J})$ is the right abelian envelope for $(\fF_I \cap {}^\perp \fF_{I \cap J}) \oplus (\fF_J \cap {}^\perp \fF_{I\cap J})$, i.e. $\aA_r(\fF_{I \cup J})/\aA_r(\fF_{I \cap J}) \simeq  \aA_r(\fF_I \cap {}^\perp \fF_{I \cap J}) \oplus \aA_r(\fF_J \cap {}^\perp \fF_{I\cap J})$ and (Siii) holds.
\end{proof}

\section{Highest weight categories as abelian envelopes of thin categories}\label{sec_high_weigh_as_env}

We recall the definition of a highest weight category and assign two thin categories to it, the subcategory $\fF(\Delta_\Lambda)$ of standardly stratified objects and the subcategory $\fF(\nabla_\Lambda)$ of costandardly stratified objects. We prove that any highest weight category is the right abelian envelope of $\fF(\Delta_\Lambda)$ and the left abelian envelope of $\fF(\nabla_\Lambda)$. Conversely, we prove that the left and right abelian envelopes of any thin category are highest weight categories.
 
\vspace{0.3cm}
\subsection{Standarizable collections}~\\

Let $\Lambda$ be a finite poset. 
By generalizing the definition of standarizable collections in \cite{DR}, we say that a collection  $\{E(\la)\}_{\la \in \Lambda}$ of objects in a $k$-linear exact category $\eE$ is a \emph{weakly $\Lambda$-standarizable collection} if $\Ext^1(E(\la), E(\mu)) =0=\Hom(E(\la), E(\mu))$, for $\la \succ \mu$ or $\la$ and $\mu$ non-comparable. A weakly $\Lambda$-standarizable collection is \emph{$\Lambda$-standarizable} if in addition  $\Hom(E(\la), E(\la)) = k$, $\Ext^1(E(\la),E(\la))=0$, for any $\la \in \Lambda$. 

Let $\{E(\la)\}_{\la \in \Lambda}$ be a weakly $\Lambda$-standarizable collection in an exact category $\eE$. For a lower ideal $I \subset \Lambda$, denote by $\fF_I := \fF(\{ E(\la)\}_{\la \in I})$ the extension closure of $\{E(\la)\}_{\la \in I}$, similarly denote $\tT_I:= \fF(\{E(\la)\}_{\la \notin I})$. 
\begin{LEM}\label{lem_weak_stand_and_tors_pair}
	Let  $\{E(\la)\}_{\la \in \Lambda}$ be a weakly $\Lambda$-standarizable collection in an 
	exact category $\eE$. 
	Then, $\fF_\Lambda$ has  a perpendicular torsion pair $(\tT_{I}, \fF_{I})$, for any lower ideal $I\subset \Lambda$.
\end{LEM}
\begin{proof}
	Clearly $\fF_{I} \subset \tT_I^\perp$. 
	By Lemma \ref{lem_filtration_torsion}, $\fF_{\Lambda}$ admits a perpendicular torsion pair $(\tT_I, \fF_{I})$.
\end{proof}

\begin{DEF}
	Let $\Lambda$, $\Lambda'$ be posets  on a finite set.
	We say that $\Lambda$ \emph{dominates} $\Lambda'$ if
	\begin{equation*}
	\la\preceq_{\Lambda'}\mu \, \Rightarrow \la \preceq_{\Lambda} \mu. 
	\end{equation*}
\end{DEF}

\begin{PROP}\label{prop_F_stand_thin}
	Consider a $\Lambda$-standarizable collection  $\{E(\la)\}_{\la \in \Lambda}$  in a Hom and $\Ext^1$ finite $k$-linear exact category $\eE$. Then the subcategory $\fF_\Lambda\subset \eE$ is thin with irreducible objects $\{E(\la)\}$ and
	$I\mapsto \fF_I$ is a strict left admissible $\iI(\Lambda)$-filtration on $\fF_{\Lambda}$.
\end{PROP}
\begin{proof}
	By Lemma \ref{lem_weak_stand_and_tors_pair}, category $\fF_{\Lambda}$ admits perpendicular torsion pairs $ (\tT_I, \fF_{I})$, for $I\in \iI(\Lambda)$. A choice of a full order $\la_1 \prec \la_2 \prec \ldots\prec \la_n$ on $\Lambda$  compatible with the poset structure  gives lower ideals $I_i := \{\la_1,\ldots, \la_i\}$. Then 
	$0 = \tT_{I_n} \subset \tT_{I_{n-1}} \subset \ldots \subset \tT_0 = \fF_{\Lambda}$ is a right admissible filtration. Graded factors $(\fF_\Lambda)^o_i := \tT_{I_{i}}^\perp \cap \tT_{I_{i-1}}$ are equivalent to $\fF(E(\la_{i}))$. As $\Hom(E(\la_i),E(\la_i)) = k$ and $\Ext^1(E(\la_i), E(\la_i)) = 0$, the category $(\fF_\Lambda)^o_i$ is equivalent to $k \textrm{-vect}$, hence $\fF_\Lambda$ is thin.
	By Lemma \ref{lem_irr_in_thin}, $E(\la)$'s are the irreducible objects in $\fF_\Lambda$.
	
	The vanishing of $\Hom$ and $\Ext^1$ groups in a $\Lambda$-standarizable collection implies that the opposite of the canonical poset $\Lambda_\textrm{can}$ (\ref{def-poset}) of $\fF_{\Lambda}$ is dominated by the original poset on $\Lambda$, hence a lower ideal in $\Lambda$ is a lower ideal in $\Lambda_\textrm{can}^{\opp}$. It follows from Proposition \ref{prop_strict_filtr_on_thin} that $I \mapsto \fF_I$ is a strict left admissible $\iI(\Lambda)$-filtration on $\fF_\Lambda$.
\end{proof}

\vspace{0.3cm}
\subsection{From highest weight categories to thin categories and back}\label{ssec_from_hw_to_thin_and_back}~\\

Let $k$ be an algebraically closed field, $\aA$ a Deligne finite $k$-linear category (see Appendix \ref{ssec_ab_cat_fin_len}),
and $\Lambda$ a partial order on the set of  isomorphism classes of irreducible objects in $\aA$. Let $L(\la)$ denote the irreducible object corresponding to $\la \in \Lambda$, $P(\la)$ its projective cover, and $I(\la)$ its injective hull. Denote by $\Delta(\la)$ the \emph{standard object}, i.e. the maximal quotient of $P(\la)$ with simple factors isomorphic to $L(\mu)$, for $\mu \preceq \la$. Let $\fF(\Delta_\Lambda)\subset \aA$ (or simply $\fF(\Delta)$ when the partial order is fixed) be the extension closure of the $\Delta(\la)$'s.

The partial order $\Lambda$ is \emph{adapted} if for any $M\in \aA$ with top $L(\la_1)$ and socle $L(\la_2)$, with $\la_1$ and $\la_2$ incomparable, there exists $\mu \in \Lambda$ such that $\mu\geq \la_1$, $\mu \geq \la_2$ and $L(\mu)$ is a graded factor of a filtration of $M$ (see details in \cite{DR}).

We say that $(\aA, \Lambda)$ is a  \emph{highest weight category} (hw category) if, for any $\la \in \Lambda$,  
\begin{itemize}
	\item[$(st1)$]  $\End_{\aA}(\Delta(\la))\simeq k$ and
	\item[$(st2)$] $P(\la)\in \fF(\Delta_\Lambda)$.
\end{itemize} 

By \cite[Theorem 1]{DR}, if $\Lambda$ is adapted conditions ($st1$) and ($st2$) are equivalent to the following more traditional conditions for an hw category:
\begin{enumerate}
	\item[($st1^\prime$)]  there exists an epimorphism $\Delta(\la) \to L(\la)$ whose kernel admits a filtration with graded factors $L(\mu)$, where $\mu \prec \la$,
	\item[($st2^\prime$)] there exists an epimorphism $P(\la) \to \Delta(\la)$ whose kernel admits a filtration with graded factors $\Delta(\mu)$, where $\mu \succ\la$.
\end{enumerate}

\begin{LEM}\label{lem_delta_proj_in_subcat}
	Let $(\aA, \Lambda)$ be a Deligne finite category and a partial order on the set of isomorphism classes of simple objects in $\aA$. Assume there exist objects $\{\ol{\Delta}(\la)\}_{\la \in \Lambda}\subset \aA$ which satisfy ($st1^\prime$) and ($st2^\prime$). Then $\{\ol{\Delta}(\la)\}_{\la \in \Lambda}$ are the standard objects, hence $(\aA, \Lambda)$ is an hw category.
\end{LEM}
\begin{proof}

Let $i_*\colon \fF(\{L(\mu)\}_{\mu \preceq \la}) \to \aA$ be the inclusion of the extension closure of $\{L(\mu)\}_{\mu \preceq \la}$ and $i^*$ its left adjoint. Then the standard object $\Delta(\la) = i^*P(\la)$ is the projective cover of $L(\la)$ in $\fF(\{L(\mu)\}_{\mu \preceq \la})$, see Lemma \ref{lem_preserve_projective}. We check that $\ol{\Delta}(\la)$ is a projective object in $\fF(\{L(\mu)\}_{\mu \preceq \la})$ with the maximal semi-simple quotient $L(\la)$. These properties characterise the projective cover of $L(\la)$ up to an isomorphism, hence $\ol{\Delta}(\la) \simeq \Delta(\la)$.

By ($st1^\prime$), $\ol{\Delta}(\la) \in \fF(\{L(\mu)\}_{\mu \preceq \la})$. By ($st2^\prime$), $\ol{\Delta}(\la)$ is a quotient of $P(\la)$. Hence, for any $A\in \aA$, $\Hom(\ol{\Delta}(\la), A)$ is a subspace of $\Hom(P(\la), A)$. In particular, $\Hom(\ol{\Delta}(\la), L(\mu)) \simeq 0$, for $\mu \neq \la$, and $\Hom(\ol{\Delta}(\la), L(\la)) \simeq k$. It follows that $L(\la)$ is the maximal semi-simple quotient of $\ol{\Delta}(\la)$.

Similarly, $\Hom(\ol{\Delta}(\mu), L(\nu)) \simeq 0$, for $\nu \neq \mu$. Together with ($st2^\prime$) it implies that, for the kernel $K$ of the surjective map $P(\la) \to \ol{\Delta}(\la)$ and $\mu \leq \la$, $\Hom(K, L(\mu)) =0$. It follows that, for $\mu \preceq \la$, $\Ext^1(\ol{\Delta}(\la), L(\mu)) \simeq \Ext^1(P(\la), L(\mu))\simeq 0$ i.e. $\ol{\Delta}(\la)$ is projective in $\fF(\{L(\mu)\}_{\mu \preceq \la})$.
\end{proof}

\begin{LEM}\label{lem_preserve_projective}\cite[Lemma 3.2.7]{Pop}
	Let $F\colon \aA \to \bB$ be a functor of abelian categories with exact right adjoint $G$. Then $F$ maps projective $P\in \aA$ to a projective object in $\bB$.
\end{LEM}

If $P$ is a projective generator in an hw category $\aA$, then $\aA$ is equivalent to the category of finite dimensional right modules over the algebra $A:=\End (P)$.  Then the pair $(A,\Lambda)$, of a finite dimensional $k$-algebra and a partial order on its simple modules is a \emph{quasi-hereditary algebra} \cite{CPS}. There is a one-to-one correspondence between Morita equivalence classes of quasi-hereditary algebras and hw categories \cite[Theorem 3.6]{CPS}. A possible choice for $P$ is $P=\oplus_{\la \in \Lambda} P(\la )$, algebra $A=\End (\bigoplus_{\la \in \Lambda}P(\la))$ is then the unique up to isomorphism basic algebra in the Morita equivalence class.

By \cite[Theorem 4.3]{ParSco}, algebra $(A^{\opp},\Lambda)$ is also quasi-hereditary. In particular, $\aA^{\opp}\simeq (\textrm{mod-}A)^{\opp} \simeq \textrm{mod-}A^{\opp}$ is an hw category. By transporting the standard objects from $\aA^{\opp}$ to $\aA$, we get \emph{costandard objects} $\nabla(\la)$, labelled by $\la \in \Lambda$, defined by conditions: 
\begin{itemize}
	\item[($cost1$)] there exists an injective morphism $L(\la) \to \nabla(\la)$ whose cokernel admits a filtration with graded factors $L(\mu)$, where $\mu \prec \la$,
	\item[($cost2$)] 
	there exists an injective morphism $\nabla(\la) \to I(\la)$ whose cokernel admits a filtration with graded factors $\nabla(\mu)$, where  $\mu \succ\la$.	
\end{itemize} 

$\{\Delta(\la)\}_{\la \in \Lambda}$ is a $\Lambda$-standarizable collection in a hw category $(\aA, \Lambda)$ and while  $\{\nabla(\la)\}_{\la \in \Lambda}$ is $\Lambda^{\opp}$-standarizable (\textit{cf.} \cite{DR}).

For an hw category $(\aA, \Lambda)$, denote by $\fF(\nabla_\Lambda)\subset \aA$  the extension closure of the $\nabla(\la)$'s.
\begin{PROP}\label{prop_F_Delta_thin}
	Let $(\aA, \Lambda)$ be an hw category. Then $\fF(\Delta_\Lambda)$ and $\fF(\nabla_\Lambda)$ are thin categories with irreducible objects $\{\Delta(\la)\}_{\la\in \Lambda}$, respectively $\{\nabla(\la)\}_{\la \in \Lambda}$.
\end{PROP}
\begin{proof}
	Follows from Proposition \ref{prop_F_stand_thin}.
\end{proof}

\begin{PROP}\label{prop_highest_weuight_is_right_env}
	If $(\aA, \Lambda)$ is an hw category, then $\aA$ is the right abelian envelope of $\fF(\Delta_\Lambda)$ and the left abelian envelope of $\fF(\nabla_\Lambda)$.
\end{PROP}
\begin{proof}
	By ($st2$),  $P(\la) \in \fF(\Delta_\Lambda)$, for any $\la \in \Lambda$. Hence, any object of $\aA$ is a quotient of an object of $\fF(\Delta_\Lambda)$. By \cite[Theorem 3]{Rin}, category $\fF(\Delta_\Lambda)$ is closed under the kernels of epimorphisms. By Theorem \ref{thm_if_quotient_then_envelope}, $\aA$ is the right abelian envelope of $\fF(\Delta_\Lambda)$.
\end{proof}

\vspace{0.5cm}
\subsection{Envelopes of thin categories are highest weight categories}~\\

Consider a thin category $\eE$ with the canonical poset $\Lambda$ on the set $\{E(\la)\}$ of irreducible objects in it, as in (\ref{def-poset}). By Corollary \ref{cor_envel_for_thin}, the right and left  abelian envelopes for $\eE$, $\aA_r(\eE)$ and $\aA_l(\eE)$, exist. 

\begin{THM}\label{thm_right_env_is_h_w}
	Let $\eE$ be a thin category.
	Then $(\aA_r(\eE), \Lambda^{\opp})$ is an hw category with standard objects $\{\Delta(\la):=i_R(E(\la))\}$, and $i_R$ induces an equivalence $\fF(\Delta_{\Lambda^{\opp}})\simeq \eE$. Also, $(\aA_l(\eE), \Lambda)$ is an hw category with costandard objects $\{\nabla(\la):=i_L(E(\la))\}$, and $\fF(\nabla_{\Lambda})\simeq \eE$.	
\end{THM}
\begin{proof}
	We shall prove that $\aA_r(\eE)$ is hw with standard objects $\{\Delta(\la)\}$. 
	
	With Proposition \ref{prop_proj_in_thin} we have constructed indecomposable projective objects $P(\la)$ in $\eE$. By Corollary \ref{cor_envel_for_thin}, $\aA_r(\eE)\simeq \textrm{fp}(\pP) \simeq \textrm{mod-}A$, where $\pP = \textrm{add }\{P(\la)\}_{\la \in \Lambda}$ is a projectively generating subcategory and $A:=\End_{\eE}(\bigoplus_{\la \in \Lambda}P(\la))$. 
	
	Since the Yoneda embedding is fully faithful and $P(\la)$ and $P(\mu)$ are not isomorphic for $\la \neq \mu$ (see Proposition \ref{prop_proj_in_thin}), indecomposable projective $A$-modules $\{P(\la)\}_{\la \in \Lambda}$ are pairwise non-isomorphic. For $\la \in \Lambda$, let $L(\la)\in \aA_r(\eE)$ be the only irreducible $A$-module that allows a cover (i.e. a non-trivial morphism) $P(\la) \to L(\la)$. Then $\dim_k \Hom_{\aA_r(\eE)}(P(\la), L(\mu)) = \delta_{\la,\mu}$ and the number of irreducible factors of $\Delta(\la)$ isomorphic to $L(\mu)$ equals the dimension of $\Hom_{\aA_r(\eE)}(P(\mu), \Delta(\la))$. 

	We check that $\Delta(\la)$ verify ($st1^\prime$) and ($st2^\prime$).  By  Lemma \ref{lem_delta_proj_in_subcat} it follows that $\Delta(\la)$ is the standard object.
	
	Since functor $i_R$ is fully faithful and exact (see Theorem \ref{thm_right_env_if_proj_gen} and Lemma \ref{lem_P_has_weak_kernels}),  $\Hom_{\eE}(P(\mu), E(\la)) \simeq \Hom_{\aA_r(\eE)}(P(\mu), \Delta(\la))$ and Proposition \ref{prop_proj_in_thin} implies that the kernel of the surjective map $P(\la) \to \Delta(\la)$ admits a filtration with factors $\Delta(\mu)$, for $\mu\prec \la$. Hence, ($st2^\prime$) holds. 
	
	Now we calculate irreducible factors of $\Delta(\la)$. 
	If $\la \notin I_{\mu}$, then $E(\la) \in \fF_{I_\mu}$ and $\Hom(P(\mu), E(\la)) =0$, since $P(\mu)\in \tT_{I_{\mu}}$, by Proposition \ref{prop_proj_in_thin}. 
	Therefore, irreducible factor $L(\mu)$ may occur in $E(\lambda)$ only for $\la \preceq \mu$.
	
	Moreover, taking $\Hom_\eE (-,E(\lambda))$ out of the conflation $K(\la) \to P(\la) \to E(\la)$ with $K(\la) \in \tT_{I_{< \la}}$ (see Proposition \ref{prop_proj_in_thin}) implies that $\Hom_\eE(P(\la), E(\la)) \simeq \Hom(E(\la), E(\la)) = k$, i.e. $\Delta(\la)$ has one irreducible factor isomorphic to $L(\la)$.
	$\Hom(\Delta(\la), L(\mu)) \subset \Hom(P(\la), L(\mu))$ vanish for $\la \neq \mu$,
	hence $\Delta(\la)$ is the only possible irreducible quotient of $P(\la)$. Since the other irreducible factors are $L(\mu)$ with $\la\prec \mu$,
	($st1^\prime$) holds.
	
	By Theorem \ref{thm_right_env_if_proj_gen}, $i_R\colon \eE \to \aA_r(\eE)$ is an inclusion of fully exact subcategory. Since it maps irreducible objects in $\eE$ to standard objects in $\aA_r(\eE)$, $i_R$ induces an equivalence $\eE \simeq \fF(\Delta_{\Lambda^{\opp}})$.
\end{proof}

\begin{COR} (\textit{cf.} \cite[Lemma 7.1]{MenSan}, \cite{Krause})\label{cor_derived_equiv}
	Let $(\aA, \Lambda)$ be an hw category.  Then the embeddings  $\fF(\Delta_\Lambda) \to \aA$, $\fF(\nabla_\Lambda) \to \aA$ induce derived equivalences 
	\begin{equation}\label{eqtn_F_Delta_der_equiv}
	\dD^b(\fF(\Delta_\Lambda)) \simeq \dD^b(\aA) \simeq \dD^b(\fF(\nabla_\Lambda)).
	\end{equation} 
\end{COR}
\begin{proof} 
Category $\aA$ has finite homological dimension \cite[Theorem 4.3]{ParSco}. It is the right abelian envelope for thin category $\fF(\Delta_\Lambda)$ (see Propositions \ref{prop_F_Delta_thin} and \ref{prop_highest_weuight_is_right_env}) and the left abelian envelope for $\fF(\nabla_\Lambda)$.
By Lemma \ref{lem_thin_weakly_idem_split}, categories $\fF(\Delta_\Lambda)$ and $\fF(\nabla_\Lambda)$ are weakly idempotent split. Projectively generating subcategory $\pP=\fF(\{P(\la)\}_{\la \in \Lambda})$ in $\fF(\Delta_\Lambda)$ and injectively generating subcategory $\iI = \fF(\{I(\la)\}_{\la \in \Lambda})$ in $\fF(\nabla_\Lambda)$ are generated by finitely many objects with local endomorphisms rings (see Proposition \ref{prop_proj_in_thin} and Remark \ref{rem_injective_in_thin}). 
  By Lemma \ref{lem_idempotent_split}, both $\pP$ and $\iI$ are idempotent split. By Lemma \ref{lem_P_has_weak_kernels}, category $\pP$ has weak kernels while $\iI$ has weak cokernels. 
The statement follows from Theorem \ref{thm_bounded_derived_equi} and its dual.
\end{proof} 

\vspace{0.3cm}
\subsection{An equivalence relation on highest weight structures}\label{ssec_equiv_on_hw_str}~\\

An abelian category $\aA$ can be hw for different partial orders on isomorphism classes of simples. We introduce the following equivalence relation on the hw structures on $\aA$:
\begin{align}\label{eqtn_equiv_hw_st}
&(\aA, \Lambda) \sim (\aA, \Lambda') & & \textrm{ if } \fF(\Delta_\Lambda) \textrm{ and } \fF(\Delta_{\Lambda'}) \textrm{ are equal subcategories of } \aA. &
\end{align}
\begin{EXM} 
The category $\aA$ of modules over a directed algebra has at least two equivalence classes of hw structures. Indeed, let $\{S_i\}_{i=1}^n\subset \aA$ be simple objects and $\{P_i\}_{i=1}^n$ their projective covers such that $\Hom(P_i,P_j) =0$, for $i<j$.  Then for the order $\Lambda:S_1\preceq S_2 \preceq \ldots \preceq S_n$, $\fF(\Delta_\Lambda)\subset \aA$ is the subcategory of projective objects in $\aA$. For the order $\Lambda': S_n \preceq S_{n-1} \preceq \ldots \preceq S_1$,  $\fF(\Delta_{\Lambda'})$ is $\aA$ itself.
\end{EXM}
By Proposition \ref{prop_F_Delta_thin}, the set $\Lambda$ of isomorphism classes of simple objects in an hw category $(\aA, \Lambda)$ is in bijection with the set of isomorphism classes of irreducible objects in  thin categories $\fF(\Delta_\Lambda)$ and $\fF(\nabla_\Lambda)$.
This allows us to compare the poset $\Lambda$ of $\aA$ with the canonical poset $\Lambda_{\Delta}$ of $\fF(\Delta_\Lambda)$, as in (\ref{def-poset}). 

\begin{PROP}
	Let $(\aA, \Lambda)$ be an hw category and $\Lambda_{\Delta}$ the canonical poset on the irreducible objects of the thin category $\fF(\Delta_{\Lambda})$. Then $\Lambda_{\Delta}^{\opp}$ is adapted.
\end{PROP}
\begin{proof}
	For $\la \in \Lambda$, let $P(\lambda)$ be an indecomposable projective cover of $\Delta(\la)$, as constructed in Proposition  \ref{prop_proj_in_thin}. We write $\la <_c \mu$, respectively $\la \leq_c \mu$, when $\la$ is less than, respectively less than or equal to, $\mu$ in the canonical order $\Lambda_{\Delta}$.
	
	By \cite[Lemma 2.10]{Rodriguez}, an order $(\Lambda, >)$  is adapted if and only if for any morphism $\f \colon P(\lambda) \to P(\mu)$, with $\lambda$ and $\mu$ not comparable, the image of $\f$ is contained in $\sum \textrm{im }(\psi\colon P(\nu) \to P(\mu))$, where the sum is taken over all $\nu > \mu$ and all $\psi$ in $\Hom_{\aA}(P(\nu), P(\mu))$ (this is a rephrasing of condition 1 of Lemma 2.10 in \textit{loc.cit.}).
	
	We prove by induction on $|\Lambda|$ that any $\f \colon P(\la) \to P(\mu)$, with $\la$ and $\mu$ not comparable in $\Lambda_{\Delta}$, decomposes as $\sum \psi_i \tau_i$, for some $\tau_i \colon P(\la) \to P(\nu_i)$, $\psi_i \colon P(\nu_i) \to P(\mu)$ and $\nu_i <_c \mu$. The case $|\Lambda| =1$ is clear.
	
	For the inductive step, let $\la_m \in \Lambda_{\Delta}$ be a minimal element and let $\f$ be an element of $\Hom_{\aA}(P(\la), P(\mu))$ with $\la$ and $\mu$ not comparable in the canonical order. We consider three cases: $\la = \la_m$, $\mu = \la_m$, and $\la, \mu \in \Lambda \setminus \{\la_m\}$. In the first case, $P(\la) = \Delta(\la) = \Delta(\la_m)$ is the standard object, see Proposition \ref{prop_proj_in_thin}. Moreover, $\Hom_{\aA}(P(\la), P(\mu)) =0$, as $P(\mu)$ admits a filtration with $\Delta(\nu)$ for $\nu \leq_c \mu$, again by Proposition \ref{prop_proj_in_thin}, and $\Hom_{\aA}(\Delta(\la_m), \Delta(\nu)) =0$ for any $\nu \neq \la_m$, see \eqref{def-poset}, where $E_i$'s must be replaced by $\Delta(\la)$'s. Hence, $\f =0$ has the required decomposition.
	
	If $\mu = \la_m$ and $\mu$ and $\la$ are not comparable in $\Lambda_{\Delta}$ then again $\Hom_{\aA}(P(\la), P(\mu)) =0$. Indeed, analogously as above, if $\Hom(P(\la), P(\mu)) =\Hom_{\aA}(P(\la), \Delta(\la_m))\neq 0$ then $\Hom(\Delta(\nu), \Delta(\la_m)) \neq 0$, for some $\nu \leq_c \la$. Hence, $\mu=\la_m \leq_c \nu \leq_c \la$. This contradiction implies that $\f$ is again the zero morphism, hence it has the required decomposition.
	
	It remains to consider the case $\la, \mu \in \Lambda \setminus\{\la_m\}$. By Proposition \ref{prop_F_stand_thin}, the category $\fF(\{\Delta(\la)\}_{\la \in \Lambda \setminus \{\la_m\}})$ is thin with irreducible objects $\{\Delta(\la)\}_{\la \in \Lambda \setminus \{\la_m\}}$. Note that, formula \eqref{def-poset} implies that the canonical order on the irreducible objects in $\fF(\{\Delta(\la)\}_{\la \in \Lambda \setminus \{\la_m\}})$ is induced by the embedding $\Lambda\setminus \{\la_m\} \subset \Lambda$.
	

	The proof of Proposition \ref{prop_proj_in_thin} implies that $P(\la)$, respectively $P(\mu)$, is the universal extension \eqref{univ_extens} of $Q(\la)$, respectively $Q(\mu)$, by $\Delta(\la_m)$, for projective covers $Q(\la)$, $Q(\mu)$ of $\Delta(\la)$ and of $\Delta(\mu)$ in $\fF(\{\Delta(\la)\}_{\la \neq \la_m})$. In other words, $P(\la)$ and $P(\mu)$ fit into short exact sequences
	\begin{align*}
	 &0 \to \Delta(\la_m) \otimes V_{\la} \xrightarrow{i_{\la}} P(\la) \xrightarrow{d_{\la}} Q(\la) \to 0,&\\
	 &0 \to \Delta(\la_m) \otimes V_{\mu} \xrightarrow{i_{\mu}} P(\mu) \xrightarrow{d_{\mu}} Q(\mu) \to 0,&
	\end{align*}
	for vector spaces $V_{\la} = \Ext^1_{\aA}(Q(\la), \Delta(\la_m))^{\vee}$, $V_{\mu} = \Ext^1_{\aA}(Q(\mu), \Delta(\la_m))^{\vee}$.
	
	Since $\la_m \in \Lambda_\Delta$ is minimal, $\Delta(\la_m) = P(\la_m)$ and  $\Delta(\nu) \in \Delta(\la_m)^\perp$, for any $\nu \neq \la_m$, see \eqref{def-poset}. In particular, $\Hom_{\aA}(\Delta(\la_m), Q(\mu)) =0$, \textit{cf.} Proposition \ref{prop_proj_in_thin}. Applying functor $\Hom_{\aA}(-, Q(\mu))$ to the first short exact sequence implies that the composition with $d_{\la}$ is an isomorphism $\Hom_{\aA}(Q(\la), Q(\mu)) \xrightarrow{\simeq} \Hom_{\aA}(P(\la), Q(\mu))$. Hence, for $\f \in \Hom_{\aA}(P(\la), P(\mu))$ there exists a unique $\ol{\f}\in \Hom_{\aA}(Q(\la), Q(\mu))$ such that $\ol{\f} d_{\la} = d_{\mu} \f$. As $\la$ and $\mu$ are not comparable in the canonical order on $\Lambda \setminus \{\la_m\}$ (see above), the inductive assumption implies that $\ol{\f} = \sum\ol{\psi}_i \ol{\tau}_i$, for some $\ol{\psi}_i \in \Hom_{\aA}(Q(\nu_i), Q(\mu))$, $\ol{\tau_i} \in \Hom_{\aA}(Q(\la), Q(\nu_i))$ and $\nu_i<\mu$ in the canonical order on $\Lambda\setminus \{\la_m\}$. Since the orders are consistent, we also have $\nu_i <_c \mu$. 
	
	As  $P(\nu_i)$ are projective in $\aA$, $\ol{\psi}_i d_{\nu_i}$ admits a lift to $\psi_i \colon P(\nu_i) \to P(\mu)$, i.e. $\ol{\psi}_i d_{\nu_i} = d_{\mu}  \psi_i$. Analogously, we have $\tau_i \colon P(\la) \to P(\nu_i)$ such that $\ol{\tau}_i  d_{\la} = d_{\nu_i} \tau_i$. Then
	\begin{align*}
	d_{\mu} \sum \psi_i \tau_i = \sum \ol{\psi}_i d_{\nu_i} \tau_i = \sum \ol{\psi}_i \ol{\tau}_i d_{\la} = \ol{\f} d_{\la}.
	\end{align*}
	It follows that the difference $\f- \sum \psi_i \tau_i$ lies in the kernel $\Hom_{\aA}(P(\la), \Delta(\la_m) \otimes V_{\mu})$ of $d_\mu \circ(-) \colon \Hom_{\aA}(P(\la), P(\mu)) \to \Hom_{\aA}(P(\la), Q(\mu))$, i.e. it decomposes as $P(\la) \to \Delta(\la_m)\otimes V_{\mu} = P(\la_m) \otimes V_\mu\to P(\mu)$. If $\la_m <_c \mu$, this finishes the proof that $\f$ factors via $P(\nu_i) $ with $\nu_i <_c\mu$. If, on the other hand, $\la_m$ is not comparable to $\mu$, then $d_{\mu} \colon P(\mu) \xrightarrow{\simeq} Q(\mu)$ is an isomorphism, hence so is $d_{\mu} \circ(-)$. It follows that $\f = \sum \psi_i \tau_i$, as required.
	\end{proof}

\begin{PROP}\label{prop_can_poset_on_hw}
	Let $(\aA, \Lambda)$ be an hw category and
	$\Lambda_\Delta$, respectively $\Lambda_\nabla$, the canonical poset on the irreducible objects of the thin category $\fF(\Delta_\Lambda)$, respectively $\fF(\nabla_\Lambda)$. Then  $(\aA, \Lambda_\Delta^{\opp})$ is hw with the same  standard objects, i.e. $(\aA, \Lambda) \sim (\aA, \Lambda_\Delta^{\opp})$, and $(\aA, \Lambda_\nabla)$ is hw with the same costandard objects $\nabla(\la)$.
\end{PROP}
\begin{proof}
	By Proposition \ref{prop_highest_weuight_is_right_env}, $\aA \simeq \aA_r(\fF(\Delta_\Lambda))$. By Theorem \ref{thm_right_env_is_h_w}, $(\aA, \Lambda_\Delta^{\opp})$ is an hw category with standard objects $\{\Delta(\lambda)\}_{\la \in \Lambda}$. Hence, $(\aA, \Lambda)\sim(\aA, \Lambda_{\Delta}^{\opp})$.
\end{proof}

\begin{PROP}\label{prop_compare_posets}
	Let $(\aA, \Lambda)$ be an hw category and 
	$\Lambda_\Delta$ the canonical poset structure on the irreducible objects of the thin category $\fF(\Delta_\Lambda)$. Then $\Lambda_\Delta^{\opp}$is  dominated by $\Lambda$.
	
\end{PROP}
\begin{proof}
	As simple factors $L(\la)$ of $\Delta(\mu)$ satisfy $\la\preceq \mu$ and $L(\la)$ is the maximal semi-simple quotient of $\Delta(\la)$, 
	$\Hom(\Delta(\la), \Delta(\mu))\neq 0$ implies $\la \preceq \mu$. If $\Ext^1(\Delta(\la), \Delta(\mu)) \neq 0$, then $\Hom(K(\la), \Delta(\mu)) \neq 0$, for the kernel $K(\la)$ of the epimorphism $P(\la) \to \Delta(\la)$ in ($st2^\prime$). Hence, $\Hom(\Delta(\nu), \Delta(\mu)) \neq 0$, for some $\nu \succ \la$, i.e. $\mu \succeq \nu \succ \la$. In other words,  
	\begin{equation}\label{eqtn_poset_on_std}
	\Delta(\mu) \notin \Delta(\la)^{\perp} \, \Rightarrow\, \la \preceq \mu.
	\end{equation} 

	In view of definition (\ref{def-poset}), this implies 
that $\Lambda_\Delta^{\opp}$ is dominated by $\Lambda$.
\end{proof}

\begin{PROP}\label{prop_hw_iff_dominated}
	Let $(\aA, \Lambda)$ be an hw category and 
	$\Lambda_\Delta$ the canonical poset on the irreducible objects of the thin category $\fF(\Delta_\Lambda)$. For another partial order $\Lambda'$, $(\aA, \Lambda')$ is an hw structure equivalent to $(\aA, \Lambda)$ if and only if $\Lambda_\Delta^{\opp}$ is dominated by $\Lambda'$.
\end{PROP} 
\begin{proof}
	Since $\Lambda_\Delta$ depends only on the  thin category $\fF(\Delta_\Lambda)$, Proposition \ref{prop_compare_posets} implies that if $(\aA, \Lambda) \sim (\aA, \Lambda')$ then $\Lambda_\Delta^{\opp}$ is dominated by $\Lambda'$.
	
	By Proposition \ref{prop_can_poset_on_hw}, $\{\Delta(\la)\}_{\la \in \Lambda}$ are the standard objects for $(\aA, \Lambda_\Delta^{\opp})$. In particular, they satisfy ($st1^\prime$) and ($st2^\prime$). If $\Lambda_\Delta^{\opp}$ is dominated by $\Lambda'$, then  $\{\Delta(\la)\}_{\la \in \Lambda}$ satisfy ($st1^\prime$) and ($st2^\prime$) for $\Lambda'$. Hence, $(\aA, \Lambda')$ is hw with standard objects $\{\Delta(\la)\}_{\la \in \Lambda}$, by Lemma \ref{lem_delta_proj_in_subcat}, i.e.  $(\aA, \Lambda) \sim (\aA, \Lambda')$.
\end{proof}

\vspace{0.3cm}
\subsection{Characterization of a thin category in its right envelope}\label{ssec_char_of_thin_in_env}~\\

Let $\eE$ be a thin  category and $\Lambda$ its canonical poset. 
A lower ideal $I\subset \Lambda$ yields a perpendicular torsion pair $(\tT_I, \fF_I)$ on $\eE$ (see Proposition \ref{prop_strict_filtr_on_thin}). Category $\fF_I$ is thin, hence $\aA_r(\fF_I)$ exists and is a Deligne finite category (see Corollary \ref{cor_envel_for_thin}). It follows from Propositions \ref{prop_envelope_becomes_Serre_subcat} and \ref{prop_biloc_in_Df_n} that $\aA_r(\fF_I) \subset \aA_r(\eE)$ is a bi-localising subcategory.

The category $(\aA_r(\eE), \Lambda^{\opp})$ is hw (see Theorem \ref{thm_right_env_is_h_w}). In particular, the elements of $\Lambda$ are in bijection with isomorphism classes of simple objects in $\aA_r(\eE)$.

\begin{LEM}\label{lem_env_as_Serre_on_simple}
	The category $\aA_r(\fF_I)\subset \aA_r(\eE)$ is the full subcategory of objects whose simple factors are $L(\la)$, for $\la \in \Lambda \setminus I$.
\end{LEM} 
\begin{proof}
	Condition ($st1^\prime$) implies that, for any $\la \in \Lambda \setminus I$, $L(\la)$ is a quotient of an object in $\aA_r(\fF_I)$, hence $L(\la) \in \aA_r(\fF_I)$. As $\aA_r(\fF_I) \subset \aA_r(\eE)$ is a subcategory, for any $\la \in \Lambda \setminus I$,  $L(\la) \in \aA_r(\fF_I)$ is simple and, for distinct $\la, \mu \in \Lambda \setminus I$, objects $L(\la)$ and $L(\mu)$ are not isomorphic.
	
	On the other hand, $(\aA_r(\fF_I), (\Lambda \setminus I)^{\opp})$ is an hw category (see Theorem \ref{thm_right_env_is_h_w}). In particular, there are $|\Lambda \setminus I|$ non-isomorphic simple objects in $\aA_r(\fF_I)$. Hence, $\{L(\la)\}_{\la \in \Lambda \setminus I}$ is the set of isomorphism classes of simple objects in $\aA_r(\fF_I)$. 
\end{proof}

For $\la \in \Lambda$, the complement $\Lambda \setminus I_{\la}$ of the principal ideal $I_{\la} \subset \Lambda^{\opp}$ is a lower ideal in $\Lambda$. We denote by $(\tT_\la, \fF_\la)$ the corresponding perpendicular torsion pair (see Proposition \ref{prop_strict_filtr_on_thin}) and by ${\infl_{\la}}_*\colon \tT_\la \to \eE$, ${\defl_\la}_*\colon \fF_\la \to \eE$ the inclusion functors.

Denote ${\alpha_{\la}}_* = \aA_r({\defl_\la}_*) \colon \aA_r(\fF_\la) \to \aA_r(\eE)$, ${\beta_\la}_!=\aA_r({\infl_{\la}}_*)\colon \aA_r(\tT_\la) \to \aA_r(\eE)$.
By Proposition \ref{prop_envelope_becomes_Serre_subcat}, $\aA_r(\fF_\la) \subset \aA_r(\eE)$ is a Serre subcategory, $\aA_r(\eE)/\aA_r(\fF_\la) \simeq \aA_r(\tT_\la)$, and, by Lemma \ref{lem_enveloping_adjointness}, the right adjoint $\beta_\la^*$ to ${\beta_{\la}}_!$ is the quotient functor $\aA_r(\eE)\to \aA(\tT_\la)$. Note that, by Lemma \ref{lem_env_as_Serre_on_simple}, the set of isomorphism classes of simple objects in $\aA_r(\fF_\la)$ is $\{L(\mu)\}_{\mu \in I_{\la}}$. 

\begin{THM}\label{thm_thin_in_its_envelope}
	Consider a thin category $\eE$ with the canonical poset $\Lambda$. Then an object $A$ in $\aA:=\aA_r(\eE)$ is in the subcategory $\eE\subset \aA_r(\eE)$  if and only if, for any $\la \in \Lambda$, the adjunction counit $\varepsilon_{\la} \colon {\beta_{\la}}_! \beta_{\la}^* A\to A$ is a monomorphism.
\end{THM}
\begin{proof}
	Since $(\tT_\la, \fF_\la)$ is a perpendicular torsion pair in $\eE$, Proposition \ref{prop_E_in_A_r} implies that if $A\in \eE$, then ${\beta_\la}_! \beta_{\la}^*A\xrightarrow{\varepsilon_{\la}} A$ is a monomorphism.
	
	We prove the inverse by induction on the number $|\Lambda|$ of isomorphism classes of irreducible objects in $\eE$. The case $|\Lambda |=1$ is clear.
	
	Consider a maximal $\la_0 \in \Lambda$. Category $\aA':=\aA_r(\tT_{\la_0})$ is hw with poset $\Lambda':=\Lambda\setminus \{\la_0\}$. It follows from the definition (\ref{def-poset}) that $\Lambda'$ is the canonical poset of $\tT_{\la_0}$.
	
	Take $A\in \aA$ such that $\varepsilon_\la$ is mono on $A$, for all $\la \in \Lambda$. We check that $A\in \eE$. 
	Proposition \ref{prop_E_in_A_r} implies that $A\in \eE$ if and only if $\alpha_{\la_0}^*A \in \fF_{\la_0}$, $\beta_{\la_0}^*A \in \tT_{\la_0}$ and the adjunction counit ${\beta_{\la_0}}_! \beta_{\la_0}^*A \xrightarrow{\varepsilon_{\la_0}} A$ is mono. The first condition is obvious as $\fF_{\la_0} \simeq k\textrm{-vect} \simeq \aA_r(\fF_{\la_0})$ and the third condition holds by the assumption.
	Therefore, it suffices to check that $\beta_{\la_0}^*A \in \tT_{\la_0}$.
	
	By inductive hypothesis it suffices to check that, for any $\mu \in \Lambda'$, the adjunction counit ${\beta'_{\mu}}_! {\beta'_{\mu}}^* \beta_{\la}^*A \xrightarrow{\varepsilon'_{\mu}} \beta_{\la}^*A$ is mono. Here, for the principal ideal $I'_{\mu}\subset (\Lambda')^{\opp}$ generated by $\mu$, we denote by $\tT'_\mu:=\tT_{\Lambda'\setminus I'_\mu}$ the right admissible subcategory of $\tT_{\la_0}$, by  ${\beta'_{\mu}}_! \colon \aA_r(\tT'_\mu) \to \aA'$ the inclusion and by ${\beta'_{\mu}}^* $ its right adjoint.
	
	If $\la_0 \succ \mu$, then $\Lambda \setminus I_{\mu} = \Lambda' \setminus I'_{\mu}$, hence $\tT'_{\mu} \simeq \tT_{\mu}$. The composite $\aA_r(\tT'_\mu) \xrightarrow{{\beta'_{\mu}}_!} \aA' \xrightarrow{{\beta_{\la}}_!}\aA$ is ${\beta_{\mu}}_!$. The ${\beta_{\mu}}_!\dashv \beta_{\mu}^*$ adjunction counit $\varepsilon_{\mu}$ decomposes as ${\beta_{\la_0}}_!{\beta'_{\mu}}_!{\beta'_\mu}^*\beta_{\la_0}^*A \xrightarrow{{\beta_{\la_0}}_! \varepsilon'_{\mu}} {\beta_{\la_0}}_!\beta_{\la_0}^*A \xrightarrow{\varepsilon_{\la_0}}A$. Since $\varepsilon_{\mu}$ is a monomorphism when applied to $A$, so is ${\beta_{\la_0}}_!\varepsilon'_\mu$ when applied to $\beta_{\la_0}^*A$.
	
	If $\la_0$ and $\mu$ are not comparable, then $\tT'_\mu = \tT_{\Lambda \setminus(I_\mu \cup \{\la_0\})}$.  The inclusion ${\beta_{\la_0\mu}}_!$ of the subcategory $\aA_r(\tT'_\mu)\subset \aA$ decomposes as $\aA_r(\tT'_\mu) \xrightarrow{{\beta'_\mu}_!} \aA' \xrightarrow{{\beta_{\la_0}}_!}\aA$. The category $\fF_{\la_0\mu}:=\fF_{\Lambda\setminus(I_\mu \cup \{\la_0\})}$ is a direct sum of $\fF_{\la_0}$ and $\fF_\mu$. Hence, $\aA_r(\fF_{\la_0\mu})\simeq \aA_r(\fF_{\la_0})\oplus \aA_r(\fF_\mu)$ and, in the recollement
	$$
	\aA_r(\fF_{\la_0\mu}) \xrightarrow{{\alpha_{\la_0\mu}}_*} \aA \xrightarrow{{\beta_{\la_0\mu}}^*} \aA_r(\tT'_\mu)
	$$
	functor ${\alpha_{\la_0\mu}}_*$ is a direct sum of ${\alpha_{\la_0}}_*\colon \aA_r(\fF_{\la_0}) \to \aA$ and ${\alpha_\mu}_*\colon \aA_r(\fF_\mu) \to \aA$. It follows that the derived functor $L\alpha_{\la_0\mu}^*$ is a direct sum of $L\alpha_{\la_0}^*$ and $L\alpha_\mu^*$. In view of Lemma \ref{lem_kern_of_counit}, the injectivity of $\varepsilon_{\la_0}$ and $\varepsilon_{\mu}$ applied to $A$ implies $L^1\alpha_{\la_0}^*A = 0=L^1\alpha_\mu^*A$. Hence, $L^1\alpha_{\la_0\mu}^*A =0 $ and ${\beta_{\la_0\mu}}_!\beta_{\la_0\mu}^*A \to A$ is a monomorphism (see Lemma \ref{lem_kern_of_counit}). As above, the adjunction counit $\varepsilon_{\la_0\mu}$ decomposes as ${\beta_{\la_0}}_!{\beta'_\mu}_!{\beta'_\mu}^*\beta_{\la_0}^*A \xrightarrow{{\beta_{\la_0}}_!\varepsilon'_\mu} {\beta_{\la_0}}_!\beta_{\la_0}^*A \xrightarrow{\varepsilon_{\la_0}} A$. Hence, the injectivity of $\varepsilon_{\la_0\mu}$ applied to $A$ implies that ${\beta_{\la_0}}_!\varepsilon'_\mu$ is injective when applied to $\beta_{\la_0}^*A$.
	
	Since $\aA$ has enough projective objects, in either case the injectivity of ${\beta_{\la_0}}_! \varepsilon'_\mu$ applied to $\beta_{\la_0}^*A$ implies the injectivity of ${\beta'_\mu}_!{\beta'_\mu}^*\beta_{\la_0}^*A \xrightarrow{\varepsilon'_\mu} \beta_{\la_0}^*A$ (see Lemma \ref{lem_monomorphism_in_the_quotient} below). By inductive hypothesis, $\beta_{\la_0}^*A$ is an object in $\tT_{\la_0}$. It follows that $A\in \eE$, which finishes the proof.
\end{proof}

\begin{REM}\label{rem_high_weigh_non_alg_cl}
	Conditions ($st1$) and ($st2$) define highest weight categories over arbitrary  field $k$. If in the definition of the standarizable collection we require that $\Hom(E_i,E_i)$ is a division $k$-algebra, then all statements in Section \ref{sec_high_weigh_as_env} remain true for the more general definition of a thin category (see Remark \ref{rem_thin_not_alg_cl}). The proofs go \emph{mutatis mutandis}. 
\end{REM}

\section{Ringel duality}\label{sec_ring_dual}

We present the classical Ringel duality (see \cite{Rin}) of hw categories as a duality between the left and right abelian envelope. We define Ringel duality for thin categories.

\vspace{0.3cm}
\subsection{Ringel duality for highest weight categories}\label{ssec_ring_dual_for_hw}~\\

An hw category $(\aA, \Lambda)$ is the right abelian envelope of the thin category $\fF(\Delta_\Lambda)$ and the left abelian envelope of the thin category $\fF(\nabla_\Lambda)$ (see Proposition \ref{prop_highest_weuight_is_right_env}). Let $T = \bigoplus_{\la \in \Lambda}T(\la) \in \fF(\Delta_\Lambda)$  be the direct sum of the injective hulls of $\Delta(\la)$ 
(see Remark \ref{rem_injective_in_thin}). Object $T$ is inductively constructed by \emph{universal extensions} from the irreducible objects $\{\Delta(\la)\}$ in $\fF(\Delta_\Lambda)^{\opp}$, see Proposition \ref{prop_univ_ext_in_proj}. Hence, by \cite{DR}, $T\in \aA$ is the \emph{characteristic tilting object} introduced by M. Ringel in \cite{Rin}. The algebra $\End_{\aA}(T)$ is the \emph{Ringel dual} of $ \End_{\aA}(P)$, for a projective generator $P\in \aA$, \cite{Rin}. Both algebras $\End_{\aA}(P)$ and $\End_{\aA}(T)$ are quasi-hereditary, hence the duality is between quasi-hereditary algebras.

We define Ringel duality as a duality between hw categories.
\begin{DEF}\label{def_RD}
	Given a highest weight category $(\aA, \Lambda)$ its \emph{Ringel dual} is the hw category $(\mathbf{RD}(\aA), \Lambda^{\opp})$, where $\mathbf{RD}(\aA) = \aA_l(\fF(\Delta_\Lambda))$ is the left abelian envelope of the thin subcategory $\fF(\Delta_\Lambda)\subset \aA$. 
\end{DEF} 
Since $\aA$ and $\mathbf{RD}(\aA)$ are the right and left abelian envelopes of the same thin category $\fF(\Delta_\Lambda)$, we can think about Ringel duality as a transfer from the right to left envelope.

The category $\tT$ additively generated by $\{T(\la)\}$ is injectively generating in $ \fF(\Delta_\Lambda)$.

Derived equivalences (\ref{eqtn_F_Delta_der_equiv}) together with Proposition \ref{prop_full_exc_coll_in_D_thin} imply that $\dD^b(\aA)$ admits two full exceptional sequences $(\nabla(\la_n), \ldots, \nabla(\la_1))$ and $( \Delta(\la_1), \ldots, \Delta(\la_n))$, for any choice of a full order $\la_1 \le \ldots \le \la_n$ on $\Lambda$ compatible with the poset structure. In particular, $\dD^b(\aA)$ admits the \emph{Serre functor} $\mathbb{S}$ (see \cite[Proposition 3.8]{BK1}). 

Let $(E_1,\ldots,E_n)$ be a full exceptional sequence in a triangulated category $\dD$. The \emph{left dual} exceptional sequence $(F_n,\ldots,F_1)$ is characterised by the property
\begin{equation}\label{eqtn_dual_collec}
\Hom_{\dD}(E_i,F_j[l]) = \left\{\begin{array}{cl}k, & \textrm{if }l=0, i=j,\\0, & \textrm{otherwise.} \end{array} \right. 
\end{equation} 
Collection  $(E_1,\ldots,E_n)$ is said to be \emph{right dual} to $(F_n,\ldots,F_1)$ \cite{B}.

\begin{LEM}\cite[Lemma 2.1]{Krause2}\label{lem_two_exc_coll}
	Let $(\aA, \Lambda)$ be an hw category. Then, for any choice of a full order $\la_1 \leq \ldots \leq \la_n$ on $\Lambda$ compatible with the poset structure, the exceptional sequence $(\nabla(\lambda_n),\ldots, \nabla(\lambda_1) )$  is left dual to the collection $(\Delta(\lambda_1),\ldots, \Delta(\lambda_n) )$.
\end{LEM}

\begin{THM}\label{thm_Ringel_dual_for_hwc}
	Consider an hw category $(\aA, \Lambda)$.
	\begin{enumerate}
		\item The category $\bf{RD}(\aA)$ is equivalent to the right abelian envelope of $\fF(\nabla_\Lambda)$.
		\item Category $\bf{RD}(\aA)$ has costandard objects $\{\Delta(\la)\}$ and standard objects $\{\mathbb{S}^{-1}\nabla(\la)\}$.
	\end{enumerate}
\end{THM}
\begin{proof}
	By Lemma \ref{lem_fp_as_modules} and Corollary \ref{cor_envel_for_thin}, $\mathbf{RD}(\aA) \simeq (\textrm{mod-}\End_{\aA}(T)^{\opp})^{\opp} \simeq (\End_{\aA}(T)\textrm{-mod})^{\opp} \simeq \textrm{mod-}\End_{\aA}(T)$. By \cite{DR}, $T$ is a projective generator for $\fF(\nabla_\Lambda)$. It follows from Lemma \ref{lem_fp_as_modules} and Corollary \ref{cor_envel_for_thin} that $\aA_r(\fF(\nabla_\Lambda)) \simeq \textrm{mod-}\End_{\aA}(T) \simeq \mathbf{RD}(\aA)$.
	
	Let $\Lambda_{\Delta}$ be the canonical poset of $\fF(\Delta_\Lambda)$. By Theorem \ref{thm_right_env_is_h_w}, $(\mathbf{RD}(\aA), \Lambda_{\Delta})$ is hw with costandard objects $\{\Delta(\la)\}$. It follows from Proposition \ref{prop_hw_iff_dominated} that $(\mathbf{RD}(\aA), \Lambda^{\opp})$ is hw. 
	
	Derived equivalence (\ref{eqtn_F_Delta_der_equiv}) implies that $\dD^b(\mathbf{RD}(\aA)) \simeq \dD^b(\fF(\Delta_\Lambda)) \simeq \dD^b(\aA)$. Lemma \ref{lem_two_exc_coll} applied to category $\mathbf{RD}(\aA)$ implies that standard objects in $\mathbf{RD}(\aA)$ form the collection right dual to $\{\Delta(\la)\}_{\la \in \Lambda}$. Then, isomorphisms $\Hom(\mathbb{S}^{-1}(\nabla(\mu)), \Delta(\la)) \simeq \Hom(\Delta(\la), \nabla(\mu))$ imply that $\{\mathbb{S}^{-1}\nabla(\la)\}_{\la \in \Lambda}$ are the standard objects in $\mathbf{RD}(\aA)$. 
\end{proof}

It is instructive to encode relations between various relevant categories in the diagram:
\begin{center}\label{ladder}
	\includegraphics{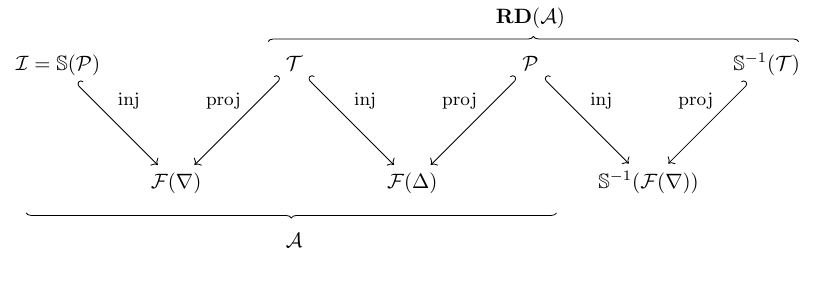}
\end{center}
All the categories in the picture are considered as full subcategories in $\dD^b(\aA)$. The categories in the upper row are additive and in the lower row are exact, while $\aA$ and $\mathbf{RD}(\aA)$ are abelian. Furthermore, category $\tT = \fF(\Delta_\Lambda) \cap \fF(\nabla_\Lambda)$ (\textit{cf.} \cite[Theorem 5]{Rin}) is a projectively generating subcategory in $\fF(\nabla_\Lambda)$ and an injectively generating subcategory in $\fF(\Delta_\Lambda)$.  Similar facts hold when shifting along the diagram.  

\vspace{0.3cm}
\subsection{Ringel duality for thin categories}\label{ssec_ring_dual_for_thin}~\\

The following fact, whose proof is basically due to Ringel, gives a characterization of a thin category $\eE$ inside its right abelian envelope.

\begin{PROP}\cite[Corollary 4]{Rin}\label{prop_F_Delta_as_vanishing_Ext_to_T}
	Let $I\in \eE$ be an injective generator for a thin category $\eE$. Then 
	$$
	\eE= \{X\in \aA_r(\eE)\,|\, \Ext^k(X,I) = 0, \textrm{ for }k\neq 0\}.
	$$
\end{PROP}
\begin{proof}
	Let $\Lambda$ be the canonical poset of $\eE$. Category $(\aA_r(\eE), \Lambda^{\opp})$ is hw with $\eE\simeq \fF(\Delta)$ (see Theorem \ref{thm_right_env_is_h_w}). Hence, $I\in \aA_r(\eE)$ is the characteristic tilting module for $(\aA_r(\eE), \Lambda^{\opp})$. The statement follows from the description $\fF(\Delta) = \{X \in \aA\,|\, \Ext^k(X, T) = 0,\textrm{ for }k\neq 0\}$, for a characteristic tilting module $T$ in a hw category $(\aA, \Lambda)$ (see \cite[Corollary 3]{Rin}).
\end{proof}	
\begin{THM}\label{thm_E_is_the_int_of_hearts}
	A  thin category $\eE$ is the intersection $\aA_r(\eE)\cap \aA_l(\eE)$ in $\dD^b(\eE)$.
\end{THM}
\begin{proof} 
	 By Corollary \ref{cor_envel_for_thin}, an injective generator $I\in \eE$ is an injective generator for $\aA_l(\eE)$. In view of the equivalences (\ref{eqtn_F_Delta_der_equiv}), the statement follows from Proposition \ref{prop_F_Delta_as_vanishing_Ext_to_T}.
\end{proof}

\begin{DEF}
	 The \emph{Ringel dual} $\bf{RDT}(\eE)$  of a thin category $\eE$ is the intersection in $\dD^b(\eE)$:
	$$
	\bf{RDT}(\eE)=\aA_l(\eE)\cap \mathbb{S}^{-1} \aA_r(\eE).
	$$
\end{DEF}

\begin{PROP}\label{prop_RIn_dual_is_thin}
	The Ringel dual of a thin category is a thin category. We have a canonical equivalence of exact categories: $\bf{RDT}(\bf{RDT}(\eE))\simeq  \eE$.
\end{PROP}
\begin{proof}
	Let $\Lambda$ be the canonical poset of $\eE$ and $\{E(\la)\}_{\la \in \Lambda}$ the set of isomorphism classes of irreducible objects in $\eE$.
	Category $\aA_l(\eE)$ is hw with costandard objects $\{E(\la)\}_{\la \in \Lambda}$ (see Theorem \ref{thm_right_env_is_h_w}). Let $\{G(\la)\}_{\la \in \Lambda}$ denote the exceptional sequence right dual to $\{E(\la)\}_{\la \in \Lambda}$. Then $\{G(\la)\}_{\la \in \Lambda}$ are standard objects in $\aA_l(\eE)$ (see Lemma \ref{lem_two_exc_coll}). Denote by $\fF_G\subset \aA_l(\eE)$ their extension closure. The category $\aA_l(\fF_G)$, Ringel dual to $\aA_l(\eE)$, is hw with costandard objects $\{G(\la)\}_{\la \in \Lambda}$ and standard objects $\{\mathbb{S}^{-1}E(\la)\}_{\la \in \Lambda}$ (see Theorem \ref{thm_Ringel_dual_for_hwc}). In particular, the category of objects with a standard filtration in $\aA_l(\fF_G)$ is $\mathbb{S}^{-1}(\eE)$.
	
	Derived equivalence (\ref{eqtn_F_Delta_der_equiv}) implies that projective generator $P \in \eE$ is a projective generator for the heart $\aA_r(\eE)$ of a \tr e on $\dD^b(\eE)$. Similarly, as $\{\mathbb{S}^{-1}E(\la)\}_{\la \in \Lambda}$ are standard objects in $\aA_l(\fF_G)$, $\mathbb{S}^{-1}(P)$ is a projective generator for the heart $\aA_l(\fF_{G})$ of a \tr e on $\dD^b(\fF_G) \simeq\dD^b(\aA_l(\eE)) \simeq \dD^b(\eE)$. It follows that $\aA_l(\fF_G) = \mathbb{S}^{-1}(\aA_r(\eE))$ as subcategories of $\dD^b(\eE)$. 
	As $\{G(\la)\}_{\la \in \Lambda}$ is the set of isomorphism classes of standard objects in $\aA_l(\eE)$, $\aA_l(\eE) \simeq \aA_r(\fF_G)$ (see Proposition \ref{prop_highest_weuight_is_right_env}).  Hence, $\bf{RDT}(\eE) =\aA_l(\eE) \cap \mathbb{S}^{-1}\aA_r(\eE) \simeq \aA_r(\fF_G) \cap \aA_l(\fF_G) \simeq \fF_G$ (see Theorem \ref{thm_E_is_the_int_of_hearts}). Thus, the thin category $\fF_G$ (see Proposition \ref{prop_F_stand_thin}) is the Ringel dual of $\eE$.
	
	As $\mathbb{S}^{-1}(\aA_r(\eE))$ is the left envelope of $\fF_G$, the double Ringel dual of $\eE$ is $\bf{RDT}(\bf{RDT}(\eE)) = \mathbb{S}^{-1}(\aA_r(\eE)) \cap \mathbb{S}^{-1}(\aA_l(\eE))$. It follows from Theorem \ref{thm_E_is_the_int_of_hearts} that the inverse of the Serre functor induces an equivalence of $\eE$ and its double Ringel dual.
\end{proof}

\begin{PROP}
	For a thin category $\eE$ we have canonical equivalences:
\begin{align}
	&\aA_r(\bf{RDT}(\eE)) \simeq\bf{RD}(\aA_r(\eE))\simeq \aA_l(\eE),& \label{eqtn_A_r_RD1}\\
	&\aA_l(\bf{RDT}(\eE)) \simeq\bf{RD}(\aA_l(\eE))\simeq \aA_r(\eE).&\label{eqtn_A_r_RD2}
	\end{align}
\end{PROP}
\begin{proof}
	It follows form the proof of Proposition \ref{prop_RIn_dual_is_thin} that, for a thin category $\eE$, $\bf{RDT}(\eE)$ is the thin category of objects with a standard filtration in $\aA_l(\eE)$. Hence, $\aA_r(\bf{RDT}(\eE)) \simeq \aA_l(\eE)$ by Proposition \ref{prop_highest_weuight_is_right_env}.  Since $\eE$ is the category of objects with a standard filtration in $\aA_r(\eE)$, (see Theorem \ref{thm_right_env_is_h_w}) category $\aA_l(\eE)$ is, by Definition \ref{def_RD}, the Ringel dual of $\aA_r(\eE)$. Equivalences (\ref{eqtn_A_r_RD1}) follow.  Applying the Ringel duality $\bf{RD}$  them gives equivalences (\ref{eqtn_A_r_RD2}).
\end{proof}
\appendix

\section{Abelian categories}\label{sec_abelian_cat}

In this appendix we gather facts about abelian categories that we use in the main body of the paper. We discuss localising and co-localising subcategories and abelian recollements. We introduce Deligne finite categories, and discuss their Serre subcategories.
\vspace{0.3cm}
\subsection{(Co)localising subcategories}\label{ssec_coloc_subcat}~\\

A full subcategory $\bB$ of an abelian category $ \aA$ is a \emph{Serre subcategory} if $\bB$ is closed under extensions, subobjects and quotients. It follows that $\bB$ is abelian, too.

Given a Serre subcategory $\bB\subset \aA$,
$$
S = \{f\in \textrm{Mor}(\aA) \,|\, \ker f, \textrm{coker } f \in \bB\}
$$
is a multiplicative system. The \emph{quotient category} $\aA/\bB$ is the localisation of $\aA$ in $S$. 

The category $\aA/\bB$ is abelian and the quotient functor $j^*\colon \aA\to \aA/\bB$ is exact. Moreover, $j^*$ is the universal exact functor $F\colon \aA \to \cC$ to an abelian category $\cC$ such that $F(B) \simeq 0$, for any $B\in \bB$.

A Serre subcategory $\bB\subset\aA$ is \emph{localising} (\emph{colocalising}) if  $j^*\colon \aA \to \aA/\bB$ admits a right (resp. left) adjoint. 
A Serre subcategory $\bB\subset \aA$ is \emph{bi-localising} if it is both localising and colocalising.

\begin{LEM}\cite[Proposition 4.4.3]{Pop}\label{lem_adj_ff}
	Let $\bB \subset \aA$ be a localising, respectively colocalising, subcategory. Then the right, respectively left, adjoint to the quotient functor $j^*\colon \aA \to \aA/\bB$ is fully faithful.
\end{LEM}
We will focus on colocalising subcategories. The dual statements are left to the reader. 
\begin{PROP}\cite[Theorem 4.4.9]{Pop}\label{prop_Popescu}
	Let $T \colon \aA\to \aA'$ be an exact functor between abelian categories with fully faithful left adjoint. Then $\ker T\subset \aA$ is a colocalising subcategory and $T$ induces an equivalence $\aA' \simeq \aA/\ker T$.
\end{PROP}

For a Serre subcategory $\bB \subset \aA$, the existence of left adjoint to the quotient functor implies the existence of left adjoint to the embedding $\bB \to \aA$:

\begin{LEM}\label{lem_adjoint_to_coloc}
	Let $i_*\colon \aA_1 \to \aA$ be an embedding of a colocalising subcategory, $j^* \colon \aA\to \aA/\aA_1$ the quotient functor and $j_!$ its left adjoint. Then $i_*$ admits the left adjont functor $i^* \colon \aA\to \aA_1$ and adjunction morphisms 
	fit into the exact sequence
	\begin{equation}\label{eqtn_def_of_i*i*} 
	j_!j^*A \xrightarrow{\varepsilon_A} A \xrightarrow{\eta_A} i_*i^*A \to 0.
	\end{equation}
\end{LEM}
\begin{proof}
	Fully faithfulness of $j_!$ (see Lemma \ref{lem_adj_ff}) implies that $j^*(\varepsilon_A)\colon j^*j_!j^* A\xrightarrow{\simeq} j^*A$ is an isomorphism. As $j^*$ is exact, it follows that $j^*( \textrm{coker }\varepsilon_A)=0$, i.e. $i_*i^*A\simeq \textrm{coker }\varepsilon_A$ is indeed an object of $\aA_1\simeq \ker j^*$. 
	
	We define the functor $i^*\colon \aA\to \aA_1$ as the cokernel of $\varepsilon\colon j_!j^* \to \Id_{\aA}$ and check that it is left adjoint to $i_*$. Applying $\Hom(-,i_*B)$ to (\ref{eqtn_def_of_i*i*}) gives an exact sequence:
	$$
	0 \to \Hom_{\aA}(i_*i^*A, i_*B) \to \Hom_{\aA}(A, i_*B) \to \Hom_{\aA}(j_!j^*A, i_*B) \simeq \Hom_{\aA_2}(j^*A, j^*i_*B) \simeq 0,
	$$
	hence the isomorphism $\Hom_{\aA_1}(i^*A, B) \simeq \Hom_\aA(i_*i^*A,i_*B) \simeq \Hom(A,i_*B)$. 
\end{proof}

By \cite{Gabriel}, the perpendicular category to a colocalising subcategory (closed objects) is equivalent to the quotient category. The following proposition, used in Section \ref{ssec_col_str_filt_on_env}, assures that a colocalising subcategory coincides with its double perpendicular.
\begin{PROP}\label{prop_colocalising}
	Let $i_*\colon \aA_1\to \aA$ be an embedding of a colocalising subcategory and $j^* \colon \aA\to \aA/\aA_1$ the quotient functor. Then 
	\begin{enumerate}
		\item $i_*\aA_1\subset \aA$ is a torsion-free part of a torsion pair,
		\item ${}^\perp(i_*\aA_1) \simeq j_!( \aA/\aA_1)$,
		\item $(j_!(\aA/\aA_1))^\perp \simeq i_*\aA_1$.
	\end{enumerate}
\end{PROP}
\begin{proof}
	By \cite[Proposition 1.2]{BelRei}, $i_*\aA_1\subset \aA$ is a torsion-free part of a torsion pair if and only if the embedding functor $i_*$ has left adjoint and, for any exact sequence 
	\begin{equation}\label{eqtn_left_ex_seq}
	0 \to i_*A_1 \to A \to i_*A'_1
	\end{equation}
	in $\aA$, object $A$ lies in the subcategory $i_*\aA_1$. By Lemma \ref{lem_adjoint_to_coloc}, functor $i^*$, left adjoint to $i_*$, exists. By applying the exact functor $j^*$ to (\ref{eqtn_left_ex_seq}), we get $j^*(A) \simeq 0$, i.e. $A\in i_*\aA_1$. 
	
	By \cite{Gabriel}, the image of $j_! \colon \aA/\aA_1 \to \aA$ is ${}^\perp(i_*\aA_1) $. In particular, $i_*\aA_1 \subset (j_!(\aA/\aA_1))^\perp$. For $A \in (j_!(\aA/\aA_1))^\perp$ the adjunction counit $j_!j^*A \to A$ is zero, hence the exact sequence (\ref{eqtn_def_of_i*i*}) implies that $A\simeq i_*i^*A$.
\end{proof}

\vspace{0.3cm}
\subsection{ Abelian recollements and bi-localising subcategories}\label{ssec_abel_recol_and_bi-loc_subcat}~\\

An \emph{abelian recollement} \cite{Kuhn} is a diagram of abelian categories and additive functors
\begin{equation}\label{eqtn_abelian_recol}
\xymatrix{\aA_1 \ar[r]|{i_*} & \aA \ar[r]|{j^*} \ar@<2ex>[l]|{i^!} \ar@<-2ex>[l]|{i^*} & \aA_2 \ar@<2ex>[l]|{j_*} \ar@<-2ex>[l]|{j_!}}
\end{equation}
such that $i^* \dashv i_* \dashv i^!$, $j_! \dashv j^* \dashv j_*$, functors $i_*$, $j_*$, $j_!$ are fully faithful, and $i_* \aA_1$ is the kernel of $j^*$. 

Given an abelian recollement (\ref{eqtn_abelian_recol}), the subcategory $i_*\aA_1 \subset \aA$ is bi-localising. The converse is also true, i.e. every bi-localising subcategory yields a recollement:

\begin{PROP}\label{prop_charact_of_abel_recol}
	The data of an abelian recollement (\ref{eqtn_abelian_recol}) is equivalent to the data of a bi-localising subcategory $i_*\aA_1$. 
\end{PROP}
\begin{proof}
	Consider abelian recollement (\ref{eqtn_abelian_recol}). Functor $j^*$ has a fully faithful right adjoint, hence it is the quotient functor  by its kernel  $i_*\aA_1$ (see Proposition \ref{prop_Popescu}). As $j^*$ has also the left adjoint, category $i_*\aA_1$ is both localising and colocalising.
	
	Let now $i_*\colon \bB\to \aA$ be an embedding of a bi-localising subcategory. The quotient functor $j^* \colon \aA \to \aA/\bB$ has right and left adjoints $j_*, j_!\colon \aA/\bB\to \aA$. By \cite[Proposition 4.4.3]{Pop}, functors $j_*$ and $j_!$ are fully faithful. By Lemma \ref{lem_adjoint_to_coloc}, functor $i_*$ has left adjoint $i^*$. 
	Similarly, define $i^!\colon \aA\to \aA/\bB$ as the kernel of the $j^*\dashv j_*$ adjunction unit. An analogous argument as in the proof of Lemma \ref{lem_adjoint_to_coloc} shows that $i_*\dashv i^!$.
\end{proof}

Adjoint functors for a 2-step filtration with bi-localising subcategories are compatible:

\begin{LEM}\label{lem_shriek_commute}
	Let $\bB \subset \cC$, $\cC\subset \aA$ be bi-localising subcategories.
	Consider the commutative diagram of inclusions and quotients:
	\[
	\xymatrix{&\aA/\cC \ar[r]^{\simeq} & \aA/\cC\\
		\bB \ar[r]^{i_*} & \aA \ar[r]^{j^*} \ar[u]^{l^*} & \aA/\bB\ar[u]^{s^*} \\
		\bB \ar[r]^{o_*} \ar[u]^\simeq & \cC \ar[r]^{p^*} \ar[u]^{k_*} & \cC/\bB \ar[u]^{r_*}} 	
	\]
	Then $\cC/\bB \subset  \aA/\bB$ is a bi-localising subcategory and functors $j_! \circ r_* \simeq k_*p_!$ are canonically isomorphic.
\end{LEM}
\begin{proof}
	A straightforward verification shows that $j^*l_! \dashv s^* \dashv j^*l_*$, i.e. $\cC/\bB\subset \aA/\bB$ is a bi-localising subcategory.
	
	As $l^*j_!r_* \simeq s^*j^*j_!r_* \simeq s^*r_* \simeq 0$, functor $j_!r_*$ takes values in the subcategory $\cC$. Hence, $j_!r_* \simeq k_*k^*j_!r_*$. The isomorphism $j^*k_* \simeq r_*p^*$ implies an isomorphism of left adjoint functors $k^*j_! \simeq p_!r^*$. Hence, $j_!r_* \simeq k_*k^*j_!r_* \simeq k_*p_!r^*r_*$. Since $r_*$ is fully faithful, the $r^*\dashv r_*$ adjunction counit yields an isomorphism $j_!r^* \simeq k_*p_!r^*r_*\simeq k_*p_!$.
\end{proof}

In Section \ref{sec_high_weigh_as_env} 
we use the following fact.
\begin{LEM}\label{lem_monomorphism_in_the_quotient}
	Consider a recollement (\ref{eqtn_abelian_recol}) of abelian categories. 
	Let $\aA_2$ have enough projectives. Then  morphism $\f$ in $\aA_2$ is a monomorphism if $j_!(\f)$ is.
\end{LEM}
\begin{proof}
	Functor $j_!$ is right exact. As $\aA_2$ has enough projectives, there exist the derived functor $Lj_! \colon \dD^-(\aA_2) \to \dD^-(\aA)$. As $j^*$ is exact, it can be derived to a $t$-exact functor $j^* \colon \dD^-(\aA) \to \dD^-(\aA_2)$, i.e. to a functor which preserves the subcategories of complexes concentrated in negative and positive degrees. Since $j^* \circ j_! \simeq \Id_{\aA_2}$, we have $j^* \circ Lj_! \simeq \Id_{\dD^-(\aA_2)}$. In particular, $j^*L^1j_!A = 0$, for any $A\in \aA_2$. It follows that $L^1j_!A \in i_* \aA_1$, for any $A\in \aA_2$.
	
	Let $\f \colon A'_2 \to A_2$ be a morphism in $\aA_2$. Denote by $K$ its kernel and by $I$ its image. If $j_!(\f)$ is a monomorphism, then the map $\alpha$ in the exact sequence
	$$
	L^1j_!I \to j_!K \xrightarrow{\alpha} j_!A'_2 \to j_!I \to 0
	$$
	is zero. It follows that $j_!K$ is a quotient of an object in $i_* \aA_1$. As $i_*\aA_1$ is a kernel of an exact functor $j^*$, it is closed under quotients. It follows that $j_!K \in i_*(\aA_1)$. 
	Then $K=j^*j_!K=0$, i.e. $\f$ is a monomorphism.
\end{proof}

\vspace{0.3cm}
\subsection{Serre subactegories in abelian categories of finite length}\label{ssec_ab_cat_fin_len}~\\

We say that a $k$-linear abelian category is \emph{of finite length} if it is $\Hom$ and $\Ext^1$-finite, contains finitely many non-isomorphic simple objects, and any object admits a finite filtration with simple graded factors. 

Serre subcategory of a finite length category is determined by simple objects in it.

\begin{LEM}\label{lem_biloc_in_fin_len}
	Let $\aA$ be an abelian category of finite length and $\mathcal{S}$ the set of simple objects in it. Then Serre subcategories in $\aA$ are in one-to-one correspondence with subsets of $\sS$. Given $I\subset \mathcal{S}$, the corresponding Serre subcategory is the extension closure of $\{S\}_{S\in I}$.	
\end{LEM}
\begin{proof}
	Given $I\subset \mathcal{S}$, let $\fF(S_i)_{i\in I}$ be  the full subcategory of  objects whose all simple factors are isomorphic to $S_i$, for some $i\in I$. It is standard that $\fF(S_i)_{i\in I}$ is a Serre subcategory.	
	Moreover, any Serre subcategory is of this form. Indeed, given Serre subcategory $\bB \subset \aA$, let $I_{\bB}\subset \sS$ be the set of simples in $\aA$ which lie in $\bB$. Then $\fF(S)_{i\in I_{\bB}}\subset \bB$. On the other hand, $\bB$ is closed under subquotients, hence together with an object $B\in \bB$ it contains all its simple factors in $\aA$, i.e. $B\in \fF(S)_{i\in I_{\bB}}$.
\end{proof}

An abelian category of finite length is a \emph{Deligne finite category} if it has a projective generator.
By \cite[Propostion 2.14]{Del}, a Deligne finite category $\aA$ with projective generator $P$ is equivalent to the category of finitely generated modules over the (finite dimensional) $k$-algebra $\End_{\aA}(P)$. In particular, Deligne finite categories are Krull-Schmidt.

\begin{PROP}\label{prop_biloc_in_Df_n}
	Consider a Deligne finite category $\aA$. Then
	\begin{enumerate}
		\item any Serre subcategory $\bB$ of $\aA$ is bi-localising. 
		\item Any Serre subcategory $\bB$ and any quotient category $\aA/\bB$ is Deligne finite.
	\end{enumerate}
\end{PROP}
\begin{proof}
	By Lemma \ref{lem_biloc_in_fin_len}, a Serre subcategory $\bB$ of a Deligne finite category $\aA$ is the extension closure $\fF({S_i})$ of the simple objects in it.

	Note that $\aA$ is equivalent to the category of right modules over a finite dimensional algebra $R$. A choice of idempotents $e_i$ such that $e_iR$ is a projective cover of $S_i$ yields an equivalence of the Serre category generated by simple objects $\{S_j\}_{j\in J}$ with $\textrm{mod-}R/RfR$, where $f = \sum_{i\notin J}e_i$. In particular, $\bB$ is Deligne finite.
	
	 Then the quotient category is $\textrm{mod-}fRf$, in particular it is Deligne finite. The quotient  functor $\Hom_R(fR,-)\simeq (-)\otimes_R Rf\colon \aA \to \textrm{mod-}fRf$ has the left $(-)\otimes_{fRf}fR$ and the right $\Hom_{fRf}(Rf,-)$ adjoint.
\end{proof}

\bibliographystyle{alpha}
\bibliography{../../ref}
\end{document}